\documentclass[10pt]{amsart}

\usepackage[utf8]{inputenc}
\usepackage[T1]{fontenc}
\usepackage[english]{babel}
\usepackage{graphicx}
\usepackage{amsmath}
\usepackage{amssymb}
\usepackage{amsthm}
\usepackage{fullpage}
\usepackage{mathrsfs}
\usepackage{mathtools}
\usepackage{tikz-cd}
\usepackage[hidelinks]{hyperref}
\usepackage{amsaddr}

\hypersetup{allcolors=black}

\newtheorem{theorem}{Theorem}
\numberwithin{theorem}{section}
\newtheorem{lemma}[theorem]{Lemma}

\newtheorem{corollary}[theorem]{Corollary}
\newtheorem{remark}{Remark}
\numberwithin{remark}{section}
\newtheorem{example}{Example}
\numberwithin{example}{section}

\newcommand\C{\mathbb C}
\newcommand\N{\mathbb N}

\newcommand\Z{\mathbb Z}
\newcommand\Aa{\mathcal A}
\newcommand\Bb{\mathcal B}

\newcommand\Ff{\mathcal F}
\newcommand\Gg{\mathcal G}

\newcommand\Ii{\mathcal I}
\newcommand\Jj{\mathcal J}
\newcommand\Mm{\mathcal M}
\newcommand\Oo{\mathcal O}
\newcommand\Pp{\mathcal P}
\newcommand\Qq{\mathcal Q}
\newcommand\Ss{\mathcal S}
\newcommand\id{\text{id}}

\newcommand\Supp{\text{Supp}}
\newcommand\Forests{\text{Forests}}
\newcommand\Graphs{\text{Graphs}}

\newcommand\RTrees{\text{RootedTrees}}
\newcommand\Set{\text{Set}}
\newcommand\Trees{\text{Trees}}
\newcommand\inv{^{-1}}

\newcommand\qi{\simeq_q}
\DeclareMathOperator\Adj{Adj}
\DeclareMathOperator\Aut{Aut}
\DeclareMathOperator\aut{aut}
\DeclareMathOperator\co{co}
\DeclareMathOperator\Diag{Diag}

\DeclareMathOperator\Epi{Epi}
\DeclareMathOperator\epi{epi}

\DeclareMathOperator\Hom{Hom}

\DeclareMathOperator\Jor{Jor}
\DeclareMathOperator\Mon{Mon}
\DeclareMathOperator\mon{mon}

\DeclareMathOperator\Qu{Qu}
\DeclareMathOperator\Quo{Quo}
\DeclareMathOperator\quo{quo}

\DeclarePairedDelimiter{\abs}{\lvert}{\rvert}
\DeclarePairedDelimiter{\norm}{\lVert}{\rVert}

\makeatletter
\let\oldabs\abs
\def\abs{\@ifstar{\oldabs}{\oldabs*}}
\let\oldnorm\norm
\def\norm{\@ifstar{\oldnorm}{\oldnorm*}}
\let\oldbracket\bracket
\def\bracket{\@ifstar{\oldbracket}{\oldbracket*}}
\let\oldpare\pare
\def\pare{\@ifstar{\oldpare}{\oldpare*}}
\makeatother
\allowdisplaybreaks

\title{Quantum properties of $\Ff$-cographs}
\author{Paul Meunier\textsuperscript{*}}\thanks{\textsuperscript{*}Research primarily supported by the ENS Lyon, France, and partially supported by Methusalem grant METH/21/03 of the Flemish Government and by PhD grant 11PAL24N funded by the Research Foundation Flanders (FWO)}
\address{Mathematics department, KU Leuven}

\date\today

\begin{document}

\begin{abstract}
	We initiate a systematic study of quantum properties of finite graphs, namely, quantum asymmetry, quantum symmetry, and quantum isomorphism. 
	We define the Schmidt alternative for a class of graphs, which reveals to be a useful tool for studying quantum symmetries of graphs. 
	After showing that quantum isomorphic graphs have quantum isomorphic centers and connected components, we solve the aforementioned problems for the classes of cographs and forests. 
	We also compute their quantum automorphism groups for the first time. 
	In doing so, we extend to the noncommutative setting a theorem of Jordan. 
	Using general results on $\Ff$-cographs, we extend the precedent results to $\Gg_5$-cographs and tree-cographs, two distinct strictly proper superclasses of cographs and forests respectively. 
	Finally, we show that quantum isomorphic planar graphs are isomorphic.
\end{abstract}

\maketitle

\addtocontents{toc}{\protect\setcounter{tocdepth}{1}}
\tableofcontents

\section{Introduction}

Noncommutative combinatorics is the piece of land in the noncommutative world hosting combinatorics. 
First manifestations of noncommutative combinatorics include using combinatorial techniques for operator algebras: one can think of free probability theory for von Neumann algebras, or graph $C^*$-algebras and the Eliott program, or more recently the study of easy quantum groups through combinatorial categories. 
While until recently this aspect was more developped, already at the birth of noncommutative geometry Alain Connes noticed that combinatorial objects had noncommutative structure (see the noncommutative insight on Penrose tilings in~\cite{ConnesBook}, Chapter 2, Section 3). 
In a recent work~\cite{ManRob}, Laura Man\v{c}inska and David Roberson shed light on purely graph theoretical notions going back to a work of L\'aszl\'o Lovász in 1967~\cite{LovOpStruct} by using the theory of quantum groups. 
While interactions go both ways, noncommutative combinatorics can be seen not only as two theories talking to each other but also as the unveiling of noncommutative structures hidden in combinatorial objects or the exploration of the combinatorial organisation of noncommutative spaces: and this is only the beginning.

The present work settles in the field of quantum graph theory. 
We emphasize that by this we mean the study of quantum properties of graphs, and not the study of the operator systems known as ``noncommutative graphs'' or sometimes ``quantum graphs.'' 
We are confident in the use of the adjective ``quantum'' since the properties of graphs studied are really the quantum analogue of the classical ones, in the physical sense initially described by the Bell inequalities~\cite{Bell}. 
The study of quantum symmetries of combinatorial objects can be traced back to the following question of Connes, who, according to Shuzhou Wang~\cite{Wang}, asked at the Les Houches Summer School of 1995:
\begin{center}
	\textit{what is the quantum automorphism group of a space?}
\end{center}
Wang initiated an answer to this general question in 1998~\cite{Wang} by introducing the \textit{quantum permutation group} $S_n^+$ and showing it is suited for being the quantum automorphism group of a set of $n$ points (here $n \in \N$). 
Following this definition, Julien Bichon introduced in 2003~\cite{Bichon2003} a first definition for the \textit{quantum automorphism group of a graph}. 
It is worth noting that his definition is slightly different from the modern one we use here (coming from the work of Teodor Banica~\cite{banica2005quantum}). 
Given a (finite simple) graph $G$, one obtains a compact matrix quantum pseudogroup $\Qu(G)$, which is a quantum subgroup of $S_n^+$, where $n\geq 1$ is the number of vertices of~$G$. 

During two decades, studies of quantum automorphism groups of finite graphs have been on the sideroad of the study of quantum groups and not much is known about them. 
The main line of results has been to look at specific graphs, often very regular (mostly Cayley graphs) and compute their quantum automorphism group. 
The field gained recent attraction thanks to results of Lupini, Man\v{c}inska, and Roberson, who in~\cite{LupManRob} incorporated notions from algebraic graph theory and from quantum information theory to unify some notions from across the literature, in particular, the notion of \textit{quantum isomorphism}. 
Following this, the breakthrough work of Man\v{c}inska and Roberson~\cite{ManRob} answers deep and purely combinatorial questions using heavily the quantum group machinery, and provides an operator-algebra free description of quantum isomorphism: they obtain that two graphs $G$ and $H$ are quantum isomorphic if and only if they have the same number of morphisms from every planar graph. 

In the present work, we aim at initiating a systematic study of quantum properties of graphs. 
This goes along with the following ideas:
\begin{enumerate}
	\item rather than confining the study to very regular graphs, we want to consider arbitrary graphs, and use graph-theoretical results to decompose them (see Theorem~\ref{thm:cographs} or Theorem~\ref{thm:decomposition_F_cographs} for instance, as well as the inductive decomposition of rooted trees presented in Section~\ref{sec:forests}),
	\item rather than directly trying to compute the quantum automorphism group of a graph, we are investigating more general properties, such as quantum asymmetry, quantum symmetry, and quantum isomorphism,
	\item finally, following an approach common in structural graph theory, rather than focusing on particular graphs, we shift the focus to study the properties of graph classes. 
\end{enumerate}

Regarding the second point, these properties are directly related to the underlying $C^*$-algebras of the quantum groups, about which still very little is known. 
Not only does understanding the $C^*$-algebra seems a necessary first step into understanding the quantum group structure of the quantum automorphism group of a graph, it already allows us to compute the quantum automorphism groups in some cases. 

We systematise the study of quantum asymmetry, quantum symmetry, and quantum isomorphism for graph classes by introducing the following properties for a class of graphs $\Ff$ (the technical terms will be defined later, see Section~\ref{sec:preli}):
\begin{itemize}
	\item[(QA)]\label{ax:quantum_asymmetry} (quantum asymmetry) for $G\in \Ff$, $\Aut(G)$ is trivial if and only if $\Qu(G)$ is trivial,
	\item[(QI)]\label{ax:quantum_iso} (quantum isomorphism) for $G, H\in \Ff$, we have that $G$ is quantum isomorphic to $H$ if and only if $G$ is isomorphic to $H$,
	\item[(SA)]\label{ax:schmidt_alternative} (Schmidt alternative\footnote{The name is inspired by the Tits alternative in group theory.}) for $G\in \Ff$, the graph $G$ has quantum symmetry if and only if it satisfies the Schmidt criterion.
\end{itemize}

We introduce the Schmidt alternative as it appears to be a useful tool to solve the quantum symmetry problem for graph classes. 
We call a class satisfying all three axioms (QA), (QI), and (SA) a \textit{tractable} class. 
We prove that the classes of cographs and forests are tractable, as well as the class of $\Ff$-cographs for every tractable class $\Ff$. 
This allows us to obtain the quantum automorphism groups of the graphs in the class.\\

In Section~\ref{sec:preli}, we recall elementary notions of graph theory and quantum group theory, and gather general results that will be useful in the rest of the article. 
In particular, we give a short new proof that connectedness is preserved under quantum isomorphism, and we generalise the criterion of Schmidt for quantum symmetry. 
In Section~\ref{sec:qi_sums}, we prove structural results about quantum isomorphism. 
We show that quantum isomorphic graphs have quantum isomorphic centers (as a consequence of a more general result, see Theorem~\ref{thm:qi_center}), and show that quantum isomorphic graphs have quantum isomorphic connected components. 
In Section~\ref{sec:cographs}, we compute our first example by studying the class of cographs. 
In Section~\ref{sec:tractable}, we obtain general results regarding the axioms (QA), (SA), (QI), and a strengthening of the latter called superrigitiy, which allow us to study generally the class of $\Ff$-cographs when $\Ff$ is a class of graphs with certain stability properties. 
This also allows us to compute the quantum automorphism groups of $\Ff$-cographs as a function of the quantum automorphism groups of the graphs in $\Ff$. 
As an example, we show that the class of $\Gg_5$-cographs is tractable and superrigid and we compute the quantum automorphism groups of its graphs. 
In Section~\ref{sec:forests}, we study the quantum properties of forests. 
We show that they form a tractable and superrigid class of graphs and obtain for the first time their quantum automorphism groups; they are exactly the quantum groups obtained from the trivial one by taking free products and wreath products with a quantum permutation group. 
This generalises to the noncommutative setting a theorem of Jordan~\cite{jordan1869assemblages}. 
These results immediately extend to tree-cographs thanks to the results of Section~\ref{sec:tractable}. 
In Appendix~\ref{app:lovasz}, we give a modern exposition of results of Lovász~\cite{LovOpStruct} and adapt them to the notion of $\Ff$-isomorphism introduced by Mančinska and Roberson in~\cite{ManRob}. 
This allows us to prove that quantum isomorphic planar graphs are isomorphic. 

We emphasize that the computational results shall not be seen as the aim of the article but rather as the illustration of our techniques and general approach. 
Indeed, the door is now open to follow a similar approach for many classes of graphs, in order to obtain a more systematic understanding of quantum properties of graphs.

Most of the results of the present article (including the results for cographs, forests, and the computation of their quantum automorphism groups) were obtained by the author as part of his studies for the diploma of ENS Lyon, France, in the summer of 2022 (see a related presentation~\cite{MeIHES}). 
While we were completing the writing of this article, we learned that the results concerning the quantum automorphism groups of trees were recently obtained independently by De Bruyn, Kar, Roberson, Schmidt, and Zeman in the preprint~\cite{de2023quantum}. 

\section{Preliminaries}\label{sec:preli}

We recall the theoretical background needed for the article, and introduce the terminology and the notations that will be used in the rest of the article. 
In subsection~\ref{subsec:graph_intro}, we recall basics of graph theory. 
In subsection~\ref{subsec:qu_G}, we recall the definition of the quantum automorphism group of a graph. 
In subsection~\ref{subsec:qu_symm_schmidt}, we introduce the Schmidt alternative, and generalise a result by Schmidt (see Theorem~\ref{thm:qu_schmidt}). 
In subsection~\ref{subsec:qu_iso}, we recall the notion of quantum isomorphism, and prove folklore results. 
We point out that while these results and their proofs seem well-known, there does not always seem to be an explicit proof in the literature, so this might be the first time they are written down (we think of Lemma~\ref{lem:transitivity_qi} in particular). 
Finally, in subsection~\ref{subsec:metric_qu}, we gather useful rigidity results about magic unitaries, and prove a new result (Lemma~\ref{lem:hall_mu}).

\subsection{Succinct reminder of graph theory}\label{subsec:graph_intro}

For the standard definitions and notations of graph theory that we do not introduce here, we refer the reader to~\cite{BondyMurty}. 

A \textit{graph} is a pair $G = (V(G),E(G))$ where $V(G)$ is a finite set and $E(G) \subset \Pp(V(G))$ is such that for any $e \in E(G)$ we have $\# e = 2$. 
The elements of $V(G)$ are called the \textit{vertices} of $G$ and the elements of $E(G)$ are called the \textit{edges} of $G$. 
This definition is sometimes referred to in the literature as the one of \textit{finite simple graphs}, but since all of our graphs will be finite and simple, we will simply say graph. 

A \textit{graph morphism} $f \colon G \to H$ is a function $f \colon V(G) \to V(H)$ such that if $xy \in E(G)$ then $f(x)f(y) \in E(H)$. 
It is easy to check that graphs and graph morphisms satisfy the axioms of a category (with usual composition and identities). 
Two graphs $G$ and $H$ are \textit{isomorphic} if they are in this category, that is, if there exist two graph morphisms $f \colon G\to H$ and $g \colon H\to G$ such that $fg = 1_H$ and $gf = 1_G$.

\begin{remark}\label{rk:iso_graphs}
	Since we work up to isomorphism, it is common practice to identify isomorphic graphs. 
	In particular, every graph is isomorphic to a graph whose vertex set is of the form $\{1,\ldots,n\}$, 
	So by a graph we will often mean the isomorphism class of graphs with vertex set $\{1,\ldots,n\}$. 
	This also allows us to talk about the set of all graphs.
\end{remark}

Let $G$ be a graph. 
If $\{x,y\} \in E(G)$, we say that $x$ and $y$ are \textit{neighbors} and denote it by $x\sim_G y$, or simply $x\sim y$ when the graph is clear from the context. 
For $x\in V(G)$, the set of all its neighbors is called the \textit{neighborhood} of $x$ and is denoted by $N_G(x)$, or simply $N(x)$. 
The \textit{degree} of $x$ is the number of its neighbors, we denote it by $deg_G(x)$ or $deg(x)$. 

If $H$ is a graph such that $V(H) \subset V(G)$ and $E(H) \subset E(G)$, then $H$ is called a \textit{subgraph} of $G$. 
If moreover we have that $E(H) = E(G) \cap \Pp(V(H))$, then $H$ is called an \textit{induced subgraph} of $G$. 
Given $X\subset V(G)$, there is exactly one induced subgraph of $G$ with vertex set $X$, we denote it by $G[X]$. 

A \textit{class} of graphs is a subset of the set of all graphs (see Remark~\ref{rk:iso_graphs}). 
It is said to be \textit{hereditary} if it is closed under taking induced subgraphs. 

A set $S\subset V(G)$ is called a \textit{stable} set if the induced subgraph $G[S]$ has no edges. 
A graph $G$ is a \textit{bipartite graph} if its vertex set can be partitionned into two stable sets.

A \textit{matching} in a graph $G$ is a subset $R \subset E(G)$ such that its elements are two-by-two disjoint. 
Given $X \subset V(G)$, a matching $R$ is said to be \textit{saturating $X$ } if $X \subset \bigcup R$. 

\begin{theorem}[Hall's marriage theorem]\label{thm:hall}
	Let $G = (A\sqcup B, E)$ be a bipartite graph. 
	Then $G$ admits a matching saturating $A$ if and only if it satisfies the following condition:
	\begin{center}\label{hall_condition}
		\begin{itemize}
			\item[(H)] for every subset $S \subset A$, we have $\abs{S} \leq \abs{N(S)}$
		\end{itemize}
	\end{center}
	where $N(S) = \bigcup_{x\in S} N(x)$. 
\end{theorem}

Let $G$ be a graph on $n$ vertices. 
We define the following matrix, known as the \textit{adjacency matrix} of $G$, as the characteristic function of $E(G)$: that is, writing $\Adj(G) = A$, we have $A = (a_{xy})_{x,y\in V(G)}$, with $a_{xy} =1$ if $xy \in E(G)$ and 0 otherwise. 
Notice that the adjacency matrix of a graph is a self-adjoint matrix.

Given an integer $k\geq 1$, a \textit{walk} of length $k$ in $G$ is a sequence $x_1,\ldots,x_k \in V(G)$ of vertices in $G$ such that $x_i \sim_G x_{i+1}$ for all $1\leq i\leq k-1$. 
A \textit{path} is an injective walk. 
Given $x$ and $y\in V(G)$, we denote by $d_G(x,y)$, or sometimes $d(x,y)$, the minimum of the length of a walk in $G$ from $x$ to $y$. 
Notice this quantity can be infinite when $x$ and $y$ are not in the same connected components of $G$. 
It is easy to see that $d_G(\cdot,\cdot)$ defines a metric on $V(G)$, sometimes referred to as the \textit{graph metric} on $G$.

The following results are elementary graph theory results and we recall them without proof.

\begin{lemma}\label{lem:power_adja}
	Let $k\geq 1$ and fix $x$, $y\in V(G)$. 
	Then $[\Adj(G)^k]_{xy}$ is the number of walks of length $k$ from $x$ to $y$. 
	In particular, $[\Adj(G)^2]_{xx} = deg(x)$. 
\end{lemma}

The \textit{complement} of a graph $G$ is the graph $G^c$ with $V(G^c) = V(G)$ and such that for every $x$, $y\in V(G)$ with $x\neq y$, we have $xy \in E(G^c)$ if and only if $xy \notin E(G)$. 
It is easy to see that $\Adj(G^c) = J_n - \Adj(G) - I_n$, where $J_n$ is the matrix with all coefficients equal to 1. 

\begin{lemma}\label{lem:complement}
	If $G$ is disconnected, then $G^c$ has diameter at most 2. 
	In particular, $G^c$ is connected.
\end{lemma}

We will often use the Kronecker symbol $\delta_{x,y}$ which is equal to 1 when $x=y$ and 0 otherwise, and in different contexts.

For $n\geq 1$, we denote by $S_n$ the permutation group on $n$ elements. 
For $\sigma \in S_n$, we denote by $P_\sigma$ the permutation matrix associated to it, which we recall is defined by $[P_\sigma]_{ij} = \delta_{i,\sigma(j)}$.

\begin{lemma}\label{lem:isomorphism}
	Let $G$ and $H$ be two graphs with vertex set $\{1,\ldots,n\}$ and let $\sigma \in S_n$. 
	Then $P_{\sigma\inv} \Adj(H) P_\sigma = \Adj(G)$ if and only if $\sigma$ induces an isomorphism between $G$ and $H$.
\end{lemma}

\begin{corollary}\label{coro:automorphism}
	Let $G$ be a graph with $V(G) = \{1,\ldots,n\}$ and let $\sigma \in S_n$. 
	Then $P_\sigma$ and $\Adj(G)$ commute if and only if $\sigma$ induces an automorphism of $G$.
\end{corollary}

For a graph $G$, the set of automorphisms of $G$ naturally forms a group, the automorphism group of $G$, denoted by $\Aut(G)$. 
For $f\in \Aut(G)$, the \textit{support} of $f$, denoted by $\Supp(f)$, is defined by
\[ \Supp(f) = \{ x\in V(G) \mid f(x) \neq x\}. \]

$G$ is said to satisfy \textit{Schmidt's criterion} if it has two nontrivial automorphisms with disjoint support. 
This definition is due to Schmidt~\cite{SchmCrit}.

\subsection{The quantum automorphism group of a graph}\label{subsec:qu_G}

Since we will only work with compact matrix quantum groups which are subgroups of the permutation quantum groups, we follow the practice of Freslon~\cite{FreslonBook} and define them directly. 
We refer the reader to~\cite{FreslonBook} for a general reference on the theory of compact matrix quantum groups.

Recall that given an integer $n\geq 0$, we have the involutive ring of $*$-polynomials on $n$ variables $\C[x_1,\ldots,x_n,x_1^*,\ldots,x_n^*]$, whose elements are called \textit{$*$-polynomials}. 

Fix $n\in \N$ and let $x =(x_{ij})_{1\leq i,j\leq n}$ be $n^2$ variables. 
Consider the following sets of $*$-polynomials:
\begin{enumerate}
	\item $\Pp_0 = \{ x_{ij}^* - x_{ij} \mid 1\leq i,j\leq n\} \cup \{ x_{ij}^2-x_{ij} \mid 1\leq i,j\leq n\}$,
	\item $\Pp_1 = \{ \sum_{j=1}^n x_{ij} - 1 \mid 1\leq i\leq n\} \cup \{ \sum_{i=1}^n x_{ij} -1 \mid 1\leq j\leq n\}$,
	\item $\Pp_2 = \{ x_{ij}x_{ik} - \delta_k^jx_{ij} \mid 1\leq i,j,k\leq n\} \cup \{ x_{ij}x_{kj} - \delta_k^ix_{ij} \mid 1\leq i,j,k\leq n\}$.
\end{enumerate}
Let $\Pp_s = \Pp_0 \cup \Pp_1 \cup \Pp_2$. 
We define the $*$-algebra $\Aa_s(n)$ to be the universal $*$-algebra generated by $n^2$ variables with relations given by $\Pp_s$. 

\begin{remark}\label{rem:univ_c*_permutation}
	The relations $\Pp_s$ define a universal $C^*$-algebra in which $\Aa_s(n)$ embeds itself. 
	In a $C^*$-algebra, relations $\Pp_2$ would be redundant, but they might not be for a $*$-algebra. 
	At the time of the writing, we do not know of a counter-example, while thinking it should exist.
\end{remark}

The following theorem is due to Wang~\cite{Wang}. 
The present approach is a reformulation of Wang's definition of $S_n^+$.

\begin{theorem}\label{thm:wang_sn+}
	There is a $*$-morphism $\Delta \colon \Aa_s(n) \to \Aa_s(n)\otimes \Aa_s(n)$ such that $\Delta(x_{ij}) = \sum_{k=1}^n x_{ik}\otimes x_{kj}$.
\end{theorem}

\begin{remark}\label{rk:wang_sn+}
	To prove Theorem~\ref{thm:wang_sn+}, by universality, it is enough to show that if $p = (p_{ij})_{1\leq i,j\leq n}$ satisfies relations~$\Pp_s$, then so does $(\sum_{k=1}^n p_{ik}\otimes p_{kj})_{1\leq i,j\leq n}$. 
	The latter is true and easy to check, we will use it freely.
\end{remark}

In this article, we will only work with quantum groups which are quantum subgroups of $S_n^+$ for some $n\geq 1$. 
We refer to them as \textit{quantum permutation groups}. 
They are compact quantum groups, and even quantum compact matrix pseudogroups. 
However, there is a difficulty in the literature regarding the general definition of a compact quantum group: it involves a tensor product of $C^*$-algebras. 
This can bring unnecessary analytical complications when one is dealing with algebraic notions, which in the case of quantum permutation groups case can be avoided by working at the level of $*$-algebras and sticking to the algebraic tensor product of $*$-algebras. 

Hence, following Freslon~\cite{FreslonBook}, for $n\in \N$, we call a pair $(\Aa,p)$, where $\Aa$ is a unital $*$-algebra generated by $p = (p_{ij})_{1\leq i,j\leq n}$, a \textit{quantum permutation group of size $n$} if:
\begin{enumerate}
	\item the relations $\Pp_s$ are satisfied,
	\item\label{rel:comultiplication} there is a $*$-morphism $\Delta \colon \Aa \to \Aa \otimes \Aa$ satisfying $\Delta(p_{ij}) = \sum_{k=1}^n p_{ik}\otimes p_{kj}$.
\end{enumerate}
Notice that, when it exists, the morphism $\Delta$ is unique, since the image of the generators of the $*$-algebra $\Aa$ is given. 
The morphism $\Delta$ is called the \textit{comultiplication} in $\Aa$. 
If $\Gamma = (A,p)$ is a quantum permutation group of size $n$, we call $\Aa$ its \textit{algebra of coefficients} and denote it by $\Oo(\Gamma)$. 
Because it is generated by projections (hence of bounded norm in a representation), it admits a universal $C^*$-enveloppe that we will denote by~$C(\Gamma)$. 

We will also use the same notation $C(G)$ when $G$ is a classical group to refer to the Gelfand dual of $G$, that is, the set of continuous functions from $G$ to $\C$. 
This is a compact quantum group with abelian $C^*$-algebra. 
When $G$ is finite, it fits into our context in the following way: let $n = \# G$ and label the elements of $G$ by $g_1,\ldots,g_n$. 
For $1\leq i,j\leq n$, define $f_{ij} \colon G\to \C$ by $f_{ij}(g) = \delta_{g_i,gg_j}$. 
Now let $\Aa$ be the $*$-subalgebra of $C(G)$ spanned by the $f_{ij}$. 
It is easily observed that $(\Aa,f)$, where $f = (f_{ij})$, satisfies the axioms of a quantum permutation group. 
Moreover, the comultiplication is the dual of the usual multiplication in the group. 
Thus we also denote $\Aa$ by $\Oo(G)$ in that case. 

The result of Wang can then be restated as the fact that the pair $(\Aa_s(n),p)$ is a quantum permutation group (of size $n$). 
Following the usual convention, we will denote it by $S_n^+$. 

Let $n$ and $m \in \N$ and consider $\Aa = (A,(x_{ij})_{1\leq i,j\leq n})$ and $\Bb = (B,(y_{kl})_{1\leq k,l\leq m})$ two quantum permutation groups, of size $n$ and $m$ respectively. 
Denote by $\Delta_A$ (resp. $\Delta_B$) the comultiplication in $\Aa$ (resp. in $\Bb$). 
A \textit{morphism of quantum groups} from $\Aa$ to $\Bb$ is a $*$-morphism $f \colon A \to B$ such that the following diagram commutes:
\begin{center}
	\begin{tikzcd}
		A \arrow[r, "\Delta_A"] \arrow[d, "f" swap] & A\otimes A \arrow[d, "f\otimes f"]\\
		B \arrow[r, "\Delta_B"] & B\otimes B.
	\end{tikzcd}
\end{center}
Quantum permutation groups (of arbitrary sizes) and morphisms of quantum groups satisfy the axioms of a category, allowing one to talk about isomorphism of quantum groups. 
The following lemma allows us to easily build quantum permutation groups.

\begin{lemma}\label{lem:qu_subgroups}
	Let $n\in \N$ and let $R$ be a set of $*$-polynomials in $n^2$ variables $x=(x_{ij})_{1\leq i,j\leq n}$. 
	Let $A$ be the universal $*$-algebra generated by $x$ and the relations $\Pp_s \cup R$. 
	For $1\leq i,j\leq n$, define $y_{ij} \in A\otimes A$ by $y_{ij} = \sum_{k=1}^n x_{ik}\otimes x_{kj}$, and let $y = (y_{ij})_{1\leq i,j\leq n}$. 
	If $P(y) = 0$ for all $P\in R$, then $(A,x)$ is a quantum permutation group.
\end{lemma}

\begin{proof}
	The only thing we need to check is the existence of the comultiplication. 
	By universality of $A$, it is enough to check that $y$ satisfy the relations $\Pp_s\cup R$ in $A\otimes A$. 
	It follows from remark~\ref{rk:wang_sn+} that $y$ satisfies the relations $\Pp_s$, and by assumption, $y$ satisfies the relations $R$. 
	This implies the result.
\end{proof}

Lemma~\ref{lem:qu_subgroups} gives a surjective quantum group morphism from $S_n^+$ to $(A,x)$. 
In the noncommutative setting, this motivates us to call $(A,x)$ a \textit{quantum subgroup} of $S_n^+$. 
In fact, all quantum permutation groups of size $n$ are quantum subgroups of $S_n^+$.

\begin{example}\label{ex:sn}
	Let $R$ be the relations given by the commutators of all pairs of variables. 
	Then one obtains for $A$ the algebra $C(S_n)$ of continuous complex functions on the permutation group $S_n$ with pointwise multiplication. 
	As mentioned before about finite groups, the generators can be identified by considering $x_{ij} \colon \sigma \mapsto \delta_{\sigma(i),j}$. 
	The comultiplication is the Gelfand dual of the multiplication in $S_n$. 
	Finally, let us mention that $\Oo(S_n^+)$ is commutative for $n\geq 3$ and noncommutative and infinite-dimensional for $n\geq 4$.
\end{example}

Bichon~\cite{Bichon2003} defines the quantum automorphism group of a graph on $n$ vertices as a certain quantum subgroup of $S_n^+$. 
Actually, his definition is slightly different from the modern one that we give here, and which seems to come from work of Banica (see for instance~\cite{banica2005quantum}). 
Theorem~\ref{thm:qu_G} is a rewriting of these now standard notions with our terminology.

Let $G$ be a graph with vertex set $V(G) = \{1,\ldots,n\}$ for some integer $n\geq 1$. 
Given two matrices $u$ and $v$, we denote their commutator by $[u,v] = uv-vu$. 
We define a set $R_G$ of $*$-polynomials on $n^2$ variables by $R_G = \{ [Adj(G),\cdot] \}$. 
This makes sense since the function $x \mapsto [\Adj(G), x]$ is a $*$-polynomial in $n^2$ variables. 

\begin{theorem}\label{thm:qu_G}
	Let $A$ be the universal unital $*$-algebra generated by a family $p$ of $n^2$ variables satisfying relations $\Pp_s\cup R_G$. 
	Then $(A,p)$ is a quantum permutation group. 
\end{theorem}

\begin{proof}
	Let $q_{ij} = \sum_{k=1}^n p_{ik}\otimes p_{kj} \in A\otimes A$ and let $q = (q_{ij})_{1\leq i,j\leq n}$. 
	By lemma~\ref{lem:qu_subgroups}, we only need to check that $[\Adj(G),q] = 0$. 
	By a slight abuse, we work with matrices directly indexed by $V(G)^2$. 
	For $x$, $y\in V(G)$, we have that the coefficient $xy$ of $\Adj(G)q$ is:
	\begin{align*}
		[\Adj(G)q]_{xy} &= \sum_{z\in V(G)} a_{xz}q_{zy}\\
		&= \sum_{z\in V(G)} \sum_{t\in V(G)} a_{xz} p_{zt}\otimes p_{ty}\\
		&= \sum_{t\in V(G)} \left( \sum_{z\in V(G)} a_{xz}p_{zt}\right) \otimes p_{ty}\\
		&= \sum_{t\in V(G)} \left(\sum_{z\in V(G)} p_{xz}a_{zt} \right) \otimes p_{ty}\ \text{since $\Adj(G)$ and $p$ commute}\\
		&= \sum_{z\in V(G)} p_{xz} \otimes \left(\sum_{t\in V(G)} a_{zt} p_{ty}\right)\\
		&= \sum_{z\in V(G)} p_{xz}\otimes \left( \sum_{t\in V(G)} p_{zt}a_{ty}\right)\ \text{idem}\\
		&= \sum_{t\in V(G)} q_{xt}a_{ty}\\
		&= [q\Adj(G)]_{xy},
	\end{align*}
	which concludes the proof.
\end{proof}

\begin{remark}\label{rem:qu_G_abelian}
	Theorem~\ref{thm:qu_G} and Example~\ref{ex:sn} show that the abelianisation $(B,q)$ of $(A,p)$ is a quotient of $C(S_n)$: it is the quotient of $C(S_n)$ by the relations $R_G$. 
	Writing it down, we see that the relations $R_G$ imply that the generators of $(B,q)$ are given by the permutation matrices which commute with $\Adj(G)$, that is, by the classical automorphisms of $G$ by Corollary~\ref{coro:automorphism}. 
	This shows that $(B,q)$ corresponds to the dual of $\Aut(G)$ -- that is, $B = C(\Aut(G))$, and the comultiplication in $B$ corresponds to the dual of the group multiplication in $\Aut(G)$. 
\end{remark}

\subsection{Quantum symmetry and the Schmidt alternative}\label{subsec:qu_symm_schmidt}

Remark~\ref{rem:qu_G_abelian} has the following important consequence: $C(\Qu(G)) = C(\Aut(G))$ if and only if $C(\Qu(G))$ is abelian. 

When it is not, one can argue that the quantum automorphism group of $G$ contains strictly more than the classical symmetries of $G$, in which case we say that $G$ has \textit{quantum symmetry}. 
Notice how this is a property of the underlying $*$-algebra of $\Qu(G)$. 

We point out something that can lead to confusion: following standard terminology, by a \textit{quantum asymmetric} graph, we mean a graph with trivial quantum automorphism group, \emph{not} a graph which does not have quantum symmetry. 
This will be always clear from context (and we point out that a quantum asymmetric graph indeed does not have quantum symmetry).

Let $X$ be a unital $C^*$-algebra and let $n\in \N$.
A \textit{magic unitary} with coefficients in $X$ is a matrix $U \in \Mm_n(X)$ satisfying relations $\Pp_s$. 
Let us show that in this case relations $\Pp_2$ are redundant.

\begin{lemma}\label{lem:mu_p2_redundant}
	Let $U\in \Mm_n(X)$. 
	The matrix $U$ is a magic unitary if and only if it satisfies relations $\Pp_0\cup \Pp_1$, that is, its coefficients are self-adjoint projections summing to 1 on rows and columns.
\end{lemma}

\begin{proof}
	The only thing to prove is that if $U$ satisfies $\Pp_0\cup \Pp_1$, then it satisfies $\Pp_2$. 
	Let $1\leq i,j,l \leq n$ with $l\neq j$. 
	It comes:
	\[ u_{ij} = u_{ij}\left(\sum_{k=1}^n u_{ik}\right)u_{ij} = \sum_{k=1}^n u_{ij}u_{ik}u_{ij}.\]
	Since all coefficients are projections, we obtain that $0 = \sum_{k\neq j} u_{ij}u_{ik}u_{ij}$. 
	Now, since we are working in a $C^*$-algebra, we have $0\leq u_{ij}u_{il}u_{ij} \leq \sum_{k\neq j} u_{ij}u_{ik}u_{ij} = 0$, so $u_{ij}u_{il}u_{ij}=0$. 
	So we have $(u_{ij}u_{il})(u_{ij}u_{il})^*=0$, which implies that $u_{ij}u_{il}=0$, as desired. 

	Finally, noticing that the transpose of $U$ also satisfies relations $\Pp_0\cup \Pp_1$, we can apply what precedes to show that $u_{ji}u_{li} = 0$ for $1\leq i,j,l\leq n$ with $j\neq l$. 
	This concludes the proof.
\end{proof}

\begin{remark}
	Notice that a magic unitary is a unitary in $\Mm_n(X)$.
\end{remark}

Let $G$ be a graph with vertex set $\{1,\ldots,n\}$. 
We say that a magic unitary is \textit{adapted to $G$} if it commutes with the adjacency matrix of $G$. 

Magic unitaries correspond to representations of $\Aa_s(n)$ in a $C^*$-algebra. 
Magic unitaries adapted to $G$ correspond to representations of $\Oo(\Qu(G))$ in a $C^*$-algebra. 
Since $\Oo(\Qu(G))$ is abelian if and only if its enveloppe is, we can characterise quantum symmetry with magic unitaries. 

\begin{lemma}\label{lem:characterisation_qs}
	Let $n\in \N$ and let $G$ be a graph with $V(G) = \{1,\ldots,n\}$. 
	Then $G$ has quantum symmetry if and only if there exists a unital $C^*$-algebra $X$ and a magic unitary $U \in \Mm_n(X)$ adapted to $G$ with at least two noncommuting coefficients.
\end{lemma}

\begin{proof}
	Let $(A,p) = \Qu(G)$. 
	If $G$ does not have quantum symmetry, by definition, $\Oo(\Qu(G))$ is not abelian, and so the coefficients of $p$ do not commute. 
	Since $\Oo(\Qu(G))$ embeds into its universal $C^*$-enveloppe $C(\Qu(G))$, we have that $p \in \Mm_n(C(\Qu(G)))$ is a magic unitary adapted to $G$ with noncommuting coefficients, as desired. 

	Conversely, assume there exists $X$ and $U$ as in the statement. 
	By universality of $C(\Qu(G))$, there is a $*$-morphism $f \colon C(\Qu(G)) \to X$ such that $f(p_{ij}) = u_{ij}$, for $1\leq i,j\leq n$. 
	Since not all coefficients of $U$ commute, not all coefficients of $p$ commute neither, so $\Oo(\Qu(G))$ is not abelian, and $G$ has quantum symmetry, as desired. 

	This concludes the proof.
\end{proof}

It is natural to ask whether there actually are graphs with quantum symmetry. 
The answer is positive.

\begin{lemma}\label{lem:k4}
	The complete graph on 4 vertices $K_4$ has quantum symmetry.
\end{lemma}

\begin{proof}
	Indeed, notice that the relations defining magic unitaries already imply commutation with $\Adj(K_4)$, that is, every $4\times 4$ magic unitary commutes with $\Adj(K_4)$. 
	So by lemma~\ref{lem:characterisation_qs} it is enough to exhibit a magic unitary of dimension 4 with noncommuting coefficients. 
	Let $(p,q)$ be any pair of noncommuting projections, for instance, $p = \begin{pmatrix}
		1 & 0\\
		0 & 0
	\end{pmatrix}$ and $q = \frac{1}{2} \begin{pmatrix}
		1 & 1\\
		1 & 1
	\end{pmatrix}$. 
	Now $U = \begin{pmatrix}
		p & 1-p & 0 & 0\\
		1-p & p & 0 & 0\\
		0 & 0 & q & 1-q\\
		0 & 0 & 1-q & q
	\end{pmatrix}$ is a magic unitary with coefficients in $\Mm_2(\C)$ which are not all commuting.
\end{proof}

The precedent lemma can be generalised as follows. 
Let $G$ be a graph with $V(G) = \{1,\ldots,n\}$ and let $U$ be a magic unitary adapted to $G$ with coefficients in a unital $C^*$-algebra $X$. 
We define its \textit{support} to be $\Supp(U) = \{x \in V(G) \mid u_{xx}\neq 1\}$. 

\begin{theorem}\label{thm:qu_schmidt}
	Let $G$ be a graph on $n$ vertices for some $n\in \N$. 
	Let $X$ and $Y$ be two unital $C^*$-algebras, and let $U \in \Mm_n(X)$ and $V\in \Mm_n(Y)$ be two magic unitaries adapted to $G$. 
	Assume that:
	\begin{itemize}
		\item both $U\neq I_n$ and $V\neq I_n$,
		\item and $\Supp(U)\cap \Supp(V) = \varnothing$.
	\end{itemize}
	Then, there are two nontrivial $C^*$-algebras $X'$ and $Y'$ such that their unital free product $X'*_1 Y'$ is a quotient of $C(\Qu(G))$. 
	In particular, $C(\Qu(G))$ is nonabelian and infinite-dimensional, and $G$ has quantum symmetry.
\end{theorem}

\begin{proof}
	We need to bring ourselves back to the case where $U$ and $V$ have nonscalar coefficients. 
	Assume all coefficients of $U$ are scalar. 
	Then $U = P_f$ is a permutation matrix with $f$ an automorphism of $G$. 
	Since $U\neq 1$, we have $f\neq \id$, so $f$ has a finite order $d\neq 1$. 
	Notice that the function $T \colon \{1,\ldots,d\} \to \Mm_n(\C)$ defined by $T(k) = P_{f^k}$ gives a well-defined magic unitary (with slight abuse of notation) $T \in \Mm_n(\C^d)$ whose coefficients are not all scalar (this is actually the unitary representation of $\Z/d\Z$ given by $f$ as in the classical proof of Schmidt's criterion). 
	We claim it is adapted to $G$: indeed, this is equivalent to the fact that $P_{f^k}$ commutes with $\Adj(G)$ for all $k$, which is immediate. 
	Now $T$ is a magic unitary adapted to $G$ with nonscalar coefficients, and it clearly has the same support as $U = P_f$. 
	This shows that up to changing $U$ into $T$ and $X$ into $\C^d$ we can assume that both $U$ and $V$ have nonscalar coefficients.

	So from now on we assume that $U$ and $V$ have nonscalar coefficients. 
	Up to replacing $X$ and $Y$ with the $C^*$-subalgebras generated by the coefficients of $U$ and $V$ respectively, we bring ourselves to the situation where the coefficients of $U$ generate $X$, the coefficients of $V$ generate $Y$, and $X\neq \C$ and $Y\neq \C$. 

	Fix an ordering of $V(G)$ with $\Supp(U) <  \Supp(V) \leq V(G)\setminus (\Supp(U) \cup \Supp(V))$. 
	Notice this is possible since $\Supp(U) \cap \Supp(V) = \varnothing$. 
	With this ordering, $U$ and $V$ are block-diagonal with $U = \Diag(U_1,1,1)$ and $V= \Diag(1,V_1,1)$, where $U_1$ and $V_1$ are magic unitaries of the corresponding dimensions, and where 1 stands for the identity matrices of the correct dimensions. 
	Let $Z = X *_1 Y$ be the unital free product of $X$ and $Y$, which comes together with the inclusions $i_X \colon X \to Z$ and $i_Y \colon Y \to Z$. 
	We write $W_1=i_X(U_1)$, $W_2 = i_Y(V_1)$ and $W_3 = 1$. 
	Set $W = \Diag(W_1,W_2,1)$: it is clearly a well-defined magic unitary with coefficients in $Z$. 
	Let us check it is adapted to $G$. 
	Let us write the adjacency matrix of $G$ in the corresponding block form $A = (A_{ij})_{1\leq i,j\leq 3}$. 
	It comes:
	\begin{align*}
		WA &= \begin{pmatrix}
			i_X(U_1A_{11}) & i_X(U_1A_{12}) & i_X(U_1A_{13})\\
			i_Y(V_1A_{21}) & i_Y(V_1A_{22}) & i_Y(V_1A_{23})\\
			A_{31} & A_{32} & A_{33}
		\end{pmatrix}\\
		&= \begin{pmatrix}
			i_X(A_{11}U_1) & A_{12} & A_{13}\\
			A_{21} & i_Y(A_{22}V_1) & A_{23}\\
			A_{31} & A_{32} & A_{33}
		\end{pmatrix}\ \text{since $U$ and $V$ commute with $A$}\\
		&= AW.
	\end{align*}
	So $W$ is a magic unitary adapted to $G$. 
	By universality, there is a $*$-morphism $f\colon C(\Qu(G))$ to $Z$ defined by $f(u_{ij}) = w_{ij}$, where $u = (u_{ij})_{1\leq i,j\leq n}$ is the fundamental representation of $\Qu(G)$. 
	By construction, coefficients of $U_1$ generate $X$, and the coefficients of $V_1$ generate $Y$, so the coefficients of $W$ generate $Z$. 
	So $f$ is a surjective $*$-morphism. 
	This shows the desired result. 
\end{proof}

The particular case when $U$ and $V$ are both permutation matrices is known as \textit{Schmidt's criterion} (see~\cite{SchmCrit}, Theorem 2.2 for the original statement, and also~\cite{FreslonBook}, Exercise 7.6 and its correction). 
Hence we refer to Theorem~\ref{thm:qu_schmidt} as the \textit{quantum Schmidt criterion}. 

Schmidt's criterion will be of particular interest in the present article. 
Indeed, this criterion is a sufficient condition for quantum symmetry and is relatively easy to check (in particular, it is a decidable criterion). 
However, it is in general not necessary: in a recent preprint, Schmidt and Roberson~\cite{RobSchm} prove the existence of a graph $G$ which has quantum symmetry and such that $C(\Qu(G))$ is finite-dimensional. 
By Theorem~\ref{thm:qu_schmidt}, we have that $G$ cannot satisfy the quantum Schmidt criterion -- let alone the classical one. 
So $G$ does not satisfy Schmidt's criterion and has quantum symmetry, which proves that Schmidt's criterion is not a necessary condition for quantum symmetry. 

This motivated the author to define the \textit{Schmidt alternative.} 
Let $\Ff$ be a class of graphs. 
We say that $\Ff$ satisfies the Schmidt alternative if a graph of $\Ff$ has quantum symmetry if and only if it satisfies Schmidt's criterion. 
When this is the case, this gives a satisfactory answer to the quantum symmetry problem for $\Ff$. 
In order to give some examples, we will prove that cographs, forests, tree-cographs, and $\Gg_5$-cographs (respectively Theorem~\ref{thm:cographs_QS}, Theorem~\ref{thm:tractable_forests}, Theorem~\ref{thm:tractable_tree_cographs}, and Theorem~\ref{thm:G5}) satisfy the Schmidt alternative. 
These results should not be seen as an objective \textit{per se} but rather as a motivation for this line of research. 
We will also prove general results about the Schmidt alternative in Section~\ref{sec:tractable}, which suggest that families satisfying the Schmidt alternative might have an accessible quantum structure, prompting them as natural first candidates for a systematic study of quantum properties of graphs.

\subsection{Quantum isomorphism}\label{subsec:qu_iso}

A \textit{quantum isomorphism} from a graph $G$ to a graph $H$ is a magic unitary $P$ such that $P\Adj(G) = \Adj(H)P$. 
If there is a quantum isomorphism from $G$ to $H$, we say that $G$ is \textit{quantum isomorphic} to $H$. 
Notice that by definition $G$ and $H$ have the same number of vertices in that case. 

The fact that this is an equivalence relation is the object of the following lemmas. 
The proof is folklore and has been communicated to the author by Freslon. 
However, we did not find it written anywhere, so we include it here for completeness. 
We start with an operation on magic unitaries which is interesting on its own.

Let $X$ and $Y$ be two $*$-algebras, and let $P\in \Mm_n(X)$ and $Q\in \Mm_n(Y)$ be two matrices satisfying relations $\Pp_s$, where $n\in \N$ is some integer. 
Let $Z = X \otimes Y$ be the tensor product of $X$ and $Y$ (this is the algebraic tensor product of $*$-algebras). 
Define $P*Q = R \in \Mm_n(Z)$ by
\[ r_{ij} = \sum_{k=1}^n p_{ik}\otimes q_{kj}.\]

\begin{lemma}\label{lem:convo_mu}
	The matrix $P*Q \in \Mm_n(X\otimes Y)$ satisfies relations $\Pp_s$. 
\end{lemma}

\begin{proof}
	Let $R=P*Q$. 
	For $1\leq i,j\leq n$, it is clear that $r_{ij}$ is self-adjoint, and it comes:
	\begin{align*}
		r_{ij}^2 &= \sum_{k=1}^n \sum_{l=1}^n p_{ik}p_{il}\otimes q_{kj}q_{lj}\\
		&= \sum_{k=1}^n p_{ik} \otimes q_{kj}\quad \text{by $\Pp_2$}\\
		&= r_{ij},
	\end{align*}
	so relations $\Pp_0$ are satisfied.
	
	For $1\leq i\leq n$, we have:
	\begin{align*}
		\sum_{j=1}^n r_{ij} &= \sum_{j=1}^n \sum_{k=1}^n p_{ik}\otimes q_{kj}\\
		&= \sum_{k=1}^n p_{ik}\otimes \left(\sum_{j=1}^n q_{kj}\right)\\
		&= \left(\sum_{k=1}^n p_{ik}\right) \otimes 1\\
		&= 1\otimes 1,
	\end{align*}
	as desired, and:
	\begin{align*}
		\sum_{j=1}^n r_{ji} &= \sum_{j=1}^n \sum_{k=1}^n p_{jk}\otimes q_{ki}\\
		&= \sum_{k=1}^n \left(\sum_{j=1}^n p_{jk}\right) \otimes q_{ki}\\
		&= 1 \otimes \left(\sum_{k=1}^n q_{ki} \right)\\
		&= 1\otimes 1,
	\end{align*}
	as desired, so relations $\Pp_1$ are satisfied. 

	Finally, let $1\leq i,j,k\leq n$. 
	We have:
	\begin{align*}
		r_{ij}r_{ik} &= \sum_{l=1}^n \sum_{m=1}^n p_{il}p_{im}\otimes q_{lj}q_{mk}\\
		&= \sum_{l=1}^n p_{il} \otimes q_{lj}q_{lk}\\
		&= \delta_k^j \sum_{l=1}^n p_{il}\otimes q_{lj}\\
		&= \delta_k^j r_{ij},
	\end{align*}
	and similarly $r_{ij}r_{kj} = \delta_k^j r_{ij}$, so relations $\Pp_2$ are satisfied. 
	This concludes the proof that $R$ satisfies $\Pp_s = \Pp_0 \cup \Pp_1 \cup \Pp_2$.
\end{proof}

\begin{lemma}\label{lem:transitivity_qi}
	Let $G$, $H$, and $L$ be three graphs. 
	Assume that $P$ is a quantum isomorphism from $G$ to $H$ and that $Q$ is a quantum isomorphism from $H$ to $L$. 
	Then $Q*P$ gives rise to a quantum isomorphism from $G$ to $L$.
\end{lemma}

\begin{proof}
	Let $X$ and $Y$ be the $C^*$-algebras of coefficients of $P$ and $Q$ respectively. 
	Let $Z = X\otimes_{max} Y$ (or any $C^*$-algebras that contain $X\otimes_{alg} Y$). 
	By Lemma~\ref{lem:convo_mu}, we know that $P*Q$ is a magic unitary with coefficients in $Z$. 
	So we only have to check the commutation relations. 
	Let $A = \Adj(G)$, $B=\Adj(H)$, $C=\Adj(L)$, and $R = Q*P$. 
	By assumption, we have $PA = BP$ and $QB = CQ$. 
	For $x \in V(L)$ and $y\in V(G)$, it comes:
	\begin{align*}
		[RA]_{xy} &= \sum_{z\in V(G)} r_{xz}a_{zy}\\
		&= \sum_{z\in V(G)} \sum_{t \in V(H)} q_{xt}\otimes p_{tz}a_{zy}\\
		&= \sum_{t\in V(H)} q_{xt} \otimes [PA]_{ty}\\
		&= \sum_{t\in V(H)} p_{xt} \otimes [BP]_{ty}\\
		&= \sum_{t\in V(H)} \sum_{z\in V(H)} q_{xt} \otimes b_{tz}p_{zy}\\
		&= \sum_{z\in V(H)} [QB]_{xz} \otimes p_{zy}\\
		&= \sum_{z\in V(H)} [CQ]_{xz} \otimes p_{zy}\\
		&= \sum_{z\in V(H)} \sum_{t\in V(L)} c_{xt} q_{tz} \otimes p_{zy}\\
		&= \sum_{t\in V(L)} c_{xt} r_{ty}\\
		&= [CR]_{xy},
	\end{align*}
	as desired. 
	This concludes the proof.
\end{proof}

We can now prove the following.

\begin{theorem}\label{thm:qi_is_equivalence_relation}
	Quantum isomorphism is an equivalence relation.
\end{theorem}

\begin{proof}
	Since the identity matrix in $\Mm_n(\C)$ is a magic unitary, and since the adjoint of a magic unitary is a magic unitary, it is clear that quantum isomorphism os reflexive and symmetric. 
	Finally, it is transitive by Lemma~\ref{lem:transitivity_qi}. 
	Thus quantum isomorphism is an equivalence relation.
\end{proof}

When $G$ and $H$ are quantum isomorphic, we write $G\qi H$. 
We conclude this section by a few useful results and remarks. 

\begin{lemma}\label{lem:classical_qi_is_i}
	Let $G$ and $H$ be two graphs and let $U$ be a quantum isomorphism from $G$ to $H$. 
	Assume that the coefficients of $U$ commute. 
	Then $G$ is isomorphic to $H$.
\end{lemma}

\begin{proof}
	Let $n = \abs{V(G)} = \abs{V(H)}$ by definition. 
	Let $X$ be the $C^*$-algebra generated by the coefficients of $U$. 
	By assumption, it is a commutative unital $C^*$-algebra, so by the classification of abelian $C^*$-algebras there exists a nonempty compact Hausdorff space $X$ such that $X$ is isomorphic to the $C^*$-algebra $C(X)$. 
	We will work through this isomorphism and consider $u_{ij} \in C(X)$ for all $1\leq i,j\leq n$. 
	For $1\leq i,j\leq n$, we have $u_{ij}^2 = u_{ij}$, so, for every $x \in X$, we have $u_{ij}(x) \in \{0,1\}$. 
	Fix $x_0 \in X$. 
	Now, notice that the evaluation at $x_0$ is a unital $*$-morphism from $C(X)$ to $\C$. 
	This implies that $U(x_0) = (u_{ij}(x_0))_{1\leq i,j\leq n} \in \Mm_n(\{0,1\})$ is a quantum isomorphism from $G$ to $H$. 
	Since its coefficients are 0 and 1, it is a permutation matrix. 
	Hence, by Lemma~\ref{lem:isomorphism}, it is given by an isomorphism between $G$ and $H$, and $G$ and $H$ are isomorphic. 
	This concludes the proof.
\end{proof}

\begin{lemma}\label{lem:square_mu}
	Let $X$ be a unital $C^*$-algebra and let $m$ and $n\in \N$. 
	Let $P\in \Mm_{m,n}(X)$. 
	Assue that, for all $1\leq i\leq m$ and for all $1\leq j\leq n$, we have
	\[ \sum_{k=1}^n p_{ik} = 1_X = \sum_{k=1}^m p_{kj}.\]
	Then $m=n$.
\end{lemma}

\begin{proof}
	This is a double-summation:
	\[m.1_X = \sum_{i=1}^m \sum_{j=1}^n p_{ij} = n.1_X.\]
	Since $X$ is nontrivial, we have $1_X \neq 0$ so $m=n$, as desired.
\end{proof}

Lemma~\ref{lem:square_mu} has the following consequence: if one had allowed magic unitaries to be rectangular, that is, had defined them to be projection-valued rectangular matrices with rows and columns summing up to 1, then one would not have obtained any new magic unitaries. 
In particular, defining quantum isomorphism in this way would not change the relation. 
This justifies the saying that quantum isomorphism preserves the number of vertices (even though in our context it is part of the definition). 
Beside that consequence, Lemma~\ref{lem:square_mu} will also prove itself useful in a few proofs.

The following lemma will often be used without explicit reference.

\begin{lemma}\label{lem:mu_vs_qi}
	Let $G$ and $H$ be two graphs on $n\geq 1$ vertices and let $U$ be a magic unitary of size $n$. 
	Then $U$ is a quantum isomorphism from $H$ to $G$ if and only if $\begin{pmatrix}
		0 & U\\
		U^* & 0
	\end{pmatrix}$ is a magic unitary adapted to $G+H$ with vertices of $G$ enumerated before the ones of $H$.
\end{lemma}

\begin{proof}
	Direct computation.
\end{proof}

\begin{lemma}\label{lem:mu_complement}
	Let $G$ be a graph on $n\geq 1$ vertices and $U$ a magic unitary of size $n$. 
	Then $[U,\Adj(G)] = 0$ if and only if $[U,\Adj(G^c)] = 0$. 
\end{lemma}

\begin{proof}
	This follows the definition of a magic unitary and from the fact that $\Adj(G^c) = J_n -I_n - \Adj(G)$, where $I_n$ is the identity matrix and $J_n$ is the matrix with all coefficients equal to 1.
\end{proof}

It follows immediately that $\Qu(G) = \Qu(G^c)$.

\begin{remark}\label{rk:auto_complement}
	Since permutation matrices are magic unitary, it follows from Lemma~\ref{lem:mu_complement} that $G$ and $G^c$ have the same automorphisms. 
	In particular, we have that $\Aut(G) = \Aut(G^c)$, and that $G$ satisfies Schmidt's criterion if and only if $G^c$ does, that $G$ is asymmetric if and only if $G^c$ is.
\end{remark}

\subsection{Metric aspects of quantum automorphisms}\label{subsec:metric_qu}

Quantum isomorphism is a weaker relation than isomorphism: in~\cite{QiNotI}, the authors exhibit a pair of quantum isomorphic graphs which are not isomorphic. 
However, in practice, many properties remain preserved under quantum isomorphism. 
Indeed, one source of rigidity for quantum isomorphism is that it preserves some metric aspects of the graphs. 
In this section, we gather a few of these results. 
Lemma~\ref{lem:hall_mu} is ours (though the result is well-known in the context of doubly stochastic matrices). 
The rest of the results of this section seem to be more or less well-known but some proofs are new. 
We try to mention sources when we know them.

We start with a key lemma about magic unitaries.

\begin{lemma}\label{lem:hall_mu}
	Let $X$ be a unital $C^*$-algebra. 
	Let $n\geq 1$ and let $P \in \Mm_n(X)$ be a magic unitary. 
	Then there exists a permutation $\sigma \in S_n$ such that, for all $1\leq i\leq n$, we have $p_{i\sigma(i)} \neq 0$.
\end{lemma}

\begin{proof}
	This is a consequence of Hall's marriage theorem. 
	Indeed, let $G$ be the following bipartite graph: let $A = \{ (r,i) \mid 1\leq i\leq n\}$ and $B = \{ (c,j) \mid 1\leq j\leq n\}$ and define $V(G) = A \sqcup B$. 
	For $1\leq i,j\leq n$, we set $(r,i)(c,j) \in E(G)$ if $p_{ij}\neq 0$. 
	Notice that $G$ is a bipartite graph. 
	We claim that $G$ satisfies condition~\ref{hall_condition}. 
	Let $S \subset A$. 
	It comes:
	\begin{align*}
		\abs{S}.1_X &= \sum_{(r,i) \in S} \sum_{j=1}^n p_{ij}\\
		&= \sum_{(r,i) \in S} \sum_{(c,j) \in N(S)} p_{ij}\ \text{since by construction $p_{ij}=0$ if $(c,j) \notin N(S)$}\\
		&= \sum_{(c,j)\in N(S)} \sum_{(r,i) \in S} p_{ij}\\
		&\leq \sum_{(c,j) \in N(S)} \sum_{i=1}^n p_{ij}\\
		&= \abs{N(S)}.1_X,
	\end{align*}
	as desired. 
	Hence, $G$ satisfies condition~\ref{hall_condition}, so, by Hall's marriage theorem (Theorem~\ref{thm:hall}), there exists a matching of $G$ saturating $A$. 
	That is, there is an injective function $\sigma$ from $\{1,\ldots,n\}$ to itself -- hence, a permutation in $S_n$ -- such that, for every $1\leq i\leq n$, we have $(r,i)(c,\sigma(i)) \in E(G)$, i.e. $p_{i\sigma(i)}\neq 0$. 
	This concludes the proof.
\end{proof}

The graph $G$ in the proof of Lemma~\ref{lem:hall_mu} is sometimes referred to as the \textit{positivity graph} of $P$ (in the context of doubly stochastic matrices).

The next lemma is a direct consequence of Lemma 3.2.3 of~\cite{Fulton2006}, with the same proof (that we include for the sake of completeness).

\begin{lemma}\label{lem:degree}
	Let $G$ be a graph, let $U$ be a magic unitary adapted to it, and let $x$, $y\in V(G)$ be such that $u_{xy}\neq 0$. 
	Then $deg(x)=deg(y)$.
\end{lemma}

\begin{proof}
	Because $U$ is adapted to $G$, it commutes with $A$, and so with $A^2$, and we have: 
	\[ \sum_{z\in V(G)} [A^2]_{xz} u_{zy} = [A^2U]_{xy} = [UA^2]_{xy} = \sum_{z\in V(G)} u_{xz} [A^2]_{zy}.\]
	Multiplying by $u_{xy}$ on the left, we reach:
	\[ [A^2]_{xx} u_{xy} = [A^2]_{yy} u_{xy}.\]
	Now because $u_{xy} \neq 0$, we have $[A^2]_{xx} = [A^2]_{yy}$, and by Lemma~\ref{lem:power_adja} we have $[A^2]_{zz} = deg(z)$ for any vertex $z$, which concludes the proof.
\end{proof}

Lemmas~\ref{lem:hall_mu} and Lemma~\ref{lem:degree} already allow us to obtain in an explicit way that quantum isomorphism preserves the number of edges.

\begin{lemma}\label{lem:qi_same_nb_edges}
	Two quantum isomorphic graphs have the same number of edges.
\end{lemma}

\begin{proof}
	Let $G$ and $H$ be two quantum isomorphic graphs on $n\geq 1$ vertices and let $U$ be a quantum isomorphism between them. 
	By Lemma~\ref{lem:hall_mu}, there is an ordering of the vertices of $G$ and $H$ such that $u_{ii} \neq 0$ for all $1\leq i\leq n$. 
	Let $A = \Adj(G)$ and $B= \Adj(H)$ for this ordering. 
	Notice that $\Diag(A,B) = \Adj(G+H)$ and that $\begin{pmatrix}
		0 & U\\
	U^* & 0\end{pmatrix}$ is a magic unitary adapted to $G+H$, so we can apply lemma~\ref{lem:degree} and find that $\deg_G(x_i) = \deg_H(y_i)$, where $(x_1,\ldots,x_n)$ is the enumeration of $V(G)$ used and $(y_1,\ldots,y_n)$ the one of $V(H)$. 
	So $2\abs{E(G)} = \sum_{i=1}^n \deg_G(x_i) = \sum_{i=1}^n \deg_H(y_i) = 2\abs{E(H)}$ and we are done.
\end{proof}

Notice that preservation of the number of vertices (in the sense given by Lemma~\ref{lem:square_mu}) and edges (as given by Lemma~\ref{lem:qi_same_nb_edges}) can also be otbained directly from~\cite{ManRob}: indeed,  if $G \qi H$, because $K_1$ and $K_2$ are planar graphs, we have that $\# V(G) = \# \Hom(K_1,G) = \# \Hom(K_1,H) = \# V(H)$ and $2\# E(G) = \# \Hom(K_2,G) = \# \Hom(K_2,H) = 2\# E(H)$, as desired. 
However, this does not seem to give an explicit way to work with the magic unitaries.

The next lemma is already hinted at in Section 3 of Fulton's thesis~\cite{Fulton2006}, and is already known for the fundamental representation of $\Qu(G)$ (see Lemma~3.2 in~\cite{Schm2020}). 
We generalise it to arbitrary quantum isomorphisms.

\begin{lemma}\label{lem:distance}
	Let $G$, $H$ be a two graphs and $U$ be a quantum isomorphism from $G$ to $H$. 
	Let $x$, $y \in V(H)$ and $a$, $b \in V(G)$ and suppose that $u_{xa}u_{yb}\neq 0$. 
	Then $d_H(x,y) = d_G(a,b)$.
\end{lemma}

\begin{proof}
	Let us prove the contrapositive. 
	Up to symmetry, we can assume that $d(x,y) < d(a,b)$. 
	We set $k = d(x,y)$ and by hypothesis we have $k < +\infty$. 
	Let us write $A = \Adj(G)$ and $B=\Adj(H)$. 
	We have:
	\begin{align*}
		u_{xa} [UA^k]_{xb} u_{yb} &= u_{xa} \left( \sum_{c \in V(G)} u_{xc}[A^k]_{cb} \right) u_{yb}\\
		&= u_{xa}u_{yb} [A^k]_{ab}
	\end{align*}
	on the one hand and
	\begin{align*}
		u_{xa} [B^kU]_{xb} u_{yb} &= u_{xa} \left( \sum_{z \in V(H)} [B^k]_{xz}u_{zb} \right) u_{yb}\\
		&= u_{xa}u_{yb} [B^k]_{xy}
	\end{align*}
	on the other hand. 
	Since $U$ is a quantum isomorphism from $G$ to $H$, we have $UA^k = B^kU$, so $[UA^k]_{xb} = [B^kU]_{xb}$. 
	By what precedes, we reach:
	\[ u_{xa}u_{yb}[B^k]_{xy} = u_{xa}u_{yb} [A^k]_{ab}.\]
	By Lemma~\ref{lem:power_adja}, we have $[A^k]_{ab} = 0$, so $u_{xa}u_{yb}[B^k]_{xy} = 0$. 
	Since $[B^k]_{xy} \neq 0$ (again by lemma~\ref{lem:power_adja}), we obtain that $u_{xa}u_{yb} = 0$, as desired. 
	This concludes the proof.
\end{proof}

This gives us the following corollary. 
Even though the result seems to be well-known, the following elementary proof is apparently not and we include it. 

\begin{corollary}\label{coro:connected_qi}
	If $G \qi H$ and $G$ is connected, then $H$ is connected.
\end{corollary}

\begin{proof}
	Let $U$ be a quantum isomorphism from $G$ to $H$ and assume that $G$ is connected. 
	Take $a$, $b\in V(H)$. 
	We have $1 = \sum_{x\in V(G)} u_{xa}$ so there exists $x\in V(G)$ with $u_{xa}\neq 0$. 
	Now $u_{xa} = u_{xa} 1 = u_{xa} \sum_{y\in V(G)} u_{yb} = \sum_{y\in V(G)} u_{xa}u_{yb}$, so there exists $y\in V(G)$ such that $u_{xa}u_{yb} \neq 0$. 
	By Lemma~\ref{lem:distance}, we have $d(a,b) = d(x,y) < +\infty$. 
	This being true for any $a$, $b\in V(H)$, we conclude that $H$ is connected.
\end{proof}

\section{Quantum isomorphism, center, and sums of graphs}\label{sec:qi_sums}

In this section, we obtain structural results about quantum isomorphism that will be key for the rest of this article. 
In subsection~\ref{subsec:center}, we prove that quantum isomorphic graphs have quantum isomorphic centers (and slightly more) in Theorem~\ref{thm:qi_center}. 
In subsection~\ref{subsec:qi_sums}, we prove that quantum isomorphic graphs have quantum isomorphic connected components in Theorem~\ref{thm:qi_for_disconnected_graphs}. 
In subsection~\ref{subsec:wreath}, we recall the work of Bichon on wreath products~\cite{Bichon2004} and adapt it to our context.

\subsection{Center}\label{subsec:center}

The aim of this section is Theorem~\ref{thm:qi_center}, which has the immediate consequence that quantum isomorphic graphs have quantum isomorphic centers. 
This should be seen as a generalisation Fulton's similar work on center of trees~\cite{Fulton2006}.

Let $G$ be a graph. 
For $x\in V(G)$, we use the notation:
\[ r_G(x) = \sup_{y\in V(G)} d_G(x,y).\]
We sometimes write $r$ instead of $r_G$ when it is clear from context which graph we are working with. 
Notice that by finiteness of $V(G)$ there is always a vertex $y\in V(G)$ such that $r(x) = d(x,y)$.

The \textit{center} of $G$ is defined as
\[ Z(G) = \{ x\in V(G) \mid r(x) \text{ is minimal}\}.\]
Notice that $Z(G) \neq \varnothing$. 
Also, if $G$ is not connected, then $Z(G) = G$, so this notion presents more interest for connected graphs. 
Even though we define the center of $G$ as a set of vertices, we will often commit the abuse of writing $Z(G)$ both for the set of vertices and for the induced subgraph of $G$ induced by this set of vertices. 

The notion of center is already present in the article of Jordan of 1869~\cite{jordan1869assemblages}. 
Since we are not aware of any earlier definition, we believe the notion might be due to Jordan. 

We start by showing that the function $r$ is preserved by quantum isomorphisms.

\begin{lemma}\label{lem:qi_r}
	Let $G$ and $H$ be two graphs and let $U$ be a quantum isomorphism between them. 
	Let $x\in V(H)$ and $a\in V(G)$. 
	If $u_{xa}\neq 0$, then $r_H(x) = r_G(a)$.
\end{lemma}

\begin{proof}
	Assume that $u_{xa} \neq 0$. 
	Take $y\in V(H)$ such that $d(x,y) = r_H(x)$ and $b \in V(G)$ such that $d(a,b) = r_G(a)$. 
	Now we have $u_{xa} = \sum_{c\in V(G)} u_{xa}u_{yc}$ and $u_{xa} = \sum_{z\in V(H)} u_{xa}u_{zb}$ so there exist $c\in V(G)$ and $z\in V(H)$ such that $u_{xa}u_{yc} \neq 0$ and $u_{xa}u_{zb} \neq 0$. 
	By Lemma~\ref{lem:distance}, this means that
	\[ r_H(x) = d(x,y) = d(a,c) \leq r_G(a)\]
	and
	\[ r_G(a) = d(a,b) = d(x,z) \leq r_H(x),\]
	hence the result.
\end{proof}

\begin{corollary}\label{coro:qi_r}
	Let $G$ and $H$ be two quantum isomorphic graphs. 
	Then there is a bijection $f\colon V(G) \to V(H)$ such that $r_H\circ f = r_G$.
\end{corollary}

\begin{proof}
	Up to a permutation of the vertices, by Lemma~\ref{lem:hall_mu}, we can assume that all diagonal coefficients of $U$ are nonzero. 
	Sending the $i$-th vertex of $G$ to the $i$-th vertex of $H$ gives us a bijection $f\colon V(G) \to V(H)$. 
	By Lemma~\ref{lem:qi_r}, we have that $r_H\circ f = r_G$, as desired.
\end{proof}

\begin{corollary}\label{coro:center}
	Let $G$ and $H$ be two graphs and $U$ a quantum isomorphism from $G$ to $H$. 
	Then $U$ preserves the centers of $G$ and $H$.
\end{corollary}

\begin{proof}
	Let $a\in Z(G)$ and $x\in V(H)\setminus Z(G)$. 
	We want to show that $u_{xa} = 0$. 
	By contradiction, let us assume that $u_{xa}\neq 0$. 
	By Lemma~\ref{lem:qi_r}, we have $r_G(a) = r_H(x)$. 
	Let $y\in Z(H)$. 
	For every $b\in V(G)$, we have:
	\[ r_G(b) \geq r_G(a) = r_H(x) > r_H(y),\]
	so by Lemma~\ref{lem:qi_r} we get $u_{yb} = 0$. 
	Since this is true for every $b\in V(G)$, and since $U$ is a magic unitary, we have $0 = \sum_{b\in V(G)} u_{yb} = 1$: this is absurd. 
	Hence $u_{xa} = 0$, as wanted.

	For the converse, let $a\in V(G)\setminus Z(G)$ and let $x\in Z(H)$. 
	Now $U^*$ is a quantum isomorphism from $H$ to $G$, so by what precedes we have $0 = [U^*]_{xa} = u_{ax}^* = u_{ax}$, hence $u_{ax}=0$. 
	
	This concludes the proof.
\end{proof}

We now obtain our main theorem.

\begin{theorem}\label{thm:qi_center}
	Let $G$ and $H$ be two graphs and let $U$ be a quantum isomorphism between them. 
	Let $I = r_G(V(G)) \subset \N$, and, for $i\in I$, let $C_i = r_G\inv(\{i\})$ and $D_i = r_H\inv(\{i\})$. 
	We let $U_{ij}$ be the block in $U$ indexed by $D_i\times C_j$. 
	We write $U_i = U_{ii}$. 
	For $i$ and $j\in \N$, we claim that:
	\begin{enumerate}
		\item $U_{ij} = 0$ if $i\neq j$,
		\item $\# C_i = \# D_i$,
		\item $U_i$ is a quantum isomorphism from $G[C_i]$ to $H[D_i]$.
	\end{enumerate}
	In particular, letting $i_0 = \min r_G(V(G))$, we have that $U_{i_0}$ is a quantum isomorphism from $Z(G)$ to $Z(H)$.
\end{theorem}

\begin{proof}
	Notice $I = r_H(V(H))$ by Corollary~\ref{coro:qi_r} so the construction of $(D_i)_{i\in I}$ makes sense as well.

	(1) is a direct consequence of Lemma~\ref{lem:qi_r}. 
	This also implies that $U_i$ satisfies the condition of Lemma~\ref{lem:square_mu}, so it is square, and $\# C_i= \# D_i$ for all $i\in I$. 
	This proves (2).

	Finally, let us prove (3). 
	We work in the block-decomposition given by the $C_i$ and $D_i$ for $i\in I$, and we write $A = (A_{ij})_{i,j\in I}$ for the adjacency matrix of $G$ in that decomposition, and similarly we write $B = (B_{ij})_{i,j\in I}$ for the adjacency matrix of $H$. 
	Now by (2) $U$ is diagonal by block in this decomposition, that we write $U = \Diag(U_i)_{i\in I}$, and every diagonal block is a magic unitary. 
	Since $UA = BU$, we immediately obtain that $U_iA_{ii} = B_{ii}U_i$ for all $i\in I$. 
	Now by construction $A_{ii} = \Adj(G[C_i])$ and $B_{ii} = \Adj(H[D_i])$, so this proves (3). 

	The conclusion follows since by definition of the center we have $C_{i_0} = Z(G)$ and $D_{i_0} = Z(H)$.
\end{proof}

\subsection{Quantum isomorphism of disconnected graphs}\label{subsec:qi_sums}

The main goal of this section is Theorem~\ref{thm:qi_for_disconnected_graphs}, namely, that quantum isomorphic graphs have quantum isomorphic connected components. 
On the way, we also prove results that are interesting on their own and will be regularly used later on.

Let $G$ and $H$ be two graphs, with connected components $G_1,\ldots,G_k$ and $H_1,\ldots,H_l$ for some $k$ and $l\in \N$. 
Let $U$ be a quantum isomorphism from $G$ to $H$. 
For simplicity of notation, we will consider that $U$ is indexed by $V(H)\times V(G)$.

For $1\leq i \leq l$, $1\leq j\leq k$, $a\in V(G)$, and $x\in V(H)$, we define:
\begin{itemize}
	\item $p_i(U)(a) = \sum_{y\in V(H_i)} u_{ya}$,
	\item $q_j(U)(x) = \sum_{b\in V(G_j)} u_{xb}$.
\end{itemize}
When no confusion arises, we will feel free to write $p_i(a)$ (resp. $q_j(x)$) instead of $p_i(U)(a)$ (resp. $q_j(U)(x)$).

The case with $k=2=l$ in the next lemma seems to be folklore and can be for instance found as an intermediate step in the proof of Theorem 4.4 of~\cite{LupManRob}. 
We generalise it here for an arbitrary number of connected components. 
It will be key for the main theorem of this section, Theorem~\ref{thm:qi_for_disconnected_graphs}, as well as simplify the proof of Theorem~\ref{thm:wreath_product}.
We have not encountered the present generalisation in the literature.

\begin{lemma}\label{lem:p_ij}
	Let $G$ and $H$ be two graphs, with connected components $G_1,\ldots,G_k$ and $H_1,\ldots,H_l$. 
	Let $U$ be a quantum isomorphism from $G$ to $H$. 
	Fix $1\leq i\leq l$ and $1\leq j\leq k$. 
	The function $a \mapsto p_i(a)$ (resp. $x \mapsto q_j(x)$) is constant on the connected components of $G$ (resp. $H$). 
	Moreover, writing $p_{ij}$ (resp. $q_{ji}$) for $p_i(a)$ for any $a\in V(G_j)$ (resp. for $q_j(x)$ for any $x\in V(H_i)$), the following are true:
	\begin{enumerate}
		\item $p_{ij}$ and $q_{ij}$ are self-adjoint projections,
		\item $p_{ij} = q_{ji}$,
		\item $\sum_{s=1}^k p_{is} = 1 = \sum_{s=1}^l p_{sj}$.
	\end{enumerate}
	In particular, $k=l$ and the matrix $(p_{ij})_{1\leq i,j\leq k}$ is a magic unitary.
\end{lemma}

\begin{proof}
	Let us start by proving that $p_i$ is constant on connected components of $G$. 
	Let $a$, $b\in V(G)$ be in the same connected component. 
	It is clear that $p_i(a)$ and $p_i(b)$ are self-adjoint projections. 
	We claim that $p_i(a) = p_i(a)p_i(b)$. 
	It comes:
	\begin{align*}	
		p_i(a)(1-p_i(b)) &= \left( \sum_{x \in V(H_i)} u_{xa} \right) \left( \sum_{x \in V(H)} u_{xb} - \sum_{x \in V(H_i)} u_{xb} \right)\\
		&= \left( \sum_{x \in V(H_i)} u_{xa} \right) \left( \sum_{x \in V(H\setminus H_i)} u_{xb} \right)\\
		&= \sum_{\substack{x\in V(H_i) \\ y \in V(H\setminus H_i)}} u_{xa}u_{yb}\\
		&=0
	\end{align*}
	by Lemma~\ref{lem:distance}, since $d(x,y) < +\infty = d(a,b)$. 
	Thus $p_i(x) = p_i(x)p_i(y)$. 
	
	Applying what precedes to $y$ and $x$, we obtain $p_i(y) = p_i(y)p_i(x)$. 
	It comes:
	\[ p_i(y) = p_i(y)^* = p_i(x)p_i(y) = p_i(x),\]
	as desired. 
	This concludes the proof that $p_{ij}$ is a well-defined self-adjoint projection.

	Notice that $U^*$ is a quantum isomorphism from $H$ to $G$. 
	For $x \in V(H)$, we have
	\[ p_j(U^*)(x) = \sum_{c \in V(G_j)} [U^*]_{cx} = \sum_{c \in V(G_j)} u_{xc} = q_j(U)(x).\]
	Hence applying what precedes to $U^*$ shows that $q_{ji}$ is a well-defined self-adjoint projection. 
	This concludes the proof of (1).

	Fix $a \in V(G_j)$ and $x \in V(H_i)$. 
	It comes (we omit $U$ again):
	\begin{align*}
		p_{ij}(1-q_{ji}) &= p_i(a)(1- q_j(x))\\
		&= \left( \sum_{y\in V(H_i)} u_{ya} \right) \left( \sum_{b \in V(G)} u_{xb} - \sum_{b\in V(G_j)} u_{xb} \right)\\
		&= \left( \sum_{y\in V(H_i)} u_{ya} \right) \left( \sum_{b\in V(G\setminus G_j)} u_{xb} \right)\\
		&= \sum_{\substack{y\in V(H_i) \\ b\in V(G\setminus G_j)}} u_{ya}u_{xb}\\
		&= 0
	\end{align*}	
	by Lemma~\ref{lem:distance} since $d(y,x) < +\infty = d(a,b)$. 
	So $p_{ij} = p_{ij}q_{ji}$.

	Now, notice that $q_{ji}(U) = q_j(U)(x) = p_j(U^*)(x) = p_{ji}(U^*)$. 
	Combining everything, it comes:
	\[ q_{ji}(U) = p_{ji}(U^*) = p_{ji}(U^*)^* = (p_{ji}(U^*)q_{ij}(U^*))^* = q_{ij}(U^*)p_{ji}(U^*) = p_{ij}(U)q_{ji}(U) = p_{ij}(U).\]
	This concludes the proof of (2).

	Assertion (3) is immediate and the conclusion follows by Lemma~\ref{lem:square_mu}.
\end{proof}

We use the following notation: for $Y \subset V(H)$ and $X \subset V(G)$, we denote by $U[Y,X]$ the submatrix of $U$ indexed by $Y\times X$. 
Given a graph $L$ and a vertex $x \in V(L)$, we write $C_L(x)$, or simply $C(x)$ when the context is clear, for the connected component of $x$ in $L$.

\begin{theorem}\label{thm:qi_for_connected_components}
	Let $G$ and $H$ be two graphs and $U$ be a quantum isomorphism from $G$ to $H$ with coefficients in a unital $C^*$-algebra $X$. 
	Let $x \in V(H)$ and $a \in V(G)$ such that $u_{xa} \neq 0$. 
	Set $V = U[C_H(x),C_G(a)]$. 
	Then there exists a nonzero projection $p\in X$ such that:
	\begin{enumerate}
		\item for every $y\in C(x)$, we have $\sum_{b\in V(C(a))} v_{yb} = p$,
		\item for every $b\in C(a)$, we have $\sum_{y\in V(C(x))} v_{yb} = p$,
		\item $V\Adj(C_G(a)) = \Adj(C_H(x))V$.
	\end{enumerate}
	In particular, $V$ is a square matrix, and it is a quantum isomorphism from $C_G(a)$ to $C_H(x)$ with coefficients in the $C^*$-subalgebra of $X$ generated by the coefficients of $V$.
\end{theorem}

\begin{proof}
	The first two assertions follow immediately from Lemma~\ref{lem:p_ij}. 
	This also implies that $V$ is a square matrix by Lemma~\ref{lem:square_mu}. 
	It remains to prove that $V\Adj(C(a)) = \Adj(C(x)) V$. 
	Let us fix an enumeration of $V(G)$ finishing by the vertices of $C(a)$ and an enumeration of $V(H)$ finishing by the vertices of $C(x)$. 
	In block form, $U$ is of the following form:
	\[ U = \begin{pmatrix}
		V_1 & V_2\\
		V_3 & V
	\end{pmatrix}.
	\]
	We similarly write
	\[ A = \Adj(G) = \begin{pmatrix}
		A_1 & A_2\\
		A_3 & A_4
	\end{pmatrix},
	\]
	with $A_4$ indexed by $C(a)\times C(a)$, and
	\[ B = \Adj(H) = \begin{pmatrix}
		B_1 & B_2\\
		B_3 & B_4
	\end{pmatrix},
	\]
	with $B_4$ indexed by $C(x)\times C(x)$. 
	Now, since $UA=BU$, it comes:
	\begin{align*}
		0 &= [UA-BU]_{22}\\
		&= [UA]_{22} - [BU]_{22}\\
		&= V_3 A_2 + V A_4 - B_3 V_2 - B_4 V.
	\end{align*}
	Notice that if $q$ is a coefficient of $V_3$ or $V_2$, then $pq = 0 = qp$ by orthogonality of the rows of $U$, and if $q$ is a coefficient of $V$, then $q\leq p$. 
	This implies that $pV_3 = 0$, $V_2p = 0$, and $pVp = V$. 
	Moreover, the coefficients of $A_4$ and $B_4$ are scalar, so they commute with $p$. 
	Multiplying by $p$ on both sides, we get
	\[ 0 = pV A_4 p - pB_4 Vp = VA_4 - B_4V,\]
	as desired. 
	This concludes the proof of assertion (3).

	Finally, let $Y$ be the $C^*$-subalgebra of $X$ generated by the coefficients of $V$. 
	Notice $p\in Y$ by (1), and $p$ is a unit for $Y$ (in particular, $Y\neq 0$ since $p\neq 0$). 
	Now (with a slight abuse) we consider $V$ to be with coefficients in the unital $C^*$-algebra $Y$, and it is a magic unitary by (1) and (2), and a quantum isomoprhism from $C(x)$ to $C(a)$ by (3), as desired. 
	This concludes the proof.
\end{proof}

In particular, two vertices in the same quantum orbit have quantum isomorphic connected components.

The next theorem follows from the precedent one through an \textit{ad hoc} application of Hall's marriage theorem through Lemma~\ref{lem:hall_mu}.

\begin{theorem}\label{thm:qi_for_disconnected_graphs}
	Two graphs $G = G_1 + \ldots + G_k$ and $H = H_1 + \ldots + H_l$, written as sums of their connected components, are quantum isomorphic if and only if $k=l$ and there exists a permutation $\sigma \in S_k$ such that for all $1\leq i\leq k$ the graph $G_i$ is quantum isomorphic to $H_{\sigma(i)}$.
\end{theorem}

\begin{proof}
	The reverse implication is immediate (consider the diagonal matrix $\Diag(U_1,\ldots,U_l)$ where $U_i$ is a quantum isomorphism from $G_i$ to $H_{\sigma(i)}$). 
	Let us prove the forward implication.

	Assume that $G$ and $H$ are quantum isomorphic and let $U$ be a quantum isomorphism from $H$ to $G$. 
	By Lemma~\ref{lem:p_ij}, we obtain that $k=l$ and a magic unitary $P = (p_{ij}(U))_{1\leq i,j\leq k}$. 
	Applying Lemma~\ref{lem:hall_mu}, we obtain a permutation $\sigma\in S_k$ such that $p_{i\sigma(i)}\neq 0$ for all $1\leq i\leq k$. 
	So for $1\leq i\leq k$ we have by Theorem~\ref{thm:qi_for_connected_components} that $G_i$ is quantum isomorphic to $H_{\sigma(i)}$, as desired. 
	This concludes the proof.
\end{proof}

\subsection{Wreath product}\label{subsec:wreath}

The construction of the wreath product of a compact quantum group with $S_n^+$ is due to Bichon (Section 3 of~\cite{Bichon2004}), who then proceeds to show that $\Qu(d.G) = \Qu(G) \wr S_d^+$ when $G$ is a connected graph (Section 4 of~\cite{Bichon2004}). 
However, at the time of the writing of~\cite{Bichon2004}, the theory was still at a very early stage. 
In particular, the definition of the quantum automorphism group of a graph was slightly different. 
For the sake of coherence and completeness, we include a proof of his result here. 
This proof essentially follows along the same lines as the original one, and is slightly facilitated by our preparation and the modern terminology.

A detailed description of the wreath product in the case of matrix compact quantum groups can be found in Section 7.2.2 of~\cite{FreslonBook}. 
In particular, when $A$ is a matrix compact quantum group of size $m\geq 1$, one can obtain the fundamental representation of $A \wr S_n^+$ for some $n\geq 1$ as a matrix compact quantum group of size $mn$, by giving the fundamental representation. 
We believe the formula for the generators of the fundamental representation to be due to Freslon and we recall it here since it will be particularly useful.

Let $m$ and $n\geq 1$ be integers. 
Consider $X = (A,u)$ a quantum permuation group of size $m$ and write $S_n^+ = (S,p)$. 
The \textit{wreath product} of $X$ by $S_n^+$, denoted by $X \wr S_n^+$, is defined as the quantum permutation group with algebra $W = A^{*n}*S / \{ [\nu_i(a),p_{ij}]\}$ and fundamental representation of size $mn$ given in block form by $w = (w_{ij,kl})_{1\leq i,j\leq n,\ 1\leq k,l\leq m}$ with $w_{ij,kl} = \nu_i(u_{kl})p_{ij}$. 
Here $*$ stands for the unital fre product of $*$-algebras and $\nu_i \colon A\to A^{*n}$ is the $i$-th canonical inclusion. 

We can now state and prove Bichon's theorem with the modern definition of $\Qu(G)$.

\begin{theorem}\label{thm:wreath_product}
	Let $d\geq 1$ and $G$ be a connected graph. 
	Then $\Qu(d.G) = \Qu(G) \wr S_d^+$.
\end{theorem}

\begin{proof}
	This proof contains heavy notation, so let us fix it.

	Let us write $H=d.G$ and label its connected components $G_1,\ldots,G_d$, where each component is isomorphic to $G$. 
	Let $V$ be the fundamental representation of $\Qu(d.G)$. 
	Up to fixing an enumeration of $V(G)$, we can assume that $V(G) = \{1,\ldots,n\}$ for some $n\in \N$, and, similarly, we set $V(G_i) = \{(i,1),\ldots,(i,n)\}$ for $1\leq i\leq d$. 
	For $1\leq i,j\leq d$ and $1\leq k,l\leq n$, we denote by $v_{ij,kl}$ the coefficient of $V$ indexed by the pair $((i,k),(j,l))$. 
	We think of $v_{ij,kl}$ as ``the coefficient correponding to the $k$-th vertex in $G_i$ and the $l$-th vertex of $G_j$.'' 
	This allows us to obtain the block form $V = (V_{ij})_{1\leq i,j\leq d}$, where $V_{ij} = (v_{ij,kl})_{1\leq k,l\leq n}$ is $V[V(G_i),V(G_j)]$. 
	Notice that since $V$ is a magic unitary, we have by Lemma~\ref{lem:p_ij} the magic unitary $P(V) = (p_{ij}(V))_{1\leq i,j\leq d}$, where, for any $1\leq l\leq n$, we have $p_{ij}(V) = \sum_{x \in V(G_i)} v_{x,(j,l)} = \sum_{k=1}^n v_{ij,kl}$. 

	We denote the fundamental representation of $\Qu(G)$ by $U = (u_{kl})_{1\leq k,l\leq n}$ and of $S_d^+$ by $S=(s_{ij})_{1\leq i,j\leq d}$. 
	We also write $A = \Adj(G)$ and $B = \Adj(H)$. 
	Notice that $B = \Diag(A,\ldots,A) = I_d\otimes A$. 

	To prove the result, it is sufficient to exhibit an isomorphism of $*$-algebras $\varphi \colon \Oo(\Qu(d.G)) \to \Oo(\Qu(G) \wr S_d^+)$ such that $\varphi$ commutes with the coproducts on algebraic tensors. 
	We start by exhibiting the desired isomorphism. 

	Write $\Aa = \Oo(\Qu(d.G))$ and $\Bb = \Oo(\Qu(G) \wr S_d^+)$. 
	Modulo a slight abuse justified, we consider $\Oo(S_d^+)$ as a embedded in $\Bb$. 
	Recall that $\Bb = \Oo(S_d^+) * \Oo(\Qu(G))^{*d} / \Ii$ for some ideal $\Ii$, and is equipped with insertion maps $\nu_i \colon \Oo(\Qu(G)) \to \Bb$ for $1\leq i\leq d$. 
	For $1\leq i,j\leq d$ and $1\leq k,l\leq n$, let $w_{ij,kl} = s_{ij}\nu_i(u_{kl})$, where we recall that $\nu_i \colon \Oo(\Qu(G)) \to \Oo(\Qu(G) \wr S_d^+)$ is the insertion in the $i$-th coordinate. 
	Notice that this gives us a matrix $W \in \Mm_{dn}(\Bb)$. 
	Moreover, since $s_{ij}$ and $\nu_i(u_{kl})$ commute, we have that $w_{ij,kl}$ is a self-adjoint projection. 
	Fixing $1\leq i\leq d$ and $1\leq k\leq n$, it comes:
	\[ \sum_{j=1}^d \sum_{l=1}^n w_{ij,kl} = \sum_{j=1}^d \sum_{l=1}^n s_{ij}\nu_i(u_{kl}) = \left(\sum_{j=1}^d s_{ij}\right) \nu_i\left(\sum_{l=1}^n u_{kl}\right) = 1,\]
	as well as:
	\[ \sum_{j=1}^d \sum_{l=1}^n w_{ji,lk} = \sum_{j=1}^d \sum_{l=1}^n s_{ji} \nu_j(u_{lk}) = \sum_{j=1}^d s_{ji} \nu_j\left(\sum_{l=1}^n u_{lk}\right) = \sum_{j=1}^d s_{ij} = 1.\]
	This shows that $W$ is a magic unitary. 
	Finally, noticing that it is of the same dimension as $B$, and writing it in the natural block form, we notice that $W$ commutes with $B$ if and only if $W_{ij} = (w_{ij,kl})_{1\leq k,l\leq n}$ commutes with $A$ for every $1\leq i,j\leq d$. 
	So, given $1\leq i,j \leq d$ and using the fact that the coefficients of $A$ are scalar, it comes:
	\begin{align*}
		W_{ij}A &= s_{ij}\nu_i(U)A = s_{ij} \nu_i(UA) = s_{ij}\nu_i(AU)\ \text{(since $U$ is adapted to $G$)}\\
		&= s_{ij} A\nu_i(U) = As_{ij}\nu_i(U) = AW_{ij},
	\end{align*}
	as desired. 
	Thus $W \in \Mm_{nd}(\Bb)$ is a magic unitary adapted to $d.G$.

	By the universal property of $\Oo(\Qu(d.G))$, there is a unital $*$-morphism $\varphi \colon \Aa \to \Bb$ such that $\varphi(v_{ij,kl}) = s_{ij}\nu_i(u_{kl})$ for $1\leq i,j\leq d$ and $1\leq k,l\leq n$. 

	Let us exhibit an inverse to $\varphi$. 
	Since, by Lemma~\ref{lem:p_ij}, the matrix $P(V) = (p_{ij}(V))_{1\leq i,j\leq n}$ is a magic unitary, by the universal property of $\Oo(S_d^+)$, there is a unital $*$-morphism $f_0 \colon \Oo(S_d^+) \to \Aa$ such that $f_0(s_{ij}) = p_{ij}(V)$ for all $1\leq i,j\leq d$. 
	Now, let us fix $1\leq i\leq d$. 
	For $1\leq k,l\leq n$, we set $z_{i,kl} = \sum_{r=1}^d v_{ir,kl}$, it is clearly a self-adjoint projection. 
	We define $Z_i = (z_{i,kl})_{1\leq k,l\leq n} \in \Mm_n(\Aa)$. 
	We claim that $Z_i$ is a magic unitary commuting with $A$. 
	For $1\leq k \leq n$, it comes:
	\[ \sum_{l=1}^n z_{i,kl} = \sum_{l=1}^n \sum_{r=1}^d v_{ir,kl} = 1,\]
	and:
	\[ \sum_{l=1}^n z_{i,lk} = \sum_{l=1}^n \sum_{r=1}^d v_{ir,lk} = \sum_{r=1}^d p_{ir}(V) = 1.\]
	Finally, since $V$ commutes with $B = \Diag(A,\ldots,A)$, we have that the block $V_{ij} = (v_{ij,kl})_{1\leq k,l\leq n}$ commutes with $A$ for any $1\leq j\leq d$. 
	Given $1\leq k,l\leq n$, we reach:
	\begin{align*}
		[Z_i A]_{kl} &= \sum_{t=1}^n [Z_i]_{kt} a_{tl} = \sum_{t=1}^n \sum_{r=1}^n v_{ir,kt}a_{tl} = \sum_{r=1}^d [V_{ir}A]_{kl} = \sum_{r=1}^d [AV_{ir}]_{kl}\\
		&= \sum_{r=1}^d \sum_{t=1}^n a_{kt}v_{ir,tl} = \sum_{t=1}^n a_{kt} [Z_i]_{tl} = [AZ_i]_{kl},
	\end{align*}
	as desired. 
	This shows that $Z_i \in \Mm_n(\Aa)$ is a magic unitary adapted to $G$. 
	By the universal property of $\Oo(\Qu(G))$, there is a unital $*$-morphism $f_i \colon \Oo(\Qu(G)) \to \Aa$ such that $f_i(u_{kl}) = \sum_{r=1}^d v_{ir,kl}$ for any $1\leq k,l\leq n$.

	We are now given unital $*$-morphisms $f_0 \colon \Oo(S_d^+) \to \Aa$ and $f_1,\ldots,f_d \colon \Oo(\Qu(G)) \to \Aa$. 
	By universal property of the coproduct, that is, the unital free product, we have a unital $*$-morphism $f \colon \Oo(S_d^+)*\Oo(\Qu(G))^{*d} \to \Aa$ such that, modulo slight abuses of notation, the following are true for any $1\leq i,j\leq d$ and any $1\leq k,l\leq n$:
	\begin{itemize}
		\item $f(s_{ij}) = p_{ij}(V)$,
		\item $f(\nu_i(u_{kl})) = z_{i,kl}$.
	\end{itemize}
	We claim that for any $1\leq i\leq d$ and any $1\leq k,l\leq n$, we have $p_{ij}(V)z_{i,kl} = v_{ij,kl} = z_{i,kl}p_{ij}(V)$. 
	Indeed, by Lemma~\ref{lem:p_ij}, we have
	\[ p_{ij}(V) = q_{ji}(V) = q_j(V)((i,k)) = \sum_{t=1}^n v_{ij,kt}.\]
	Hence, it comes:
	\begin{align*}
		z_{i,kl}p_{ij}(V) &= \left( \sum_{r=1}^d v_{ir,kl} \right) \left( \sum_{t=1}^n v_{ij,kt} \right)\\
		&= \sum_{\substack{1\leq r\leq d\\1\leq t\leq n}} v_{ir,kl}v_{ij,kt}\\
		&= v_{ij,kl}
	\end{align*}
	since $V$ is a magic unitary, as well as
	\begin{align*}
		p_{ij}(V)z_{i,kl} &= \left( \sum_{t=1}^n v_{ij,kt} \right)\left( \sum_{r=1}^d v_{ir,kl} \right)\\
		&= \sum_{\substack{1\leq r\leq d\\1\leq t\leq n}} v_{ij,kt}v_{ir,kl}\\
		&= v_{ij,kl}.
	\end{align*}
	By definition of $\Oo(\Qu(G)\wr S_d^+)$, this implies that we have a unital $*$-morphism $\psi \colon \Bb \to \Aa$ such that $f = \psi \circ \pi$, where $\pi \colon \Oo(S_d^+)*\Oo(\Qu(G))^{*d} \to \Bb$ is the canonical quotient map. 
	In particular, $\psi$ is a unital $*$-morphism from $\Bb$ to $\Aa$ such that $\psi(s_{ij}) = p_{ij}(V)$ and $\psi(\nu_i(u_{kl})) = z_{i,kl}$ for any $1\leq i,j\leq d$ and $1\leq k,l\leq n$.

	Let us check that $\psi$ and $\varphi$ are inverse to each other. 
	Fix $1\leq i,j\leq d$ and $1\leq k,l\leq n$. 
	It comes:
	\[ (\psi \circ \varphi)(v_{ij,kl}) = \psi(s_{ij}\nu_i(u_{kl})) = p_{ij}(V)z_{i,kl} = v_{ij,kl}\]
	by the above computation. 
	Conversely, we have:
	\[ (\varphi \circ \psi)(s_{ij}) = \varphi(p_{ij}(V)) = \sum_{t=1}^n \varphi(v_{ij,tl}) = \sum_{t=1} s_{ij}\nu_i(u_{tl}) = s_{ij}\nu_i\left(\sum_{t=1}^n u_{tl}\right) = s_{ij},\]
	as well as:
	\[ (\varphi\circ \psi)(\nu_i(u_{kl})) = \varphi(z_{i,kl}) = \sum_{r=1}^d \varphi(v_{ir,kl}) = \sum_{r=1}^d s_{ir}\nu_i(u_{kl}) = \nu_i(u_{kl}).\]
	Thus $\varphi$ and $\psi$ are two $*$-isomorphisms inverse to each other, and $\Aa$ and $\Bb$ are isomorphic $C^*$-algebras.

	In order to conclude the proof, we only need to check that $\varphi$ preserves the comultiplication. 
	For $1\leq i,j\leq d$ and $1\leq k,l\leq n$, it comes:
	\begin{align*}
		\Delta(\varphi(v_{ij,kl})) &= \Delta(s_{ij}\nu_i(u_{kl}))\\
		&= \Delta(s_{ij})\Delta(\nu_i(u_{kl}))\\
		&= \left( \sum_{t=1}^d s_{it}\otimes s_{tj} \right) \left( \sum_{t=1}^d (\nu_i\otimes \nu_t)(\Delta(u_{kl}))(s_{it}\otimes 1) \right)\\
		&= \left( \sum_{t=1}^d s_{it}\otimes s_{tj} \right) \left( \sum_{t=1}^d (\nu_i\otimes \nu_t) \left(\sum_{r=1}^n u_{kr}\otimes u_{rl}\right)(s_{it}\otimes 1) \right)\\
		&= \sum_{1\leq t_1,t_2 \leq d} \sum_{r=1}^n (s_{it_1}\otimes s_{t_1j}) (\nu_i(u_{kr})\otimes \nu_{t_2}(u_{rl})) (s_{it_2}\otimes 1)\\
		&= \sum_{1\leq t_1,t_2 \leq d} \sum_{r=1}^n (s_{it_1}\nu_i(u_{kr}) s_{it_2}) \otimes (s_{t_1j} \nu_{t_2}(u_{rl}))\\
		&= \sum_{1\leq t_1,t_2 \leq d} \sum_{r=1}^n (s_{it_1}s_{it_2} \nu_i(u_{kr})) \otimes (s_{t_1j} \nu_{t_2}(u_{rl}))\ \text{since $\nu_i(u_{kr})$ and $s_{it_2}$ commute}\\
		&= \sum_{t=1}^d \sum_{r=1}^n (s_{it}\nu_i(u_{kr})) \otimes (s_{tj}\nu_t(u_{rl}))\ \text{since $s_{it_1}s_{it_2} = 0$ if $t_1\neq t_2$}\\
		&= \sum_{t=1}^d \sum_{r=1}^n \varphi(v_{it,kr})\otimes \varphi(v_{tj,rl})\\
		&= (\varphi \otimes \varphi) \left( \sum_{t=1}^d \sum_{r=1}^n v_{it,kr}\otimes v_{tj,rl} \right)\\
		&= (\varphi \otimes \varphi) (\Delta(v_{ij,kl})),
	\end{align*}
	as desired.

	Thus, $\varphi \colon \Qu(d.G) \to \Qu(G)\wr S_d^+$ is an isomorphism of quantum groups.
\end{proof}

\section{Cographs}\label{sec:cographs}

We thank Pournajafi for introducing the class of cographs to the author, and suggesting them as natural candidates. 
Indeed, the class of cographs is the smallest class of graphs containing $K_1$ and closed under sums and complements, so the results about sums of graphs of the precedent section naturally apply.

In this section, we establish the Schmidt alternative for the class of cographs, as well as give a complete description of cographs with quantum symmetry. 
We conclude by proving that quantum isomorphic cographs are isomorphic. 
Altogether, this shows that cographs are a tractable class, and are thus a first example of an infinite tractable class. 
The computation of their quantum automorphism groups is left to the next section.

\subsection{Definition and symmetries}

The class of \textit{cographs} is the smallest class of graphs (for inclusion) containing $K_1$ and being stable by finite sums and complement. 
It is a class of graphs with a well-understood structure, and, in particular, multiple equivalent characterisations. 
We will only use the following ones, that we recall without proof.

\begin{theorem}\label{thm:cographs}
	Let $G$ be a graph. 
	The following are equivalent:
	\begin{enumerate}
		\item $G$ is a cograph,
		\item $G=K_1$ or there are two cographs $G_1$, $G_2$ such that $G= G_1 + G_2$ or $G^c = G_1 + G_2$,
		\item all induced subgraphs of $G$ on at least two vertices have twin vertices,
		\item $G$ does not contain $P_4$ as an induced subgraph.
	\end{enumerate}
\end{theorem}

In particular, cographs form a hereditary class of graphs.

\begin{corollary}\label{coro:cographs}
	An induced subgraph of a cograph is a cograph. 
	In particular, a connected component of a cograph is a cograph.
\end{corollary}

\begin{proof}
	This follows immediately from (4) of Theorem~\ref{thm:cographs}.
\end{proof}

\begin{lemma}\label{lem:symmetric_cographs}
	A cograph on at least two vertices admits a nontrivial automorphism.
\end{lemma}

\begin{proof}
	Let $G$ be a cograph on at least two vertices. 
	By Theorem~\ref{thm:cographs}, $G$ admits two twin vertices $x$ and $y$. 
	We claim that the function $\sigma \colon V(G) \to V(G)$ which exchanges $x$ and $y$ and fixes all of the other vertices is a nontrivial automorphism of $G$. 
	It is clearly a well-defined nontrivial transposition of $V(G)$, so we only need to check that it is a graph morphism.  
	Notice that on $V(G) \setminus \{x,y\}$, it coincides with the identity. 
	So it is enough to check that $\sigma$ preserves edges which admit at least one endpoint in $\{x,y\}$. 

	Let $z\in V(G)$. 
	Because $x$ and $y$ play symmetric roles, it is enough to check that if $xz\in E(G)$, then $\sigma(x)\sigma(z) \in E(G)$. 
	Thus, let us assume that $xz\in E(G)$. 
	If $z = y$ we have $\sigma(x)\sigma(y) = yx = zx \in E(G)$ as desired. 
	Otherwise, we have $z\neq y$, so $z\in N(x)\setminus \{y\}$. 
	Because $x$ and $y$ are twin, we have $z \in N(y) \setminus \{x\}$. 
	Hence it comes
	\[ \sigma(x)\sigma(z) = yz \in E(G),\]
	also as desired. 
	This concludes the proof.
\end{proof}

This has the following corollary.

\begin{corollary}\label{coro:symmetric_cographs}
	A cograph is quantum asymmetric if and only if it is asymmetric, in which case it is $K_1$.
\end{corollary}

\begin{proof}
	We already know that a quantum asymmetric graph is asymmetric, so it is in particular true for a cograph. 
	Conversely, if $G$ is an asymmetric cograph, then by Lemma~\ref{lem:symmetric_cographs} we have $G=K_1$, and it is clearly quantum asymmetric. 
	This concludes the proof.
\end{proof}

\subsection{The Schmidt alternative for cographs}

Let us now prove the Schmidt alternative for cographs.

We start by defining useful sequences of graphs. 
Consider $G$ a graph. 
We define inductively the following sequence $(Z_n(G))_{n\geq 1}$: we set $Z_1(G) = G$ and for $n\in \N$ we set
\[Z_{n+1}(G) = K_1 + Z_n(G)^c.\]

\begin{lemma}\label{lem:Zsequences}
	Let $G$ be a graph. 
	For $n\geq 2$, we claim that:
	\begin{enumerate}
		\item $Z_n(G)$ is disconnected,
		\item if $G$ is a cograph, then $Z_n(G)$ is a cograph,
		\item $\# V(Z_n(G)) = n-1 + \# V(G)$,
		\item if $n\geq 3$, writing $x$ the vertex of $K_1$ in $Z_n(G) = K_1 + Z_{n-1}(G)^c$, we have that $x$ is the unique vertex of degree 0 of $Z_n(G)$,
		\item if $Z_2(G)$ does not satisfy Schmidt criterion, then $Z_n(G)$ neither,
		\item if $Z_2(G)$ does not have quantum symmetry, then $Z_n(G)$ neither.
	\end{enumerate}
\end{lemma}

\begin{proof}
	(1), (2), and (3) are clear. 
	For (4), we have $deg(x) = 0$ by construction. 
	Moreover, we have $n-1\geq 2$, so, by (1), $Z_{n-1}(G)$ is disconnected. 
	Then, $Z_{n-1}(G)^c$ is connected by Lemma~\ref{lem:complement}. 
	Because $n-1\geq 2$, by (3), we have $\# V(Z_{n-1}(G)) = n-2 + \# V(G) \geq n-1 \geq 2$, so $Z_{n-1}(G)^c$ is connected graph on at least two vertices: this implies that all of its vertices have degree at least 1, which proves (4).\\

	Let us prove (5). 
	Assume that $Z_2(G)$ does not satisfy Schmidt criterion and let us prove that $Z_n(G)$ neither by induction on $n\geq 2$. 
	For $n=2$, there is nothing to prove. 
	So let us assume that $Z_n(G)$ does not satisfy Schmidt criterion for some $n\geq 2$ and consider $f$ and $g \in \Aut(Z_{n+1}(G))$ two non-trivial automorphisms of $Z_{n+1}(G)$. 
	Notice that $n+1\geq 3$ so by (4) $x$ is the unique vertex of degree 0, where $\{x\} = V(K_1)$ in the decomposition $Z_{n+1} = K_1 + Z_n(G)^c$. 
	Because $f$ and $g$ preserve degrees, we have $f(x) = x = g(x)$. 
	Hence $f$ and $g$ preserve $Z_n(G)^c$ and their restrictions $\tilde f$ and $\tilde g$ are automorphisms of $Z_n(G)^c$. 
	Moreover, they are non-trivial, as otherwise they would be the identity on $Z_{n+1}(G)$. 
	By induction hypothesis, we have that $Z_n(G)$ does not satisfy Schmidt criterion, so by remark~\ref{rk:auto_complement}, $Z_n(G)^c$ neither. 
	Then $\varnothing \neq \Supp(\tilde f) \cap \Supp(\tilde g) \subset \Supp(f) \cap \Supp(g)$, so $f$ and $g$ do not have disjoint support. 
	This being true for any two non-trivial automorphisms of $Z_{n+1}(G)$, we conclude that $Z_{n+1}(G)$ does not satisfy Schmidt criterion. 
	This concludes the induction and the proof of (5).\\

	Finally, let us prove (6) in the same way. 
	We assume that $Z_2(G)$ does not have quantum symmetry and we prove that $Z_n(G)$ neither by induction on $n\geq 2$. 
	For $n=2$, there is once again nothing to prove. 
	So we assume that $Z_n(G)$ does not have quantum symmetry for some $n\geq 2$ and we want to prove that $Z_{n+1}(G)$ neither. 
	Once again, we have $n+1\geq 3$ so by (4) $x$ is the unique vertex of degree 0. 
	We enumerate the vertices of $Z_{n+1}(G)$, denoting them by $x_1,\ldots,x_m$, with $x_m=x$, where $m = \# V(Z_{n+1}(G))$. 
	Let $U$ be a magic unitary adapted to $Z_{n+1}(G)$ with this enumeration. 
	By Lemma~\ref{lem:degree} and (4), we have that $u_{m,m} = 1$ and all other coefficients in the last row and column are 0. 
	So we can write $U$ in blocks in the following way:
	\[ U = \begin{pmatrix}
		& & & 0\\
		& V & & \vdots\\
		& & & 0\\
		0 & \ldots & 0 & 1
	\end{pmatrix},\]
	with $V$ being a block of dimension $m-1$. 
	Notice that $V$ is a magic unitary. 
	Now with the same enumeration, writing $A = \Adj(Z_{n+1}(G))$, we have
	\[ A = \begin{pmatrix}
		& & & 0\\
		& B & & \vdots\\
		& & & 0\\
		0 & \ldots & 0 & 0
	\end{pmatrix},\]
	with $B = \Adj(Z_n(G)^c)$. 
	So both $U$ and $A$ are block-diagonal with blocks of corresponding dimensions $n$ and 1: the fact that $UA = AU$ implies that $VB = BV$. 
	This shows that $V$ is a magic unitary adapted to $Z_n(G)^c$. 
	Thus $V$ is also adapted to $Z_n(G)$ by Lemma~\ref{lem:complement}. 
	Now, by induction hypothesis, we have that $Z_n(G)$ does not have quantum symmetry, so all of the coefficients of $V$ commute. 
	But the remaining coefficients of $U$ are 0 and 1, so all of the coefficients of $U$ commute too. 
	This being true for any magic unitary adapted to $Z_{n+1}(G)$, we conclude that $Z_{n+1}(G)$ does not have quantum symmetry. 
	It ends the induction, the proof of (6), and the proof of the lemma.
\end{proof}

We define two sequences of cographs by $X_n = Z_n(K_1)$ and $Y_{n+1} = Z_n(K_2)$ for $n\geq 1$ (notice that $K_2$ is indeed a cograph). 
The shift of indices for $(Y_n)_{n\geq 2}$ is simply to ensure that $\# V(Y_n) = n$.

We can now prove the main result of this section.

\begin{theorem}\label{thm:cographs_QS}
	Let $G$ be a cograph. 
	Then the following are equivalent:
	\begin{enumerate}
		\item $G$ does not have quantum symmetry,
		\item $G$ does not satisfy Schmidt criterion,
		\item $G=K_1$ or letting $n=\# V(G)$ we have $G \in \{X_n,X_n^c,Y_n,Y_n^c\}$.
	\end{enumerate}
\end{theorem}

\begin{proof}
	The implication (1) $\Rightarrow$ (2) is the contrapositive of Schmidt criterion, which we know is true.\\

	Let us check that (3) $\Rightarrow$ (1). 
	$K_1$ does not have quantum symmetry. 
	Together with Lemma~\ref{lem:complement}, this shows that we only need to check that $X_n$ and $Y_n$ do not have quantum symmetry for all $n\geq 2$.
	
	For $(X_n)$, we have that $X_2 = Z_2(K_1) = 2K_1$ does not have quantum symmetry. 
	Applying Lemma~\ref{lem:Zsequences}, we see that, for all $n\geq 2$, the graph $X_n = Z_n(K_1)$ does not have quantum symmetry.

	For $(Y_n)$, we have that $Y_2 = Z_1(K_2) = K_2$ does not have quantum symmetry. 
	Moreover, remark~\ref{rem:qu_G_abelian} ensures that $Y_3 = K_1 + Y_2^c = K_1 + K_2^c = 3K_1$ neither. 
	So $Z_2(K_2) = Y_3$ does not have quantum symmetry, which, by Lemma~\ref{lem:Zsequences}, implies that $Y_{n+1} = Z_n(K_2)$ neither, for all $n\geq 2$. 
	This concludes the proof that (3) $\Rightarrow$ (1).\\

	Finally, let us prove that (2) $\Rightarrow$ (3). 
	For $n\geq 2$, we write $R_n = \{X_n,X_n^c,Y_n,Y_n^c\}$. 
	Let us start by checking that the implication is true if $G$ has at most two vertices. 
	If $\# V(G)=1$, then $G=K_1$, so (3) is satisfied, and also (2) $\Rightarrow$ (3). 
	If $\# V(G) =2$, then
	\[G \in \{K_2,2K_1\} = \{Y_2,Y_2^c\} \subset R_2,\]
	so (3) is satisfied, and also (2) $\Rightarrow$ (3).

	Let us now prove that (2)$\Rightarrow$ (3) by induction on $n=\# V(G)$ for all $n\geq 3$. 
	
	Consider the initial case $n=3$. 
	It is easy to check that
	\[ R_3 = \{ K_1 + K_2, (K_1 + K_2)^c, 3K_1,K_3\},\]
	and that those are all the graphs on three vertices. 
	Hence (3) is true, and so is (2)$\Rightarrow$(3).

	Now, suppose that for some integer $n\geq 3$, the cographs on $n$ vertices which do not satisfy Schmidt criterion are exactly $\{X_n,X_n^c,Y_n,Y_n^c\}=R_n$. 
	Let $G$ be a cograph on $n+1$ vertices which does not satisfy Schmidt criterion. 
	We want to prove that $G \in R_{n+1}$. 
	Because $G$ has $n+1 \neq 1$ vertices, by Theorem~\ref{thm:cographs}, we have that there are two cographs $H_1$ and $H_2$ such that $G=H_1+H_2$ or $G^c=H_1+H_2$. 
	By assumption and by remark~\ref{rk:auto_complement}, neither $G$ nor $G^c$ satisfy Schmidt criterion. 
	With Lemma~\ref{lem:symmetric_cographs}, we see that if both $H_1$ and $H_2$ have at least two vertices, they admit nontrivial automorphisms, and so $H_1+H_2$ satisfies Schmidt criterion: this shows by contrapositive that $H_1=K_1$ or $H_2=K_1$. 
	Hence, there is a cograph $H$ such that $G=H+K_1$ or $G^c=H+K_1$. 
	Since $G$ does not satisfy Schmidt criterion, $H$ neither. 
	All in all, $H$ is a cograph on $n$ vertices which does not satisfy Schmidt criterion. 
	Applying the induction hypothesis, we reach ($Y_{n-1}$ is well-defined since $n-1\geq 2$):
	\[ H \in R_n = \{X_n,X_n^c,Y_n,Y_n^c\} = \{X_{n-1}^c + K_1, (X_{n-1}^c+K_1)^c, Y_{n-1}^c + K_1, (Y_{n-1}^c+K_1)^c\}. \]
	It comes:
	\begin{align*}
		G &\in  \{H+K_1,(H+K_1)^c\}\\
		&\subset \{ X_{n-1}^c + 2K_1, (X_{n-1}^c+K_1)^c+K_1, Y_{n-1}^c + 2K_1, (Y_{n-1}^c+K_1)^c + K_1\}\\
		&\cup \{ (X_{n-1}^c + 2K_1)^c, ((X_{n-1}^c+K_1)^c+K_1)^c, (Y_{n-1}^c + 2K_1)^c, ((Y_{n-1}^c+K_1)^c + K_1)^c\}.
	\end{align*}
	Noticing that $X_{n-1}^c$ and $Y_{n-1}^c$ are cographs on $n-1\geq 2$ vertices, they have a nontrivial automorphism by Lemma~\ref{lem:symmetric_cographs}. 
	Hence $X_{n-1}^c+ 2K_1$ and $Y_{n-1}^c+2K_1$ satisfy Schmidt criterion, and, by remark~\ref{rk:auto_complement}, their complements too. 
	Because $G$ does not satisfy Schmidt criterion, we obtain that:
	\begin{align*}
		G &\in \{ (X_{n-1}^c + K_1)^c + K_1, (Y_{n-1}^c + K_1)^c + K_1, ((X_{n-1}^c + K_1)^c + K_1)^c, ((Y_{n-1}^c + K_1)^c + K_1)^c \}\\
		&= \{ X_{n+1},Y_{n+1},X_{n+1}^c,Y_{n+1}^c\},
	\end{align*}
	as desired. 
	This finishes the proof of heredity and the proof of the theorem is complete.
\end{proof}

\subsection{Quantum isomorphism for cographs}

The following result could be proved directly in the spirit of the current section or as a consequence of Theorem~\ref{thm:qi_for_disconnected_graphs}. 
We give it here in the latter presentation.

\begin{theorem}\label{thm:qi_cographs}
	Two cographs are isomorphic if and only if they are quantum isomorphic.
\end{theorem}

\begin{proof}
	We prove it by induction on the number of vertices $n\geq 1$. 
	It is clear for $n=1$. 
	We now assume it true for any $k<n$ for some $n\geq 2$. 
	Let $G$ and $H$ be two quantum ismorphic cographs on $n$ vertices. 
	Since isomorphism and quantum isomorphism are preserved by taking complements, we can assume that $G$ is disconnected by Theorem~\ref{thm:cographs}. 
	We write the decomposition of $G$ and $H$ into connected components as $G = G_1 + \ldots + G_r$ and $H = H_1 + \ldots + H_l$ for some integers $r\geq 2$ and $l\geq 1$. 
	By Theorem~\ref{thm:qi_for_disconnected_graphs}, we have $r=l$ and up to relabelling $G_i \qi H_i$, and in particular $\# V(G_i) = \#V(H_i)$. 
	Since $r\geq 2$, we reach $\# V(G_i) < \# G = n$. 
	Moreover, $G_i$ and $H_i$ are cographs by Corollary~\ref{coro:cographs}. 
	The induction hypothesis then implies that $G_i = H_i$. 
	This immediately leads to the equality $G=H$, as desired, which concludes the induction and the proof.
\end{proof}

\section{Tractable classes of graphs}\label{sec:tractable}

In this section, we introduce the notion of a tractable class of graphs and study some of its general properties. 
In subsection~\ref{subsec:sums}, we study how the quantum properties considered, as well as the quantum automorphism group of a graph, behave with respect to sums. 
In subsection~\ref{subsec:F_cographs}, we introduce the class $\co(\Ff)$ of $\Ff$-cographs associated to a class $\Ff$ and we study how the axioms behave when going from $\Ff$ to $\Ff$-cographs.  
As a result, under some conditions, we obtain in Theorem~\ref{thm:quantum_Jor_vs_Fcographs} the computation of the quantum automorphism groups of graphs in $\co(\Ff)$ as a function of the ones of the graphs in $\Ff$.\\

The notion of tractable class generalises some properties observed for the class of cographs. 
Let $\Ff$ be a class of graphs. 
We introduce the following axioms:
\begin{itemize}
	\item[(QA)] for all $G \in \Ff$, we have that $\Qu(G)$ is trivial if and only if $\Aut(G)$ is trivial,
	\item[(QI)] for all $G$ and $H\in \Ff$, we have that $G\qi H$ if and only if $G= H$,
	\item[(SA)] for all $G\in \Ff$, we have that $G$ has quantum symmetry if and only if $G$ satisfies Schmidt's criterion.
\end{itemize}
The symbols of the axioms stand respectively for quantum asymmetry, quantum isomorphism, and Schmidt alternative. 
When $\Ff$ satisfies all three, we say that $\Ff$ is \textit{tractable}. 
Notice that tractability is inherited by subfamilies.

It is noteworthy that none of these axioms are true for the class of all graphs. 
Indeed, in a recent preprint by de Bruyn, Roberson, and Schmidt~\cite{debruyn2024asymmetric}, the authors show that there are asymmetric graphs with quantum symmetry, thus showing that the axiom (QA) is not true in general. 
This also gives counter-examples to the Schmidt alternative, since asymmetric graphs do not satisfy Schmidt's criterion. 
The fact that (QI) is not true in general was shown already in 2019 in~\cite{QiNotI} by Atserias, Man{\v{c}}inska, Roberson, {\v{S}}{\'a}mal, Severini, and Varvitsiotis, where they give explicit examples of graphs that are quantum isomorphic graphs but not isomorphic.

The axioms relate to one another. 
The Schmidt alternative has some strong implications: we show that it implies (QA) in Lemma~\ref{lem:sa_qa}, and we show in Theorem~\ref{thm:asym_qasym} that, under some conditions, we can recover (QI) from (SA) for asymmetric graphs. 
However, in practice, it turns out to be often more convenient to prove first the other axioms in order to obtain the Schmidt alternative. 

\begin{lemma}\label{lem:sa_qa}
	The Schmidt alternative implies (QA).
\end{lemma}

\begin{proof}
	Let $\Ff$ be a class satisfying (SA). 
	Let $G\in \Ff$. 
	We already know that if $\Qu(G)$ is trivial, then so is $\Aut(G)$. 
	Conversely, assume that $\Aut(G)$ is trivial. 
	In particular, $G$ does not satisfy the Schmidt criterion, so, by (SA), it does not have quantum symmetry. 
	This means that $C(\Qu(G)) = C(\Aut(G)) = \C$, hence, $\Qu(G)$ is trivial, as desired.
\end{proof}

Finally, let us present a strengthening of the axiom (QI). 
A class of graphs $\Ff$ is \textit{superrigid} if, given two graphs $G$ and $H$ with $G\in \Ff$ and $G \qi H$, then $G=H$. 
This notion is due to Freslon (private communication) and is inspired by the notion of superrigidity for group von Neumann algebras. 
We will see in subsection~\ref{subsec:F_cographs} that it behaves particularly well under sums and complements.

\subsection{Sums of graphs}\label{subsec:sums}

For $\Ff$ a class of graphs, we write $\Ff^+ = \{ G_1 + \ldots + G_k \mid k\geq 1,\ G_1,\ldots,G_k\in \Ff\}$. 
Let us start by exploring how the quantum automorphism group behaves with respect to sums. 
In the classical case, it is well-known that the automorphism group behaves as follows (we state it here without proof).

\begin{theorem}\label{thm:aut_sums}
	Let $n\in \N$ and consider $n$ connected nonisomorphic graphs $G_1,\ldots,G_n$ as well as $n$ nonzero integers $a_1,\ldots,a_n \geq 1$. 
	Let $G= \sum_{i=1}^n a_iG_i$. 
	Then $\Aut(G) = \prod_{i=1}^n \Aut(G_i) \wr S_{a_i}$ and its action is the natural one.
\end{theorem}

We have an analogue in the quantum case. 
Unfortunatetly (or fortunately for the sake of the theory), when in presence of nonisomorphic but quantum isomorphic graphs, the analogy breaks. 

\begin{theorem}\label{thm:qu_aut_sums}
	Let $n\in \N$ and consider $n$ connected non quantum isomorphic graphs $G_1,\ldots,G_n$ as well as $n$ nonzero integers $a_1,\ldots,a_n \geq 1$. 
	Let $G= \sum_{i=1}^n a_iG_i$. 
	Then $\Qu(G) = *_{i=1}^n \Qu(G_i)\wr S_{a_i}^+$ and its action is the natural one. 
\end{theorem}

\begin{proof}
	By Theorem~\ref{thm:qi_for_connected_components}, it is clear that $\Qu(G) = *_{i=1}^n \Qu(a_iG_i)$, with $\Qu(a_iG_i)$ acting on $a_iG_i$. 
	Now the conclusion follows from Theorem~\ref{thm:wreath_product}.
\end{proof}

Let us now characterise when a sum of connected graphs does not have quantum symmetry.

\begin{theorem}\label{thm:sum_no_qs}
	Let $n\in \N$ and consider $n$ connected nonisomorphic graphs $G_1,\ldots,G_n$ as well as $n$ nonzero integers $a_1,\ldots,a_n \geq 1$. 
	Let $G= \sum_{i=1}^n a_iG_i$. 
	Then $G$ does not have quantum symmetry if and only if $G_i \not \qi G_j$ when $i\neq j$ and there is an integer $1\leq l\leq n$ such that:
	\begin{enumerate}
		\item for every $i\neq l$, we have $a_i = 1$ and $\Qu(G_i) = 1$,
		\item $a_l\leq 3$,
		\item $G_l$ does not have quantum symmetry,
		\item if $a_l >1$, then $\Qu(G_l) = 1$.
	\end{enumerate}
\end{theorem}

\begin{proof}
	We start by proving the conditions necessary by contrapositive.

	Clearly, if there is $1\leq i\leq n$ such that $G_i$ has quantum symmetry, then so does $G$. 
	So none of the $G_i$ have quantum symmetry if $G$ does not have quantum symmetry.

	Assume that there are $i\neq j$ such that $G_i \qi G_j$. 
	Then, by Lemma~\ref{lem:classical_qi_is_i}, there is a quantum isomorphism $U$ from $G_i$ to $G_j$ whose coefficients do not all commute. 
	By Lemma~\ref{lem:mu_vs_qi} we have that $V = \begin{pmatrix}
		0 & U\\
		U^* & 0
	\end{pmatrix}$ is a magic unitary adapted to $G_i+G_j$ with noncommutative coefficients. 
	Hence, $G_i+G_j$ has quantum symmetry, and so does $G$. 
	This shows that $G_i\not \qi G_j$ when $i\neq j$ if $G$ does not have quantum symmetry.

	Assume that there are $i\neq j$ such that $\Qu(G_i) \neq 1$ and $\Qu(G_j) \neq 1$. 
	Then, by the quantum Schmidt's criterion (Theorem~\ref{thm:qu_schmidt}), $G$ has quantum symmetry. 
	This shows that there is at most one $1\leq l\leq n$ such that $\Qu(G_l)\neq 1$ if $G$ does not have quantum symmetry.

	Assume that there are $i\neq j$ such that $a_i\geq 2$ and $a_j\geq 2$. 
	Then $a_iG_i + a_jG_j$ satisfies Schmidt's criterion (since $2G_i + 2G_j$ does) and so has quantum symmetry, which implies that $G$ has quantum symmetry. 
	This shows that there is at most one $1\leq k\leq n$ such that $a_k\geq 2$ if $G$ does not have quantum symmetry.

	Assume there is $1\leq l\leq n$ such that $\Qu(G_l) \neq 1$ and assume that $a_l\geq2$. 
	We have that $2G_l$ satisfies the quantum Schmidt criterion: indeed, let $U$ be a nontrivial magic unitary adapted to $G_l$, then $\Diag(U,1)$ and $\Diag(1,U)$ are two nontrivial magic unitaries adapted to $2G_l$, so Theorem~\ref{thm:qu_schmidt} applies, and $2G_l$ has quantum symmetry, hence $G$ too. 
	This shows that $a_l=1$. 
	Assume now there is $k\neq l$ such that $a_k\geq 2$. 
	Then $G_l + 2G_k$ satisfies the quantum Schmidt criterion: this can be seen by taking one magic unitary to be a classical automorphism exchanging the two copies of $G_k$, the other one is a nontrivial magic unitary adapted to $G_l$. 
	Hence $G_l+2G_k$ has quantum symmetry, and so does $G$. 
	All this shows that if $G$ does not have quantum symmetry, if there is $l$ such that $\Qu(G_l)\neq 1$, then $a_i=1$ for all $1\leq i\leq n$, and if there is $k$ such that $a_k\geq 2$, then $\Qu(G_i) = 1$ for all $1\leq i\leq n$. 

	Now assume that $G$ does not have quantum symmetry. 
	We know that $G_i\not \qi G_j$ if $i\neq j$: this is the first property. 

	Let $1\leq l\leq n$ be an index such that $\Qu(G_l) \neq 0$ if there is one. 
	If not, let $1\leq l\leq n$ be such that $a_l\geq 2$ if there is one. 
	If there is no such index, let $l=1$. 
	Since $4G_l$ satisfies the Schmidt criterion, we know that $a_l\leq 3$: this proves (2). 
	By what precedes, we know that if $a_l>1$, then $\Qu(G_l) = 1$: this proves (4). 
	We also know that for all $i\neq l$, we have $a_i = 1$ and $\Qu(G_i) = 1$: this proves (1). 
	Finally, we know that $G_l$ does not have quantum symmetry: this proves (3). 
	This concludes the proof that the listed conditions are necessary.

	Conversely, let us prove the reverse implication. 
	Assume that the listed properties hold. 
	Then, by Theorem~\ref{thm:qu_aut_sums}, we have $\Qu(G) = *_{i=1}^n \Qu(G_i)\wr S_{a_i}^+ = \Qu(G_l) \wr S_{a_l}^+$. 
	If $a_l>1$, we have $\Qu(G_l) = 1$, which leads to $\Qu(G) = S_{a_l}^+$. 
	Since $a_l\leq 3$ by (2), we have that $G$ does not have quantum symmetry, since $\Oo(S_i^+)$ is abelian for $i\leq 3$. 
	If $a_l=1$, then $S_{a_l}^1 = 1$, and $\Qu(G) = \Qu(G_l)$. 
	Now $\Oo(\Qu(G_l))$ is abelian by (3) and $G$ does not have quantum symmetry. 
	This shows that $G$ does not have quantum symmetry, as desired, and concludes the proof.
\end{proof}

Notice how the Schmidt criterion behaves similarly to quantum symmetry with respect to sums.

\begin{theorem}\label{thm:sum_no_schmidt}
	Let $n\in \N$ and consider $n$ connected nonisomorphic graphs $G_1,\ldots,G_n$ as well as $n$ nonzero integers $a_1,\ldots,a_n \geq 1$. 
	Let $G= \sum_{i=1}^n a_iG_i$. 
	Then $G$ does not satisfy Schmidt criterion if and only if there is an integer $1\leq l\leq n$ such that:
	\begin{enumerate}
		\item for every $i\neq l$, we have $a_i = 1$ and $\Aut(G_i) = 1$,
		\item $a_l\leq 3$,
		\item $G_l$ does not satisfy Schmidt's criterion,
		\item if $a_l >1$, then $\Aut(G_l) = 1$.
	\end{enumerate}
\end{theorem}

\begin{proof}
	Let us start by proving the reverse implication. 
	Assume that there is such an index $1\leq l\leq n$. 
	By Theorem~\ref{thm:aut_sums}, we have that $\Aut(G) = \prod_{i=1}^n \Aut(G_i) \wr S_{a_i} = \Aut(G_l) \wr S_{a_l}$, with the natural action. 
	In particular, $G$ satisfies the Schmidt criterion if and only if $a_lG_l$ does. 
	Now if $a_l>1$ then $\Aut(G_l) = 1$ and $a_l\leq 3$ so $a_l G_l$ does not satisfy the Schmidt criterion. 
	And if $a_l=1$ then $a_lG_l = G_l$ and we know that $G_l$ does not satisfy the Schmidt criterion, so $G$ does not satisfy it neither. 
	This proves the reverse implication.

	Let us prove the forward implication. 
	If there is $1\leq i\leq n$ such that $a_i >1$, take $l=i$. 
	If not, if there is $1\leq j\leq n$ such that $\Aut(G_i) \neq 1$, take $l=j$. 
	If not, take $l=1$. 
	We claim that $l$ satisfies all the desired criteria.

	First, if $a_l \geq 4$, then $4G_l$ is a subgraph of $G$ disjoint from the rest of $G$ and which satisfies Schmidt criterion, so clearly does $G$. 
	Hence $a_l\leq 3$. 

	If $G_l$ satisfies Schmidt criterion, so does $G$. 
	Thus $G_l$ does not satisfy Schmidt criterion.

	If $a_l >1$ and $\Aut(G_l) \neq 1$, then $2G_l$ satisfies Schmidt criterion and so does $G$. 
	So if $a_l >1$, we have $\Aut(G_l) = 1$. 

	Finally, there is only the first criterion to prove. 
	Let $i\neq l$. 
	Let us assume first that $a_l>1$. 
	Then if $a_i>1$, we have that $2G_l + 2G_i$ is a subgraph of $G$ disjoint from the rest of $G$ satisfying Schmidt criterion, and so does $G$: this shows that $a_i=1$. 
	If $\Aut(G_i)\neq 1$, then $2G_l +G_i$ satisfies Schmidt criterion, and so does $G$: this shows that $\Aut(G_i) = 1$. 
	Hence the proof is finished in the case where $a_l>1$.

	Assume now that $a_l =1$ and $\Aut(G_l) \neq 1$. 
	Similarly, if $a_i >1$, then $G_l + 2G_i$, and, if $\Aut(G_i) \neq 1$, then $G_i + G_l$, satisfy Schmidt criterion and make $G$ inherit it. 
	So once again $a_i=1$ and $\Aut(G_i) = 1$. 

	Finally, let us assume that $a_l=1$ and that $\Aut(G_l)=1$. 
	By construction, this means that $a_i = 1$ and $\Aut(G_i) = 1$, as desired. 
	This concludes the proof of the forward implication and the proof is finished.
\end{proof}

This allows us to investigate how the three axioms of tractability behave with respect to sums.

\begin{theorem}\label{thm:problems_sums}
	Let $\Ff$ be a class of graphs which contains the connected components of all its elements. 
	Then:
	\begin{enumerate}
		\item if $\Ff$ satisfies (QI), then so does $\Ff^+$,
		\item if $\Ff$ satisfies (QI) and (QA), then so does $\Ff^+$,
		\item if $\Ff$ satisfies (QI), (QA), and (SA), then so does $\Ff^+$.
	\end{enumerate}
\end{theorem}

\begin{proof}
	Assume $\Ff$ satisfies (QI) and let $G$, $H\in \Ff^+$ be quantum isomorphic. 
	By Theorem~\ref{thm:qi_for_disconnected_graphs}, $G$ and $H$ have the same number $k\in \N$ of connected components, and we can decompose $G = G_1 + \ldots + G_k$ and $H = H_1 + \ldots + H_k$ such that $G_i \qi H_i$ for all $1\leq i\leq k$. 
	For $1\leq i\leq k$, we have that $G_i \in \Ff$ and $H_i \in \Ff$, since by construction they are connected components of elements of $\Ff$. 
	Since $\Ff$ satisfies (QI), we have $G_i = H_i$, from which we deduce that $G=H$, as desired.

	Assume now that $\Ff$ also satisfies (QA) and let $G\in \Ff^+$ be an asymmetric graph. 
	We can write $G = \sum_{i=1}^k a_i G_i$ where each $G_i \in \Ff$ is connected and $a_i \geq 1$. 
	By Theorem~\ref{thm:aut_sums}, it comes:
	\[1 = \Aut(G) = \prod_{i=1}^k \Aut(G_i)\wr S_{a_i},\]
	from which we conclude that $\Aut(G_i) = 1$ and $a_i=1$ for all $1\leq i\leq k$. 
	For all $1\leq i\leq k$, we have $G_i \in \Ff$ and $\Ff$ satisfies (QA) by assumption, so $\Qu(G_i) = 1$. 
	Since $\Ff^+$ satisfies (QI) by the first part of this proof, Theorem~\ref{thm:qu_aut_sums} gives us $\Qu(G) = *_{i=1}^k \Qu(G_i) \wr S_{a_i}^+ = *_{i=1}^k 0 \wr S_1 = 1$. 
	This shows that $\Ff^+$ satisfies (QA), as desired.

	Finally, assume that $\Ff$ also satisfies the Schmidt alternative. 
	By what precedes, we have that $\Ff^+$ satisfies both (QI) and (QA) and we want to prove it satisfies (SA) too. 
	Let $G \in \Ff^+$ and assume that $G$ does not satisfy the Schmidt criterion. 
	Let us write $G = \sum_{i=1}^k a_iG_i$ with $G_i \in \Ff$ connected for all $1\leq i\leq k$. 
	By Theorem~\ref{thm:sum_no_schmidt}, we obtain an integer $1\leq l\leq k$ with the properties stated in the theorem. 
	Since $\Ff$ satisfies (QI) and (QA), we obtain that $G$ does not have quantum symmetry by Theorem~\ref{thm:sum_no_qs}. 
	
	This concludes the proof.
\end{proof}

Finally, we show how in some cases we can even recover (QI) from (SA) for asymmetric graphs.

\begin{theorem}\label{thm:asym_qasym}
	Let $\Ff$ be a class of graphs satisfying (SA), stable by sums, and by taking connected components. 
	If $G$ and $H\in \Ff$ are asymmetric and if $G\qi H$, then $G=H$.
\end{theorem}

\begin{proof}
	We start by proving the connected case. 
	Consider $G$ and $H\in \Ff$ two connected quantum isomorphic asymmetric graphs. 
	By Lemma~\ref{lem:sa_qa}, we have that $\Qu(G) = 1 = \Qu(H)$. 
	Notice that $G+H$ does not satisfy the Schmidt criterion. 
	Since $\Ff$ is stable by sums and satisfies the Schmidt alternative, we have that $G+H$ does not have quantum symmetry. 
	Let $U$ be a quantum isomorphism from $G$ to $H$. 
	Let $W = \begin{pmatrix}
		0 & U\\
		U^* & 0
	\end{pmatrix}$. 
	Then $W$ is a magic unitary adapted to $G+H$ when the vertices of $G+H$ are enumerated in the appropriate order. 
	Since $G+H$ does not have quantum symmetry, we have that the coefficients of $W$ commute, and so do the ones of $U$. 
	Now by Lemma~\ref{lem:classical_qi_is_i} $G$ is isomorphic to $H$, as desired. 
	
	Let $\Ff_0 = \{ L \in \Ff \mid L\text{ is connected and asymmetric}\}$. 
	By what precedes, we have that $\Ff_0$ satisfies (QI), hence by Theorem~\ref{thm:problems_sums} $\Ff_0^+$ satisfies (QI) too. 
	Let $G$ and $H\in \Ff$ be two quantum isomorphic asymmetric graphs. 
	Notice that since an asymmetric graph has asymmetric connected components, we have that $G$ and $H$ are elements of $\Ff_0^+$. 
	Because $\Ff_0^+$ satisfies (QI), we conclude that $G$ is isomorphic to $H$, as desired. 
	This concludes the proof.
\end{proof}

\subsection{$\Ff$-cographs}\label{subsec:F_cographs}

Let $\Ff$ be a class of graphs. 
The class of $\Ff$-cographs, that we denote by $\co(\Ff)$, is the smallest class of graphs (for inclusion) that contains $\Ff$ and is stable by taking sums and complement. 
Notice that cographs are exactly $\{K_1\}$-cographs.

Building on the results obtained in Section~\ref{sec:qi_sums}, we extend the work on cographs of Section~\ref{sec:cographs} to the more general context of $\Ff$-cographs. 
Indeed, we show in Theorem~\ref{thm:tractable_F_cographs} that, under some conditions, $\co(\Ff)$ is tractable when $\Ff$ is. 
Under some conditions, this allows us to obtain the quantum automorphism groups of $\Ff$-cographs as a function of the ones of graphs in $\Ff$ in Theorem~\ref{thm:quantum_Jor_vs_Fcographs}. 
These results will later allow us in Section~\ref{sec:forests} to extend to tree-cographs our results on forests, including the computation of their quantum automorphism groups.

Let $\Qq$ be a family of quantum permutation groups. 
We define the \textit{Jordan closure} of $\Qq$, that we shall by $\Jor(\Qq)$, to be the smallest family of quantum permutation groups (for inclusion) that contains $\Qq$ and is stable by free product and wreath product with $S_n^+$ for every $n\in \N$. 

With a slight abuse, we also talk about the Jordan closure of $Q$, that we denote by $\Jor(Q)$, when $Q$ is a family of groups: this is the smallest family of groups (for inclusion) which contains $Q$ and is stable by product and wreath product with $S_n$ for every $n\in \N$. 
It will always be clear from context whether we are considering groups or quantum permutation groups.

Given $\Ff$ a class of graphs, we write $\Aut(\Ff) = \{ \Aut(G) \mid G\in \Ff\}$ and $\Qu(\Ff) = \{ \Qu(G) \mid G\in \Ff\}$.

We start with some characterisations. 
When $\Ff$ is a class of graphs, we write $\Ff^c = \{ G^c \mid G\in \Ff\}$ and $\Ff^+ = \{ G_1 + \ldots + G_k \mid k\geq 1,\ G_1,\ldots,G_k \in \Ff\}$. 
We also define by induction $\Ff_0 =\Ff$, $\Ff_{2i+1} = \Ff_{2i} \cup \Ff_{2i}^c$ and $\Ff_{2i+2} = \Ff_{2i+1}^+$ for $i\geq 0$. 

\begin{lemma}\label{lem:F_cographs}
	For $\Ff$ a class of graphs, we have that $\co(\Ff) = \bigcup_{i\geq 0} \Ff_i$.
\end{lemma}

\begin{proof}
	Clear.
\end{proof}

\begin{theorem}\label{thm:decomposition_F_cographs}
	Let $\Ff$ be a class of graphs and let $G$ be a graph. 
	Then $G\in \co(\Ff)$ if and only if one of the following is true:
	\begin{enumerate}
		\item $G\in \Ff \cup \Ff^c$,
		\item there exists $k\geq 2$ and $G_1,\ldots,G_k \in \co(\Ff)$ such that $G = G_1 + \ldots + G_k$,
		\item there exists $k\geq 2$ and $G_1,\ldots,G_k \in \co(\Ff)$ such that $G^c = G_1 + \ldots + G_k$.
	\end{enumerate}
\end{theorem}

\begin{proof}
	If $G$ satisfies items (1), (2), or (3), then it is clear that $G\in \co(\Ff)$. 
	Let us prove the converse. 
	Assume that $G\in \co(\Ff)$ and let $i = \min \{ j\geq 0 \mid G \in \Ff_j\}$. 
	We have that $i < +\infty$ by Lemma~\ref{lem:F_cographs}. 
	If $i=0$, then $G\in \Ff_0 = \Ff$ and (1) is true.\\
	So let us assume that $i\neq 0$. 
	If $i$ is even, by construction, there exist $k\geq 1$ and $G_1,\ldots, G_k \in \Ff_{i-1}$ such that $G = G_1 + \ldots + G_k$. 
	By definition of $i$, we have that $G\notin \Ff_{i-1}$, so $k\geq 2$. 
	Since $\Ff_{i-1} \subset \co(\Ff)$, this shows that (2) is true.\\
	Finally, if $i$ is odd, by construction, we have $G \in \Ff_{i-1} \cup \Ff_{i-1}^c$. 
	Since $G\notin \Ff_{i-1}$, we have that $G^c \in \Ff_{i-1}$. 
	If $i=1$, then $G^c \in \Ff_0 = \Ff$ and (1) is true. 
	If $i>1$, by construction, there exist $k\geq 1$ and $G_1,\ldots,G_k \in \Ff_{i-2}$ such that $G^c = G_1 + \ldots + G_k$. 
	It remains to show that $k>1$. 
	We show it by contradiction. 
	Assume $k=1$. 
	Then $G^c \in \Ff_{i-2}$. 
	Since $i$ is odd, we have $\Ff_{i-2} = \Ff_{i-2}^c$ by construction, so $G\in \Ff_{i-2}$: a contradiction. 
	Hence $k\geq 2$ and (3) is true.
\end{proof}

\begin{lemma}\label{lem:base_case_coF}
	Let $\Ff$ be a class of graphs and define $n_0 = \min \{ \abs{V(G)} \mid G\in \Ff\}$. 
	Let $G\in \co(\Ff)$. 
	If $\abs{V(G)} = n_0$, then $G\in \Ff\cup \Ff^c$.
\end{lemma}

\begin{proof}
	By Lemma~\ref{lem:F_cographs}, we immediately have that $\abs{V(H)} \geq n_0$ for all $H\in \co(\Ff)$. 
	Now consider $G\in \co(\Ff)$ with $\abs{V(G)} = n_0$. 
	By Theorem~\ref{thm:decomposition_F_cographs}, we have that (1), (2), or (3) is true. 
	Let us prove by contrapositive that (2) and (3) are false. 
	Indeed, let $H_1,\ldots,H_k \in \co(\Ff)$, with $k\geq 2$. 
	Now for $H=H_1 + \ldots + H_k$, we have that $\abs{V(H^c)} = \abs{V(H)} \geq 2n_0 > n_0$. 
	This shows that $G$ does not satisfy (2) or (3), hence it satisfies (1) and the proof is finished.
\end{proof}

\begin{lemma}\label{lem:coF_connected_component}
	Let $\Ff$ be a class of graphs closed under taking complements and connected components. 
	Then $\co(\Ff)$ is closed under taking connected components.
\end{lemma}

\begin{proof}
	This follows immediately from Theorem~\ref{thm:decomposition_F_cographs}.
\end{proof}

We show that (QA) and (SA) are stable under taking complements.

\begin{lemma}\label{lem:tractable_complement}
	Let $\Ff$ be class of graphs. 
	Then we have the following:
	\begin{enumerate}
		\item if $\Ff$ satisfies (QA), so does $\Ff\cup \Ff^c$,
		\item if $\Ff$ satisfies (SA), so does $\Ff\cup \Ff^c$.
	\end{enumerate}
\end{lemma}

\begin{proof}
	Assume that $\Ff$ satisfies (QA). 
	Let $G\in \Ff^c$ be an asymmetric graph. 
	Notice that $G^c$ is asymmetric too, hence $1 = \Qu(G^c) = \Qu(G)$. 
	This shows (1).

	Assume that $\Ff$ satisfies (SA). 
	Let $G\in \Ff^c$ and assume that $G$ does not satisfy Schmidt's criterion. 
	Then $G^c$ does not satisfy it neither, so $G^c \in \Ff$ does not have quantum symmetry by (SA), and since $\Qu(G^c) = \Qu(G)$, we have that $G$ does not have quantum symmetry neither. 
	This shows that $\Ff\cup \Ff^c$ satisfies the Schmidt alternative, hence (3) is true.
\end{proof}

Unfortunately, the axiom (QI) is not stable under taking complements. 
Indeed, let $G$ and $H$ be two graphs such that $G\qi H$, $G\not \qi H^c$, and $G\neq H$ (for instance the two graphs exhibited in~\cite{QiNotI}). 
Let $\Ff = \{G,H^c\}$. 
It is clear that $\Ff$ satisfies (QI). 
However, we have $\Ff \cup \Ff^c = \{G,G^c, H, H^c\}$, which does not satisfy (QI). 

To extend tractability to $\Ff$-cographs, we will then need the extra hypothesis that $\Ff$ is stable by complements.

\begin{theorem}\label{thm:tractable_F_cographs}
	Let $\Ff$ be a class of graphs. 
	Assume that $\Ff$ is stable by taking complements and connected components. 
	We have the following:
	\begin{enumerate}
		\item if $\Ff$ satisfies (QI), so does $\co(\Ff)$,
		\item if $\Ff$ satisfies (QI) and (QA), so does $\co(\Ff)$,
		\item if $\Ff$ satisfies (QI), (QA), and (SA), so does $\co(\Ff)$.
	\end{enumerate}
\end{theorem}

\begin{proof}
	Let us prove (1) by induction on the number of vertices. 
	Let $n_0 = \min \{ \abs{V(G)} \mid G\in \Ff\}$ and consider $G$, $H\in \co(\Ff)$ with $\abs{V(G)} = n_0$. 
	Assume that $G\qi H$. 
	This implies that $\abs{V(H)} = n_0$ and by Lemma~\ref{lem:base_case_coF} we have that $G$ and $H$ are in $\Ff\cup \Ff^c$, which satisfies (QI) by hypothesis. 
	Thus $G=H$.\\
	Now assume that $\Ff \cap \{ X \mid \abs{V(X)} \leq m\}$ satisfies (QI) for some $m\geq n_0$. 
	Let $G$ and $H \in \co(\Ff)$ be quantum isomorphic graphs on $m+1$ vertices (notice there might not be such graphs). 
	Since $\co(\Ff)$ is closed under taking complements, we can assume that $G$ and $H$ are connected. 
	If both $G$ and $H$ are in $\Ff\cup \Ff^c = \Ff$, we have that $G=H$ and we are done. 
	So up to symmetry we can assume by Theorem~\ref{thm:decomposition_F_cographs} and Lemma~\ref{lem:coF_connected_component} that there is $k\geq 2$ and $G_1,\ldots,G_k \in \co(\Ff)$ connected such that $G^c = G_1 + \ldots G_k$. 
	Then Theorem~\ref{thm:qi_for_disconnected_graphs} implies that $H^c = H_1 + \ldots + H_k$ with $H_1,\ldots,H_k \in \co(\Ff)$ and connected and with $H_i \qi G_i$ for all $1\leq i\leq k$. 
	Now since $k\geq 2$ we have that $m\geq \abs{V(G_i)} = \abs{V(H_i)}$ for all $1\leq i\leq k$, so by induction hypothesis we have that $G_i=H_i$ for all $1\leq i\leq k$ and $G=H$, as desired. 
	Hence this concludes the proof of (1).

	Combining Lemma~\ref{lem:tractable_complement} and Theorem~\ref{thm:problems_sums}, we obtain (2) and (3) by a clear induction.
\end{proof}

Let us now show that superrigidity behaves well with respect to sums and complements.

\begin{lemma}\label{lem:superrigid_complement}
	If $\Ff$ is a superrigid class, then so is $\Ff \cup \Ff^c$.
\end{lemma}

\begin{proof}
	Let $G \in \Ff \cup \Ff^c$ and let $H$ be a graph such that $G\qi H$. 
	If $G \in \Ff$, then by superrigidity $G=H$, as desired. 
	If not, then $G\in \Ff^c$, and $G^c \qi H^c$. 
	Since $G^c \in \Ff$, by superrigidity, $H^c = G^c$, so $G=H$. 
	This concludes the proof.
\end{proof}

\begin{lemma}\label{lem:superrigid_sums}
	If $\Ff$ is superrigid class, then so is $\Ff^+$.
\end{lemma}

\begin{proof}
	Let $G \in \Ff^+$. 
	We can write $G = G_1 + \ldots + G_k$ for some $k\geq 1$ with $G_i \in \Ff$ for all $1\leq i\leq k$. 
	For every $1\leq i\leq k$, we write $G_i = G_i^1 + \ldots + G_i^{l_i}$ the decomposition in connected components of $G_i$, where $l_i\geq 1$  is the number of connected components of $G_i$. 
	Now let $H$ be a graph such that $G\qi H$. 
	By Theorem~\ref{thm:qi_for_disconnected_graphs}, we have connected graphs $H_1^1,\ldots,H_1^{l_1},\ldots,H_k^1,\ldots,K_k^{l_k}$ such that $H_i^j \qi G_i^j$ for all $1\leq i\leq k$ and $1\leq j\leq l_i$, and such that $H = \sum_{i=1}^k \sum_{j=1}^{l_i} H_i^j$. 
	For $1\leq i\leq k$, we write $H_i = \sum_{j=1}^{l_i} H_i^j$. 
	Since $H_i^j \qi G_i^j$ for all $1\leq i\leq k$ and all $1\leq j\leq l_i$, we have that $H_i\qi G_i$. 
	Since $G_i \in \Ff$, by superrigidity, we have that $H_i = G_i$. 
	Now $H = \sum_{i=1}^k H_i = \sum_{i=1}^k G_i = G$, as desired. 
	This concludes the proof that superrigidity is preserved by sums.
\end{proof}

Hence superrigidity is preserved by taking the cograph closure.

\begin{theorem}\label{thm:superrigid_F_cographs}
	If $\Ff$ is a superrigid class of graphs, then so is $\co(\Ff)$.
\end{theorem}

\begin{proof}
	This follows immediately from Lemma~\ref{lem:F_cographs}, Lemma~\ref{lem:superrigid_complement}, and Lemma~\ref{lem:superrigid_sums}.
\end{proof}

For the rest of this section, we turn towards the computation of the classical and quantum automorphism groups of $\Ff$-cographs. 
For $\Ff$ a class of graphs, we denote by $\Ff_{conn}$ the subclass of connected graphs of $\Ff$. 

\begin{theorem}\label{thm:classical_Jor_F_cographs}
	Let $\Ff$ be a class of graphs. 
	Assume that:
	\begin{enumerate}
		\item $\Ff$ is stable by taking complements,
		\item $\Ff$ is stable by taking connected components.
	\end{enumerate}
	Then we have $\Aut(\co(\Ff)) \subset \Jor(\Aut(\Ff_{conn}))$.
\end{theorem}

\begin{proof}
	Let us prove it by induction on the number of vertices. 
	Let $G \in \co(\Ff)$ be a graph on $n_0 = \min \{ \abs{V(H)} \mid H\in \Ff\}$ vertices. 
	By Lemma~\ref{lem:base_case_coF}, we have $G\in \Ff\cup \Ff^c = \Ff$. 
	Let $G = \sum_{i=1}^k a_iG_i$ for some $k\in \N$ be the decomposition of $G$ into connected components. 
	By assumption, we have that $G_i\in \Ff$ for all $1\leq i\leq k$, and so $G_i \in \Ff_{conn}$. 
	By Theorem~\ref{thm:aut_sums}, we have that $\Aut(G) = \prod_{i=1}^k \Aut(G_i)\wr S_{a_i} \in \Jor(\Aut(\Ff_{conn}))$, as desired. 

	Now let us assume that $\Aut(G) \in \Jor(\Aut(\Ff_{conn}))$ for all $G\in \co(\Ff)$ on at most $n$ vertices for some $n\geq n_0$. 
	Let $G\in \co(\Ff)$ be a graph on $n+1$ vertices. 
	If $G\in \Ff\cup \Ff^c$, we are done by what precedes. 
	Otherwise, since $\Aut(G) = \Aut(G^c)$, we can assume that $G$ is disconnected by Theorem~\ref{thm:decomposition_F_cographs}. 
	Once again we write $G = \sum_{i=1}^k a_iG_i$ for some $k\in \N$ be the decomposition of $G$ into connected components. 
	By Lemma~\ref{lem:coF_connected_component}, we have that $G_i \in \co(\Ff)$ for all $1\leq i\leq k$. 
	By Theorem~\ref{thm:aut_sums}, we have that $\Aut(G) = \prod_{i=1}^k \Aut(G_i) \wr S_{a_i}$. 
	Since $G$ is disconnected, we have that $\abs{V(G_i)} < \abs{V(G)} = n+1$ for all $1\leq i\leq k$, so, by induction hypothesis, we have that $\Aut(G_i) \in \Jor(\Aut(\Ff_{conn}))$. 
	Since by definition this class is closed under taking wreath products with a permutation group and by direct products, this shows that $\Aut(G) \in \Jor(\Aut(\Ff_{conn}))$, as desired. 
	This concludes the proof.
\end{proof}

\begin{theorem}\label{thm:K1_vs_Jor}
	Let $\Ff$ be a class of graphs such that $K_1\in \Ff$. 
	Then $\Jor(\Aut(\Ff)) \subset \Aut(\co(\Ff))$. 
\end{theorem}

\begin{proof}
	It is enough to show that $\Aut(\co(\Ff))$ contains $\Aut(\Ff)$ (which is clear), is stable by direct product, and by wreath product with $S_d$ for $d\in \N$. 

	Let $G\in \co(\Ff)$ and $d\in \N$. 
	We claim that $\Aut(G) \wr S_d \in \Aut(\co(\Ff)$. 
	Indeed, if $G$ is connected, set $H = d.G$, otherwise, set $H = d.G^c$. 
	Now $H\in \co(\Ff)$, and, by Theorem~\ref{thm:aut_sums}, since $\Aut(G) = \Aut(G^c)$, we have that $\Aut(G) \wr S_d = \Aut(H) \in \Aut(\co(\Ff))$, as desired. 

	Then take $G$ and $H\in \co(\Ff)$. 
	We claim that $\Aut(G) \times \Aut(H) \in \Aut(\co(\Ff))$. 
	Indeed, notice that up to taking complements we can assume both $G$ and $H$ connected. 
	If $G\neq H$, set $L = G+H \in \co(\Ff)$. 
	We have $\Aut(G)\times \Aut(H) = \Aut(L) \in \Aut(\co(\Ff))$ by Theorem~\ref{thm:aut_sums}, as desired. 
	Assume now that $G=H$. 
	Since $K_1\in \Ff$, we have that $L' = (G+K_1)^c \in \co(\Ff)$. 
	Notice that $L'\neq G$ (since they do not have the same number of vertices), that $L'$ is connected by Lemma~\ref{lem:complement}, and that $\Aut(L') = \Aut(G+K_1) = \Aut(G) \times \Aut(K_1) = \Aut(G)$ since $G$ is connected. 
	Setting $L=G+L'$, we have that $\Aut(G)^2 = \Aut(G) \times \Aut(L') = \Aut(L) \in \Aut(\co(\Ff))$, as desired. 
	
	Hence we have shown that $\Aut(\co(\Ff))$ is stable by direct product, by wreath product with $S_d$ for $d\geq 1$, and contains $\Aut(\Ff_{conn})$. 
	By definition, we have $\Jor(\Aut(\Ff)) \subset \Aut(\co(\Ff))$, as desired. 
	This concludes the proof.
\end{proof}

Combining the precedent results, we obtain the automorphism groups of $\Ff$-cographs as a function of the automorphism groups of graphs of $\Ff$.

\begin{theorem}\label{thm:Jor_vs_Fcographs}
	Let $\Ff$ be a class of graphs. 
	Assume that:
	\begin{enumerate}
		\item $K_1\in \Ff$,
		\item $\Ff$ is stable by taking complements,
		\item $\Ff$ is stable by taking connected components.
	\end{enumerate}
	Then we have that $\Aut(\co(\Ff)) = \Jor(\Aut(\Ff_{conn})) = \Jor(\Aut(\co(\Ff)))$.
\end{theorem}

\begin{proof}
	By Theorem~\ref{thm:classical_Jor_F_cographs}, we have that $\Aut(\co(\Ff)) \subset \Jor(\Aut(\Ff_{conn}))$. 
	Since $\Ff_{conn} \subset \Ff$, we have that $\Jor(\Aut(\Ff_{conn})) \subset \Jor(\Aut(\Ff))$. 
	By Theorem~\ref{thm:K1_vs_Jor}, we have that $\Jor(\Aut(\Ff)) \subset \Aut(\co(\Ff))$. 
	This concludes the proof.
\end{proof}

When $\Ff$ satisfies (QI), these results extend to the noncommutative case as well, allowing us to explicitly compute the quantum automorphism groups of $\Ff$-cographs. 

\begin{theorem}\label{thm:quantum_Jor_F_cographs}
	Let $\Ff$ be a class of graphs. 
	Assume that:
	\begin{enumerate}
		\item $\Ff$ satisfies (QI),
		\item $\Ff$ is stable by taking complements,
		\item $\Ff$ is stable by taking connected components.
	\end{enumerate}
	Then $\Qu(\co(\Ff)) \subset \Jor(\Qu(\Ff_{conn}))$.
\end{theorem}

\begin{proof}
	Let us prove it by induction on the number of vertices. 
	Let $G \in \co(\Ff)$ be a graph on $n_0 = \min \{ \abs{V(H)} \mid H\in \Ff\}$ vertices. 
	By Lemma~\ref{lem:base_case_coF}, we have $G\in \Ff\cup \Ff^c = \Ff$. 
	Let $G = \sum_{i=1}^k a_iG_i$ for some $k\in \N$ be the decomposition of $G$ into connected components. 
	By assumption, we have that $G_i\in \Ff$ for all $1\leq i\leq k$, and so $G_i \in \Ff_{conn}$. 
	Since $\Ff$ satisfies (QI), by Theorem~\ref{thm:qu_aut_sums}, we have that $\Qu(G) = *_{i=1}^k \Qu(G_i)\wr S_{a_i}^+ \in \Jor(\Qu(\Ff_{conn}))$, as desired. 

	Now let us assume that $\Aut(G) \in \Jor(\Qu(\Ff_{conn}))$ for all $G\in \co(\Ff)$ on at most $n$ vertices for some $n\geq n_0$. 
	Let $G\in \co(\Ff)$ be a graph on $n+1$ vertices. 
	If $G\in \Ff\cup \Ff^c$, we are done by what precedes. 
	Otherwise, since $\Qu(G) = \Qu(G^c)$, we can assume that $G$ is disconnected by Theorem~\ref{thm:decomposition_F_cographs}. 
	Once again we write $G = \sum_{i=1}^k a_iG_i$ for some $k\in \N$ be the decomposition of $G$ into connected components. 
	By Lemma~\ref{lem:coF_connected_component}, we have that $G_i \in \co(\Ff)$ for all $1\leq i\leq k$. 
	Since $\co(\Ff)$ satisfies (QI) by Theorem~\ref{thm:tractable_F_cographs}, by Theorem~\ref{thm:aut_sums}, we have that $\Qu(G) = *_{i=1}^k \Qu(G_i) \wr S_{a_i}^+$. 
	Since $G$ is disconnected, we have that $\abs{V(G_i)} < \abs{V(G)} = n+1$ for all $1\leq i\leq k$, so, by induction hypothesis, we have that $\Qu(G_i) \in \Jor(\Qu(\Ff_{conn}))$. 
	Since by definition this class is closed under taking free and wreath products with $S_l^+$ for $l\geq 1$, this shows that $\Qu(G) \in \Jor(\Qu(\Ff_{conn}))$, as desired. 
	This concludes the proof.
\end{proof}

\begin{theorem}\label{thm:quantum_K1_vs_Jor}
	Let $\Ff$ be a class of graphs such that $K_1\in \Ff$. 
	Then $\Jor(\Qu(\Ff)) \subset \Qu(\co(\Ff))$. 
\end{theorem}

\begin{proof}
	It is enough to show that $\Qu(\co(\Ff))$ contains $\Qu(\Ff)$ (which is clear), is stable by direct product, and by wreath product with $S_d$ for $d\in \N$. 

	Let $G\in \co(\Ff)$ and $d\in \N$. 
	We claim that $\Qu(G) \wr S_d \in \Qu(\co(\Ff)$. 
	Indeed, if $G$ is connected, set $H = d.G$, otherwise, set $H = d.G^c$. 
	Now $H\in \co(\Ff)$, and, by Theorem~\ref{thm:aut_sums}, since $\Qu(G) = \Qu(G^c)$, we have that $\Qu(G) \wr S_d^+ = \Qu(H) \in \Qu(\co(\Ff))$, as desired. 

	Then take $G$ and $H\in \co(\Ff)$. 
	We claim that $\Qu(G) * \Qu(H) \in \Qu(\co(\Ff))$. 
	Indeed, notice that up to taking complements we can assume both $G$ and $H$ connected. 
	If $G\neq H$, set $L = G+H \in \co(\Ff)$. 
	We have $\Qu(G)\times \Qu(H) = \Qu(L) \in \Qu(\co(\Ff))$ by Theorem~\ref{thm:aut_sums}, as desired. 
	Assume now that $G=H$. 
	Since $K_1\in \Ff$, we have that $L' = (G+K_1)^c \in \co(\Ff)$. 
	Notice that $L'\neq G$ (since they do not have the same number of vertices), that $L'$ is connected by Lemma~\ref{lem:complement}, and that $\Qu(L') = \Qu(G+K_1) = \Qu(G) * \Qu(K_1) = \Qu(G)$ since $G$ is connected. 
	Setting $L=G+L'$, we have that $\Qu(G)^2 = \Qu(G) * \Qu(L') = \Qu(L) \in \Qu(\co(\Ff))$, as desired. 
	
	Hence we have shown that $\Qu(\co(\Ff))$ is stable by free product, by wreath product with $S_d^+$ for $d\geq 1$, and contains $\Qu(\Ff_{conn})$. 
	By definition, we have $\Jor(\Qu(\Ff)) \subset \Qu(\co(\Ff))$, as desired. 
	This concludes the proof.
\end{proof}

Hence we obtain the quantum automorphism groups of $\Ff$-cographs as a function of the quantum automorphism groups of graphs of $\Ff$.

\begin{theorem}\label{thm:quantum_Jor_vs_Fcographs}
	Let $\Ff$ be a class of graphs. 
	Assume that:
	\begin{enumerate}
		\item $K_1\in \Ff$,
		\item $\Ff$ satisfies (QI),
		\item $\Ff$ is stable by taking complements,
		\item $\Ff$ is stable by taking connected components.
	\end{enumerate}
	Then we have that $\Qu(\co(\Ff)) = \Jor(\Qu(\Ff_{conn})) = \Jor(\Qu(\co(\Ff)))$.
\end{theorem}

\begin{proof}
	By Theorem~\ref{thm:quantum_Jor_F_cographs}, we have that $\Qu(\co(\Ff)) \subset \Jor(\Qu(\Ff_{conn}))$. 
	Since $\Ff_{conn} \subset \Ff$, we have that $\Jor(\Qu(\Ff_{conn})) \subset \Jor(\Qu(\Ff))$. 
	By Theorem~\ref{thm:quantum_K1_vs_Jor}, we have that $\Jor(\Qu(\Ff)) \subset \Qu(\co(\Ff))$. 
	This concludes the proof.
\end{proof}

The precedent results are very practical for computations by hand because in many cases it is possible to decompose a graph by alternatively taking complements and connected components into much smaller graphs.  
We conclude this section by giving two examples, the second one generalising the first one. 
In the next section, we will use the precedent results to compute the quantum (and classical) automorphism groups of tree-cographs, thus generalising Jordan's theorem about automorphism groups of trees (see~\cite{jordan1869assemblages}) to the noncommutative case. 

We start by computing the quantum and classical automorphism groups of cographs. 

\begin{theorem}\label{thm:qu_cographs}
	Let $\Ff$ be the family of cographs. 
	We have:
	\begin{itemize}
		\item $\Aut(\Ff) = \Jor(\{1\})$,
		\item $\Qu(\Ff) = \Jor(\{1\})$.
	\end{itemize}
\end{theorem}

\begin{proof}
	Notice that cographs are by definition the $\{K_1\}$-cographs. 
	This immediately implies the desired results by Theorem~\ref{thm:Jor_vs_Fcographs} and Theorem~\ref{thm:quantum_Jor_vs_Fcographs}.
\end{proof}

We now extend the precedent result. 
Let $\Gg_n = \{ G\in \Graphs \mid \# V(G) \leq n\}$. 

\begin{theorem}\label{thm:G5}
	The class $\co(\Gg_5)$ is tractable, superrigid, and $\Qu(\co(\Gg_5)) = \Jor(\Qu(\Gg_5))$.
\end{theorem}

\begin{proof}
	By Theorem~\ref{thm:tractable_G5}, we have that $\Gg_5$ is tractable. 
	Since quantum isomorphism preserves the number of vertices, $\Gg_5$ is also superrigid. 
	We then have that $\co(\Gg_5)$ is tractable by Theorem~\ref{thm:tractable_F_cographs} and is superrigid by Theorem~\ref{thm:superrigid_F_cographs}. 
	Now it is clear that $\Gg_5$ satisfies the hypothesis of Theorem~\ref{thm:quantum_Jor_vs_Fcographs}, which implies that $\Qu(\co(\Gg_5)) = \Jor(\Qu(\Gg_5))$, as desired. 
	This concludes the proof.
\end{proof}

\begin{remark}\label{rk:coG5_cotree}
	We point out that the class of $\Gg_5$-cographs is distinct from the class of tree-cographs, hence Theorem~\ref{thm:G5} and our results in the next section (Theorem~\ref{thm:tractable_tree_cographs} and Theorem~\ref{thm:jordan_tree_cographs}) can really be seen as two distinct generalisations of our results for $K_1$-cographs, that is, cographs. 
	Indeed, we have $C_5 \in \co(\Gg_5) \setminus \co(\Trees)$ and $P_6 \in \co(\Trees) \setminus \co(\Gg_5)$.
\end{remark}

\section{Forests}\label{sec:forests}

In this section, we prove that the class of forests is tractable and we compute the quantum automorphism groups of forests. 
Doing so, we extend to the noncommutative setting a theorem of Jordan~\cite{jordan1869assemblages}. 
The present work is in the continuation of work by Fulton~\cite{Fulton2006} on quantum automorphism groups of trees. 

Following Fulton's approach, we use the center of a tree in a key way. 
We notice that it naturally leads to a structure of rooted tree on a tree which better suited for computations. 
By translating the problems to rooted trees and rooted forests, their solutions appear thus naturally as a consequence of a natural inductive structure for which we can apply our previous results. 
This allows us to obtain the desired results.

We want to point out that our approach has the advantage of being explicit. 
We will mention when some results can also be obtained in a non-explicit way.

In subsection~\ref{subsec:first_trees}, we recall Fulton's proof that trees satisfy (QA), and give an non-explicit proof that trees satisfy (QI) -- both results are later proved in an explicit manner. 
In subsection~\ref{subsec:rooted_forests}, we introduce the language of rooted forests, and define their quantum automorphism groups. 
In subsection~\ref{subsec:quantum_rooted}, we use the natural inductive structure of rooted forests to prove they form a tractable class (with axioms adapted to the rooted setting), and compute their quantum automorphism groups. 
In subsection~\ref{subsec:psi}, we define a transformation associating a rooted tree to a tree and use it to prove in Theorem~\ref{thm:tractable_forests} that forests are tractable, and in Corollary~\ref{coro:superrigid_tree_cographs} that tree-cographs are superrigid. 
In subsection~\ref{subsec:noncommutative_jordan}, we compute the quantum automorphism groups of trees and obtain the noncommutative Jordan theorem in Theorem~\ref{thm:noncommutative_jordan}. 
We finish by proving that the class of tree-cographs is tractable in Theorem~\ref{thm:tractable_tree_cographs} and computing their quantum automorphism groups in Theorem~\ref{thm:jordan_tree_cographs}.

We point out that the classification of quantum automorphism groups of finite trees in Theorem~\ref{thm:noncommutative_jordan} below was recently independently obtained by De Bruyn, Kar, Roberson, Schmidt, and Zeman in~\cite{de2023quantum} through a direct computational approach. 
Here, we will obtain this result and its generalisation to tree-cographs in Theorem~\ref{thm:jordan_tree_cographs} will come up as a consequence of our more general approach.

\subsection{First results on trees}\label{subsec:first_trees}

A \textit{tree} is a finite graph such that there is a unique path joining any two vertices. 
A \textit{forest} is a graph all of whose connected components are trees. 
We point out that forests are stable by taking subgraphs, while trees are not.

In this subsection, we will focus on trees and recall from Fulton~\cite{Fulton2006} that the class of trees satisfies axiom (QA), that is, that an asymmetric tree has a trivial quantum automorphism group. 
This is the first step into understanding the structure of trees. 
We reproduce the proof here both for the sake of completeness and because, the results of Fulton being very early in the theory (the definition of the quantum automorphism group of a graph was not stabilised yet), the modern terminology simplifies the exposition. 

Let $G$ be a graph and let $\Pp$ be a partition of $V(G)$. 
In 1971, Weichsel~\cite{Weichsel1971} calls such a partition a \textit{star partition} if:
\begin{enumerate}
	\item for all $A\in \Pp$, for all $x$, $y\in A$, we have $deg(x) = deg(y)$,
	\item for all $A$, $B\in \Pp$, for all $x\in A$, if there exists $y\in B$ such that $xy \in E(G)$, then for all $x' \in A$ there exists $y'\in B$ such that $x'y'\in E(G)$.
\end{enumerate}

His motivation comes from the two following results, the second one being the main result of his paper~\cite{Weichsel1971}.

\begin{lemma}[Weichsel]\label{lem:orbits_star}
	Let $\alpha$ be an automorphism of a graph $G$. 
	Then the orbits of $\alpha$ form a star partition.
\end{lemma}

The main result of~\cite{Weichsel1971} is the following.

\begin{theorem}[Weichsel]\label{thm:asymmetric_tree}
	A tree is asymmetric if and only if every star partition of its vertices is trivial.
\end{theorem}

The idea of Fulton~\cite{Fulton2006} is to extend both results to the quantum setting, that is, to magic unitaries. 
For coherence with modern terminology, we rephrase her result in Lemma~\ref{lem:qu_orbits_star}, though the technique used here is identical to the one she uses. 
Actually, she proves a stronger result than the one stated here, see in particular Lemma 4.2.1 of~\cite{Fulton2006}. 

\begin{lemma}[Fulton]\label{lem:qu_orbits_star}
	Let $G$ be a graph. 
	The quantum orbits of $\Qu(G)$ form a star partition.
\end{lemma}

\begin{proof}
	We write $\Qu(G) = (\Oo(\Qu(G)),u)$. 
	Let $\Pp$ be the partition of quantum orbits of $\Qu(G)$. 
	By Lemma~\ref{lem:degree}, it satisfies the first axiom of star partitions. 
	So let $A$ and $B\in \Pp$ be two quantum orbits. 
	Assume there is $x \in A$ and $y\in B$ such that $xy \in E(G)$ and let $a \in A$. 
	It comes:
	\begin{align*}
		0 &\neq u_{xa} = u_{xa}\left( \sum_{b\in V(G)} u_{yb} \right)\\
		&= \sum_{b\in B} u_{xa}u_{yb},
	\end{align*}
	so there exists $b\in B$ such that $u_{xa}u_{yb} \neq 0$. 
	By Lemma~\ref{lem:distance}, we have that $d(a,b) = d(x,y) = 1$, so $ab\in E(G)$. 
	This shows that $\Pp$ also satisfies the second axiom of star partitions, and conclude the proof.
\end{proof}

Together with Theorem~\ref{thm:asymmetric_tree}, Fulton obtains the following (Theorem 4.4.3 of~\cite{Fulton2006}).

\begin{theorem}[Fulton]\label{thm:qa_for_trees}
	A tree is asymmetric if and only if its quantum automorphism group is trivial.
\end{theorem}

\begin{proof}
	Let $T$ be a tree. 
	First, notice that if $\Qu(T)$ is trivial, then every magic unitary adapted to $T$ is the identity matrix. 
	In particular, every permutation matrix commuting with $\Adj(T)$ is the identity. 
	By Corollary~\ref{coro:automorphism}, we have $\Aut(T) = 1$.

	Conversely, assume $T$ is asymetric. 
	Let $\Pp$ be the partition of $V(G)$ given by the quantum orbits. 
	By Lemma~\ref{lem:qu_orbits_star}, we have that $\Pp$ is a star partition. 
	By Theorem~\ref{thm:asymmetric_tree}, we have that $\Pp$ is trivial. 
	This immediately implies that the fundamental representation of $\Qu(T)$ is given by the identity matrix. 
	In particular, $\Oo(\Qu(T)) = \C$, and $\Qu(T) = 1$. 
	This concludes the proof.
\end{proof}

In our setting, Fulton thus showed that the class of trees satisfies the axiom (QA). 
Let us now show in an non-explicit way that quantum isomorphic trees are isomorphic, that is, axiom (QI). 
We use the notion of fractional isomorphism. 
Two graphs $G$ and $H$ on $n\geq 1$ vertices are \textit{fractionally isomorphic} if there exists $S \in \Mm_n([0,1])$ bistochastic (that is, the sum on rows and columns is 1) such that $S\Adj(G) = \Adj(H)S$. 
It is well-known that fractional isomorphic trees are isomorphic, so it is enough to check that quantum isomorphic graphs are fractionnally isomorphic. 
This is already obtained indirectly in~\cite{QiNotI} where it is shown that quantum isomorphism implies quantum nonsignalling isomorphism, the last notion being equivalent to fractional isomorphism (Theorem 4.5 in~\cite{QiNotI}). 
We give here a direct proof.

\begin{lemma}\label{lem:qi_imiplies_fi}
	Two quantum isomorphic graphs are fractionally isomorphic.
\end{lemma}

\begin{proof}
	Let $G$ and $H$ be two quantum isomorphic graphs on $n \geq 1$ vertices. 
	By Theorem 4.4 and Theorem 2.5 of~\cite{LupManRob}, there exists a tracial unital $C^*$-algebra $(A,\tau)$ and a quantum isomorphism $U =(u_{ij})_{1\leq i,j\leq n}$ from $G$ to $H$ with coefficients in $A$. 
	Now the matrix $\tau(U) = \tau(u_{ij})_{1\leq i,j\leq n}$ is the desired bistochastic matrix.
\end{proof}

In another degree of generality, we will show in Appendix~\ref{app:lovasz} that quantum isomorphic planar graphs are isomorphic, by using the result of Mančinska and Roberson~\cite{ManRob} (see Corollary~\ref{coro:qi_planar}). 

\subsection{Rooted forests}\label{subsec:rooted_forests}

A \textit{rooted tree} is a pointed tree, i.e. a pair $(T,r)$ where $T$ is a tree and $r \in V(T)$. 
The vertex $r$ is referred to as the \textit{root} of $T$. 
Morphisms of rooted trees are graph morphisms between the underlying trees which preserve the roots. 
A \textit{rooted forest} is a disjoint union of rooted trees, we typically denote a rooted forest by $(F,\{r_1,\ldots,r_k\}) = (T_1,r_1) + \ldots + (T_k,r_k)$, where $\{r_1,\ldots,r_k\}$ is the set of roots and $T_1,\ldots,T_k$ are the connected components of $F$. 
A morphism of rooted forests is a morphism of the underlying graphs such that every root is sent to a root. 
For $(F,r)$ a rooted forest ($r$ is here a set of vertices of $F$), this leads to the natural definition of $\Aut(F,r)$, the automorphism group of the rooted forest $(F,r)$. 

Let us define the quantum automorphism group $\Qu(F,r)$ of the rooted forest $(F,r)$. 
Let $n$ be the number of vertices of $F$ and assume we have fixed an ordering on $V(F)$. 
Recall from Section~\ref{sec:preli} the relations $\Pp_s$ and $R_F$ on $n^2$ variables which are the defining relations of a magic unitary adapted to $F$. 
Define then $R_r = \{ x_{st} \mid s\in r,\ t\notin r\}$: that is, if $x$ satisfies $R_r$, then $x_{st}=0$ when $s$ is a root and $t$ is not. 
We claim the following.

\begin{lemma}\label{lem:def_qu_rooted_forest}
	Let $(A,x)$ be the universal $*$-algebra generated on $n^2$ variables $x = (x_{ij})_{1\leq i,j\leq n}$ satisfying the relations $\Pp_s \cup R_F \cup R_r$. 
	Then $(A,x)$ satisfies the axioms of a quantum permutation group.
\end{lemma}

\begin{proof}
	Let $y_{ij} = \sum_{k=1}^n x_{ik}\otimes x_{kj}$ in $A\otimes A$ for all $1\leq i,j\leq n$. 
	By Lemma~\ref{lem:qu_subgroups}, it is enough to prove that $y = (y_{ij})_{1\leq i,j\leq n}$ satisfies the relations $\Pp_s \cup R_F \cup R_r$. 
	By Theorem~\ref{thm:qu_G}, we already know that $y$ satisfies $\Pp_s\cup R_F$, so it remains to check that $y$ satisfies $R_r$. 
	Let $s\in r$ and $t\notin r$. 
	It comes:
	\begin{align*}
		y_{st} &= \sum_{z\in V(F)} x_{sz}\otimes x_{zt}\\
		&= \sum_{z\in r} x_{sz}\otimes x_{zt}\ \text{since $x_{sz} = 0$ when $z\notin r$}\\
		&= 0
	\end{align*}
	since $x_{zt} = 0$ when $z\in r$. 
	So $y$ satisfies $R_r$ and the result follows from Lemma~\ref{lem:qu_subgroups}.
\end{proof}

We define $\Qu(F,r) = (A,x)$ with $A$ and $x$ as in Lemma~\ref{lem:def_qu_rooted_forest}. 
Notice that since a rooted tree is in particular a rooted forest, this also defines the quantum automorphism group of a rooted tree. 

We introduce rooted forests not to study them for themselves but as useful tools to study forests and trees. 
In that spirit, we extend to rooted forests some terminology and results for graphs.

A magic unitary $U$ is \textit{adapted} to a rooted tree $(T,r)$ if it is adapted to the tree $T$ and if $u_{rr} = 1$. 
Two rooted trees $(T_1,r_1)$ and $(T_2,r_2)$ are \textit{quantum isomorphic} if there is a magic unitary $U$ such that $U\Adj(T_1) = \Adj(T_2) U$ (so in particular $T_1$ and $T_2$ are quantum isomorphic) and $u_{r_2r_1} = 1$. 

A magic unitary $U$ is \textit{adapted} to $F$ if it is adapted to $F$ and it preserves the roots in the sense that for $1\leq i\leq k$ and $x\in V(F)$, if $u_{r_ix} \neq 0$ or $u_{xr_i} \neq 0$, then there exists $1\leq j\leq k$ such that $x = r_j$. 
Two rooted forests $(F_1,x)$ and $(F_2,y)$, where $x$ and $y$ are sets of vertices of $F_1$ and $F_2$, are \textit{quantum isomorphic} if there exists a magic unitary $U$ such that it conjugates the adjacency matrices of $F_1$ and $F_2$ (hence $F_1$ and $F_2$ are quantum isomorphic) and such that $U$ preserves the roots, i.e. such that $u_{rs}=0$ if $r\in x$ and $s\notin y$ or $r\notin x$ and $s\in y$.

We extend naturally quantum symmetry, the (quantum) Schmidt criterion, and the Schmidt alternative for rooted forests. 
Let us check that the rooted quantum Schmidt criterion implies quantum symmetry of the rooted forest.

\begin{lemma}\label{lem:rooted_qu_schmidt}
	Let $(F,r)$ be a rooted forest. 
	Assume that there are two nontrivial (i.e. $\neq 1$) magic unitaries $U$ and $V$ adapted to $(F,r)$ with disjoint support. 
	Then $(F,r)$ has quantum symmetry.
\end{lemma}

\begin{proof}
	It is enough to notice that the same proof as the one of the quantum Schmidt criterion (Theorem~\ref{thm:qu_schmidt}) works, with the extra observation that, if we start with magic unitaries $U$ and $V$ preserving the roots, then the magic unitary $W = \Diag(i_X(U_1),i_Y(V_1),1)$ also preserves the roots.
\end{proof}

\subsection{Quantum properties of rooted forests}\label{subsec:quantum_rooted}

As mentioned in the previous section, we view rooted forests as a tool in order to describe the quantum structure of forests themselves. 
To achieve it, we need first to understand the quantum structure of rooted forests. 
Hence for the current section, we focus on rooted forests themselves.

Let $(F,r)$ be a rooted forest. 
We can write $(F,r) = \sum_{i=1}^k (T_i,r_i)$ with $k\geq 1$ the number of connected components of $F$ and where $(T_i,r_i)$ is a rooted tree for all $1\leq i\leq k$. 
Let $T$ be the tree obtained by adding a vertex $x$ to $F$ and adding an edge $\{x, r_i\}$ for all $1\leq i\leq k$. 
Taking $x$ to be the root, we obtain a rooted tree $(T,x)$. 
Since this procedure will often come up, we will denote it by $(T,x) = x. \sum_{i=1}^k (T_i,r_i)$, or also $(T,x) = x. \left( (T_1,r_1) + \ldots + (T_k,r_k) \right)$. 

Now let $(T,r)$ be a rooted tree on at least two vertices. 
Then $T\setminus \{r\}$ is naturally a rooted forest, where we take $N(r)$ as our new set of roots. 
Since a rooted forest is a sum of rooted trees, this gives us an inductive structure on rooted trees that we want to exploit in order to determine the quantum properties of rooted trees. 
This is done by the following essential lemma.

\begin{lemma}\label{lem:rooted_forest_vs_rooted_tree}
	Let $(T,r)$ be a rooted tree on at least two vertices and let $(F,\{r_i\}) = T\setminus \{r\}$ be the rooted forest obtained by removing the root. 
	Let $U$ be a magic unitary adapted to the rooted tree $(T,r)$. 
	Then $U$ is of the form
	\[ U = \begin{pmatrix}
		V & 0\\
		0 & 1
	\end{pmatrix}\]
	when enumerating the vertices with $r$ being the last one, and $V$ is a magic unitary adapted to $(F,(r_i))$. 
	Conversely, if $V$ is a magic unitary adapted to the rooted forest $(F,\{r_i\})$, then $\Diag(V,1)$ is a magic unitary adapted to $(T,r)$.
\end{lemma}

\begin{proof}
	Let $U$ be a magic unitary adapted to $(T,r)$. 
	By definition, $U = \Diag(V,1)$, where the 1 corresponds to $r$, $V$ is a magic unitary, and $U$ commutes with $\Adj(T)$. 
	Notice that we have
	\[ \Adj(T) = \begin{pmatrix}
		\Adj(F) & L^*\\
		L & 0
	\end{pmatrix},\]
	with $L$ a certain line matrix. 
	Since $[U,\Adj(T)]= 0$, we find that $[V,\Adj(F)] = 0$, so $V$ is a magic unitary adapted to the forest $F$. 
	All that is left to check is that $V$ preserves the roots of $F$, that is, $N(r)$. 
	Let $x\in N(r)$ and $y\in V(T)$ such that $u_{xy}\neq 0$. 
	Then $0 \neq u_{xy}1 = u_{xy}u_{rr}$. 
	By Lemma~\ref{lem:distance}, we have that $d(y,r) = d(x,r) = 1$, so $y\in N(r)$, as desired. 
	This concludes the proof of the first part of the lemma.

	For the second part, we consider $V$ a magic unitary adapted to $(F,(r_i))$ and we set $U = \Diag(V,1)$. 
	We only need to check that $U$ commutes with $\Adj(T)$. 
	Using the presentation of $\Adj(T)$ as above, it is clear that this happens if $LV=L$. 
	Notice that for $x \in V(F)$, writing $l_x = [L]_x$, we have $l_x = 1_{N(r)}(x)$. 
	Let $x \in N(r)$. 
	Since $V$ is adapted to $(F,(r_i))$, and since $\{r_i\} = N(r)$, we have that $v_{yx} = 0$ for $y\notin N(r)$. 
	It comes:
	\begin{align*}
		[LV]_x &= \sum_{y\in V(F)} l_yv_{yx}\\
		&= \sum_{y\in N(r)} v_{yx}\\
		&= \sum_{y \in V(F)} v_{yx}\\
		&=1\\
		&= l_x.
	\end{align*}
	Now consider $x\in V(F) \setminus N(r)$. 
	Similarly, we obtain:
	\begin{align*}
		[LV]_x &= \sum_{y\in V(F)} l_yv_{yx}\\
		&= \sum_{y\in N(r)} v_{yx}\\
		&=0\\
		&= l_x.
	\end{align*}
	Hence, we have that $LV = L$, and $U$ commutes with $\Adj(T)$, as desired. 
	This concludes the proof.
\end{proof}

From Lemma~\ref{lem:rooted_forest_vs_rooted_tree}, we obtain the following theorem.

\begin{theorem}\label{thm:rooted_forest_vs_rooted_tree}
	Let $(T,r)$ be a rooted tree on at least two vertices. 
	Then we have:
	\begin{itemize}
		\item $\Qu(T,r) = \Qu(T\setminus\{r\})$, where both quantum groups are taken with the respect to the rooted structure,
		\item $(T,r)$ satisfies the rooted Schmidt criterion if and only if the rooted forest $T\setminus \{r\}$ does,
		\item for $(T,x)$ and $(S,y)$ two rooted trees on at least two vertices, we have $(T,x) \qi (S,y)$ if and only if $T\setminus \{x\} \qi S\setminus\{y\}$ as rooted forests.
	\end{itemize}
\end{theorem}

\begin{proof}
	The first point follows immediately from Lemma~\ref{lem:rooted_forest_vs_rooted_tree}. 
	For the second point, simply notice that classical automorphisms are also given by magic unitaries (specifically, the ones with 1-dimensional coefficients) and that writing $U = \Diag(V,1)$ we have $\Supp(U) = \Supp(V)$.

	Let us prove the last point. 
	First, take $V$ a quantum isomorphism between $T\setminus \{x\}$ and $S\setminus \{y\}$. 
	It is clear that $\Diag(V,1)$ is a magic unitary establishing a quantum isomorphism between $(T,x)$ and $(S,y)$, with the root enumerated last.

	Conversely, consider $U$ a quantum isomorphism from $(T,x)$ to $(S,y)$. 
	By definition, $U$ is of the form $\Diag(V,1)$ with $V$ a magic unitary. 
	Write $(F_1,r_1) = T\setminus \{x\}$ and $(F_2,r_2) =  S\setminus \{y\}$. 
	A direct computation shows that $V$ conjugates the adjacency matrices of $F_1$ and $F_2$. 
	Let us check that $u_{ab} = 0$ when $a \notin r_2$ and $b\in r_1$. 
	Let $a \in V(S)$ and $b\in V(T)$. 
	Recall that $r_1 = N(x)$ and $r_2 = N(y)$. 
	Assume that $b \in N(y)$ and that $u_{ab}\neq 0$. 
	Now $u_{xy}u_{ab} = u_{ab} \neq 0$ so by Lemma~\ref{lem:distance} we have that $d_T(x,a) = d_S(y,b) = 1$, hence $a\in N(x) = r_1$, as desired. 
	This shows that $u_{ab} = 0$ if $a\notin r_2$ and $b\in r_1$. 
	To obtain that $u_{ab} = 0$ when $a\in r_2$ and $b\notin r_1$, apply what precedes to $U^*$, exchanging the roles of $(T,x)$ and $(S,y)$. 
	This concludes the proof.
\end{proof}

We are now able to adapt Theorem~\ref{thm:qi_for_connected_components} and Theorem~\ref{thm:qi_for_disconnected_graphs} to rooted forests.

\begin{theorem}\label{thm:qi_rooted_trees}
	Let $(F_1,R_1)$ and $(F_2,R_2)$ be two rooted forests and let $U$ be a quantum isomorphism from $(F_1,R_1)$ to $(F_2,R_2)$ with coefficients in a unital $C^*$-algebra $X$. 
	Let $x \in V(F_1)$ and $a \in V(F_2)$ such that $u_{xa} \neq 0$. 
	Set $V = U[C(x),C(a)]$, and let $r$ be the root of $C(x)$, and $s$ be the root of $C(a)$. 
	Then there exists a nonzero projection $p\in X$ such that:
	\begin{enumerate}
		\item for every $y\in C(x)$, we have $\sum_{b\in V(C(a))} v_{yb} = p$,
		\item for every $b\in C(a)$, we have $\sum_{y\in V(C(x))} v_{yb} = p$,
		\item $V\Adj(C(a)) = \Adj(C(x))V$,
		\item $V$ preserves the roots of $C(x)$ and $C(a)$.
	\end{enumerate}
	In particular, $V$ is a square matrix, and it is a quantum isomorphism of rooted trees from $(C(x),r)$ to $(C(a),s)$ with coefficients in the $C^*$-subalgebra of $X$ generated by the coefficients of $V$.
\end{theorem}

\begin{proof}
	By Theorem~\ref{thm:qi_for_connected_components}, we already have the existence of $p$ and items (1), (2), and (3), as well as the fact that $V$ induces a quantum isomorphism of graphs between the trees $C(x)$ and $C(a)$. 
	Hence, all there is to check is that $V$ preserves the roots. 
	But this is immediate since $U$ does. 
	This concludes the proof.
\end{proof}

\begin{theorem}\label{thm:qi_rooted_forests}
	Let $(F_1,R_1) = \sum_{i=1}^k (T_i,x_i)$ and $(F_2,R_2) = \sum_{j=1}^l (S_j,y_j)$ be two rooted forests. 
	Assume that they are quantum isomorphic. 
	Then $k=l$ and up to relabelling we have that $(T_i,x_i)$ is quantum isomorphic to $(S_i,y_i)$ for all $1\leq i\leq k$.
\end{theorem}

\begin{proof}
	Since quantum isomorphism of $(F_1,R_1)$ and $(F_2,R_2)$ as rooted forests implies quantum isomorphism of $F_1$ and $F_2$ as graphs, we already have by Theorem~\ref{thm:qi_for_disconnected_graphs} that $k=l$ and that up to relabelling we get that $T_i$ is quantum isomorphic to $S_i$. 
	So we just need to prove that such quantum isomorphisms can be taken with $u_{y_ix_i} = 1$.

	Let $U$ be a quantum isomorphism between $(F_1,R_1)$ and $(F_2,R_2)$ with coefficients in a unital $C^*$-algebra $X$. 
	We write it as a block matrix $(U_{ij})$ where each block is indexed by vertices in $V(S_i)\times V(T_j)$. 
	By Lemma~\ref{lem:p_ij} we have a magic unitary $P = (p_{ij}(U))_{1\leq i,j\leq k} = (p_{ij})_{1\leq i,j\leq k}$. 
	Applying Lemma~\ref{lem:hall_mu} and relabelling if necessary, we can assume that $p_{ii} \neq 0$ for all $1\leq i\leq k$. 
	Fix $1\leq i\leq k$. 
	Let $X_i$ be the $C^*$-subalgebra of $X$ generated by the coefficients of $U_{ii}$. 
	By Theorem~\ref{thm:qi_rooted_trees}, we know that $X_i$ is unital wit unit $p_{ii}(U)$, that $U_{ii}$ is a square matrix, and that it is a quantum isomorphism of rooted trees from $(T_i,x_i)$ to $(S_i,y_i)$ when viewed with coefficients in $X_i$. 
	This concludes the proof.
\end{proof}

\begin{remark}\label{rk:after_qi_rooted_forests}
	The converse of Theorem~\ref{thm:qi_rooted_forests} is obviously true.
\end{remark}

We are ready to show that rooted forests satisfy (QI), that is, that quantum isomorphic rooted forests are isomorphic.

\begin{theorem}\label{thm:qi_equals_i_for_rooted_forests}
	Two rooted forests are quantum isomorphic if and only if they are isomorphic.
\end{theorem}

\begin{proof}
	Let $(F_1,R_1)$ and $(F_2,R_2)$ be two quantum isomoprhic rooted forests. 
	We need to prove that they are isomorphic. 
	Since they are quantum isomorphic, we already have that $F_1$ and $F_2$ are quantum isomorphic as graphs, hence they have the same number of vertices $n\in \N$. 
	We will prove it by induction on $n\geq 1$. 
	For $n=1$, this is immediate. 
	So we assume that quantum isomorphism implies isomorphism for any rooted forest on at most $n-1$ vertices for some $n\geq 2$. 
	Let us write $(F_1,R_1) = (T_1,x_1) + \ldots + (T_k,x_k)$. 
	By Theorem~\ref{thm:qi_rooted_forests}, we have that $F_2$ has $k$ connected components as well, that we can enumerate as $(F_2,R_2) = (S_1,y_1) + \ldots + (S_k,y_k)$ such that for any $1\leq i\leq k$ we have that $(T_i,x_i)$ is quantum isomorphic to $(S_i,y_i)$. 
	
	Let us first assume that $k>1$. 
	In that case, for $1\leq i\leq k$, we have $\# V(T_i) < n$, so by induction hypothesis $(T_i,x_i)$ and $(S_i,y_i)$ are isomorphic. 
	This being true for all $i$, we have that $(F_1,R_1) = \sum_i (T_i,x_i) = \sum_i (S_i,y_i) = (F_2,R_2)$, so $(F_1,R_1)$ and $(F_2,R_2)$ are isomorphic, as desired.

	Now assume that $k=1$. 
	Then $(F_1,R_1) = (T,x)$ and $(F_2,R_2) = (S,y)$ are rooted trees. 
	By Theorem~\ref{thm:rooted_forest_vs_rooted_tree}, the rooted forests $T\setminus \{x\}$ and $S\setminus \{y\}$ are quantum isomorphic, so they are isomorphic by induction hypothesis. 
	But this implies that $(T,x)$ and $(S,y)$ are isomorphic as well, since they are obtained by adding a root respectively to $T\setminus \{x\}$ and $S\setminus\{y\}$, which concludes the induction.
\end{proof}

These results now allow us to compute the quantum automorphism group of a forest as a function of the quantum automorphism group of its connected components in Theorem~\ref{thm:qu_rooted_forest}. 
We start with a few prepartory lemmas.

\begin{lemma}\label{lem:qu_rooted_forest}
	Let $k\geq 1$ and let $(F,r) = \sum_{i=1}^k a_i.(T_i,r_i)$ be a rooted forest, where $a_1,\ldots,a_k\in \N$, and $(T_1,r_1),\ldots,(T_k,r_k)$ are two-by-two non isomorphic rooted trees. 
	Then $\Qu(F,r) = *_{i=1}^k \Qu(a_i.(T_i,r_i))$. 
\end{lemma}

\begin{proof}
	Let $U$ be the fundamental representation of $\Qu(F,r)$. 
	We claim that it is diagonal by block, where each block is adapted to $a_i(T_i,r_i)$, which will imply the result.

	Indeed, let $1\leq i,j\leq k$ and let $x$, $a\in V(F)$. 
	Assume that $x$ is a vertex in some copy of $(T_i,r_i)$ and $a$ is a vertex in some copy of $(T_j,r_j)$. 
	Assume moreover that $u_{xa}\neq 0$. 
	Now by Theorem~\ref{thm:qi_rooted_trees} we have that $(T_i,r_i)$ is quantum isomorphic to $(T_j,r_j)$, hence $i=j$ by Theorem~\ref{thm:qi_equals_i_for_rooted_forests}. 
	This concludes the proof.
\end{proof}

We then extend Theorem~\ref{thm:wreath_product}.

\begin{lemma}\label{lem:wreath_rooted_tree}
	Let $(T,x)$ be a rooted tree and let $d\geq 1$. 
	Then $\Qu(d.(T,x)) = \Qu(T,x) \wr S_d^+$.
\end{lemma}

\begin{proof}
	Let $n = \# V(T)$. 
	Let $U$ be the fundamental representation of $\Qu(d.(T,x))$. 
	Let $W$ be the fundamental representation of size $dn$ of $\Qu(T,x)\wr S_d^+$ obtained by Section 7.2.2 of~\cite{FreslonBook}. 
	
	Let $W$ be the fundamental representation of size $dn$ of $\Qu(T,x)\wr S_d^+$ obtained by Section 7.2.2 of~\cite{FreslonBook}. 
	Recall that it is given by coefficients $w_{ij,kl} = \nu_i(v_{kl})p_{ij}$, with $1\leq i,j\leq d$ and $1\leq k,l\leq n$, where $W$ is in a block form $W = (W_{ij})_{1\leq i,j\leq d}$, where each block is indexed by a copy of $T$. 
	Here $P = (p_{ij})_{1\leq i,j\leq d}$ is the fundamental representation of $S_d^+$ and $V = (v_{kl})_{1\leq k,l\leq n}$ is the fundamental representation of $\Qu(T,x)$. 
	In particular, $V$ is a magic unitary adapted to $T$. 
	By Theorem~\ref{thm:wreath_product}, this implies that $W$ commutes with $\Adj(d.T)$ (since $W$ is then the image of the fundamental representation of $\Qu(d.T)$, which commutes with the scalar matrix $\Adj(T)$). 
	Moreover, we claim that $W$ preserves the roots: indeed, assume that $x$ is enumerated first in $V(T)$, let $1\leq i,j\leq d$, and let $2\leq k\leq n$. 
	Since $V$ is adapted to $(T,x)$, we have $w_{ij,1k} = \nu_i(v_{1k})p_{ij} = 0$ since $v_{1k} = 0$, and similarly $v_{k1} = 0$ so $w_{ij,k1} = 0$. 
	This shows that $W$ is a magic unitary adapted to the rooted forest $d.(T,x)$. 
	Hence by universality the mapping $u_{ij,kl}\mapsto w_{ij,kl}$ extends to a surjective $*$-morphism from $\Oo(\Qu(d.(T,x))$ to $\Oo(\Qu(T,x)\wr S_d^+)$. 

	Let us show the converse. 
	For this, we know that $U$ is a magic unitary adapted to $d.T$, so, by Theorem~\ref{thm:wreath_product}, we have that $U$ is of the form $U = (p_{ij}U_i)_{1\leq i\leq d}$ for $p = (p_{ij}(U))_{1\leq i,j\leq d}$ the magic unitary given by Lemma~\ref{lem:p_ij} and with $U_1,\ldots,U_d$ being magic unitaries adapted to $T$. 
	We claim that $U_i$ is actually adapted to $(T,x)$ for all $1\leq i\leq d$, that is, that $U_i$ preserves the root. 
	Indeed, it comes:
	\begin{align*}
		1 &= \sum_{r=1}^d \sum_{y\in V(T)} u_{ir,xy}\\
		&= \sum_{r=1}^d u_{ir,xx}\ \text{since $U$ is adapted to the rooted forest $d.(T,x)$}\\
		&= [U_i]_{xx}.
	\end{align*}
	This shows that $U_i$ is actually a magic unitary adapted to $(T,x)$. 
	Hence by universality the mapping $w_{ij,kl}\mapsto u_{ij,kl}$ extends to a surjective $*$-morphism from $\Oo(\Qu(T,x)\wr S_d^+)$ to $\Oo(\Qu(d.(T,x))$. 

	The two constructed morphisms are clearly inverse of each other (since they are on the generators), and they are quantum group morphisms since they preserve the matrix structure. 
	Hence they are the desired isomorphisms of quantum groups between $\Qu(d.(T,x))$ and $\Qu(T,x)\wr S_d^+$. 
	This concludes the proof.
\end{proof}

We can now compute the quantum automorphism group of a rooted forest in the function of the quantum automorphism groups of its connected components.

\begin{theorem}\label{thm:qu_rooted_forest}
	Let $k\geq 1$ and let $(F,r) = \sum_{i=1}^k a_i.(T_i,r_i)$ be a rooted forest, where $a_1,\ldots,a_k\in \N$, and $(T_1,r_1),\ldots,(T_k,r_k)$ are two-by-two non isomorphic rooted trees. 
	Then $\Qu(F,r) = *_{i=1}^k \Qu(T_i,r_i)\wr S_{a_i}^+$. 
\end{theorem}

Let us now treat quantum asymmetry. 
Notice that we will later prove Lemma~\ref{lem:unrooting}, which could be proved independently. 
The result would then follow from Fulton's result, Theorem~\ref{thm:qa_for_trees}. 
Here we will stay in the spirit of rooted trees and prove it using the inductive structure. 

\begin{theorem}\label{thm:qa_for_rooted_forests}
	A rooted forest is asymmetric if and only if it is quantum asymmetric.
\end{theorem}

\begin{proof}
	The reverse implication is clear, so let us show that an asymmetric forest is quantum asymmetric. 
	We claim it is enough to consider the case of rooted trees. 
	Let us assume the result proved for rooted trees. 
	Let $(F,r) = \sum_{i=1}^k (T_i,r_i)$ be a rooted forest with $k\geq 2$. 
	Assume that $\Aut(F,r)$ is trivial, we want to show that $\Qu(F,r)$ is as well.
	Let $(T,x) = x. \sum_{i=1}^k (T_i,r_i)$. 
	By Theorem~\ref{thm:rooted_forest_vs_rooted_tree}, we have that $\Aut(T,x) = \Aut(F,r) = 1$, so, assuming the result proved for rooted trees, we would have that $\Qu(T,x) = 1$ and by Theorem~\ref{thm:rooted_forest_vs_rooted_tree} again that $\Qu(F,r) = \Qu(T,x) = 1$. 
	This shows that it is enough to prove the result for rooted trees.

	Let us now prove the result for rooted trees by strong induction on the number of vertices. 
	It is clearly true for a tree on 1 vertex. 
	So we assume it proved for all trees on at most $n$ vertices, for some $n\geq 1$. 
	Let $(T,x)$ be an asymmetric rooted tree on $n+1$ vertices. 
	Since $n+1\geq 1$, we can write $(T,x) = x.\sum_{i=1}^k a_i.(T_i,r_i)$ for some $k\geq 1$, some $a_1,\ldots,a_k \in \N$, and such that $(T_i,r_i) \neq (T_j,r_j)$ for $i\neq j$. 
	Let $U$ be a magic unitary adapted to $(T,x)$. 
	Notice that by Theorem~\ref{thm:rooted_forest_vs_rooted_tree}, we have that $\Aut(F,r) = \Aut(T,x) = 1$. 
	In particular, for all $1\leq i\leq k$, we have that $\Aut(T_i,r_i) = 1$ and $a_i=1$. 
	Moreover, for $1\leq i\leq k$, since $\abs{V(T_i,r_i)} < n+1$, by induction hypothesis, we have that $\Qu(T_i,r_i) = 1$. 
	By Theorem~\ref{thm:qu_rooted_forest} and Theorem~\ref{thm:rooted_forest_vs_rooted_tree}, we have that $\Qu(T,x) = \Qu(F,r) = *_{i=1}^k \Qu(T_i,r_i) \wr S_1^+ = *_{i=1}^k 1 \wr 1 = 1$. 
	Thus we have proved that $\Qu(T,x) = 1$, which concludes the induction and the proof of the theorem.
\end{proof}

Finally, let us show that rooted forests satisfy the Schmidt alternative.

\begin{lemma}\label{lem:rooted_no_schmidt}
	Let $n\in \N$ and consider $n$ nonisomorphic rooted trees $(T_1,r_1),\ldots,(T_n,r_n)$ as well as $n$ nonzero integers $a_1,\ldots,a_n \geq 1$. 
	Let $(F,r) = \sum_{i=1}^n a_i.(T_i,r_i)$. 
	Then $(F,r)$ does not satisfy the rooted Schmidt criterion if and only if there is an integer $1\leq l\leq n$ such that:
	\begin{itemize}
		\item for every $i\neq l$, we have $a_i = 1$ and $\Aut(T_i,r_i) = 1$,
		\item $a_l\leq 3$,
		\item $(T_l,r_l)$ does not satisfy Schmidt's criterion,
		\item if $a_l >1$, then $\Aut((T_l,r_l)) = 1$.
	\end{itemize}
\end{lemma}

\begin{proof}
	The proof of Theorem~\ref{thm:sum_no_schmidt} naturally translates to the rooted context.
\end{proof}

\begin{theorem}\label{thm:rooted_sa}
	The class of rooted forests satisfies the rooted Schmidt alternative.
\end{theorem}

\begin{proof}
	Recall that, by Lemma~\ref{lem:rooted_qu_schmidt}, we have that a rooted forest satsfying the rooted Schmidt criterion has quantum symmetry. 
	So it is enough to prove the converse. 
	Moreover, is enough to prove it for rooted trees. 
	Indeed, assume it is true for rooted trees and let $(F,r)$ be a rooted forest. 
	Assume that $(F,r)$ does not satisfy the rooted Schmidt criterion. 
	Then letting $(T,x) = x.(F,r)$ we have by Theorem~\ref{thm:rooted_forest_vs_rooted_tree} that $(T,x)$ does not satisfy the rooted Schmidt criterion neither. 
	Thus, by assumption, $(T,x)$ does not have quantum symmetry, and by the same theorem we have that $\Qu(F,r) = \Qu(T,x)$ so $(F,r)$ does not have quantum symmetry neither. 
	This is what we wanted. 

	Let us now prove by strong induction on the number of vertices that a rooted tree which does not satisfy the rooted Schmidt criterion does not have quantum symmetry. 
	It is clearly true for a rooted tree on 1 vertex, so let us assume it true for all rooted trees on at most $m$ vertices, for some $m\geq 1$. 
	Let $(T,x)$ be a rooted tree on $m+1$ vertices and write it $(T,x) = x.\left(\sum_{i=1}^n a_i.(T_i,r_i)\right)$ with $n\in \N$ and $a_1,\ldots,a_n \in \N$ and such that $(T_i,r_i) \neq (T_j,r_j)$ for $i\neq j$. 
	Assume that $(T,x)$ does not satisfy the rooted Schmidt criterion and let $1\leq l\leq n$ be the integer given by Lemma~\ref{lem:rooted_no_schmidt}. 
	By Theorem~\ref{thm:qa_for_rooted_forests}, Theorem~\ref{thm:qu_rooted_forest}, and Theorem~\ref{thm:rooted_forest_vs_rooted_tree}, we have that $\Qu(T,x) = *_{i=1}^n \Qu(T_i,r_i)\wr S_{a_i}^+ = \Qu(T_l,r_l) \wr S_{a_l}^+$. 
	If $a_l >1$, we also have that $\Aut(T_l,r_l) = 1$, so $\Qu(T_l,r_l) = 1$ and $\Qu(T,x) = S_{a_l}^+$. 
	Since $a_l\leq 3$, we have that $\Oo(S_{a_l}^+)$ is commutative, hence $(T,x)$ does not have quantum symmetry. 
	Now assume that $a_l = 1$. 
	We have that $(T_l,r_l)$ does not satisfy the rooted Schmidt criterion. 
	Since $\abs{V(T_l,r_l)} \leq m$, by induction hypothesis, $(T_l,r_l)$ does not have quantum symmetry. 
	This means that $\Oo(\Qu(T,x)) = \Oo(\Qu(T_l,r_l))$ is abelian, hence $(T,x)$ does not have quantum symmetry. 
	Therefore the induction is complete and the proof finished.
\end{proof}

Hence we have proved that rooted forests are tractable in the sense of Section~\ref{sec:tractable} when adapted to the rooted context. 
We can now go bak to trees and forests.

\subsection{The transformation $\Psi$}\label{subsec:psi}

Let us now build useful transformations allowing us to transport our results from the rooted context to the unrooted one. 
We start by recalling without proof the following well-known lemma. 
We encountered it first in the article of Jordan~\cite{jordan1869assemblages}. 
Since we are not aware of any earlier mention, we attribute it to Jordan.

\begin{lemma}[Jordan]\label{lem:Z(T)}
	The center of a tree is isomorphic to $K_1$ or $K_2$.
\end{lemma}

The following result exhibits a useful link between rooted forests and trees.

\begin{lemma}\label{lem:unrooting}
	Let $(F,r)$ be a rooted forest. 
	Then there exists a tree $G$ which contains $F$ as a proper subgraph and such that:
	\begin{itemize}
		\item every magic unitary adapted to $G$ (in particular, every automorphism of $G$) preserves $F$ and is the identity when restricted to $G\setminus F$,
		\item every magic unitary adapted to $F$ (in particular, every automorphism of $F$) extends by the identity to a magic unitary adapted to $G$ (in particular, to an automorphism of $G$).
	\end{itemize}
	In particular, we have that $\Aut(G) = \Aut(F)$ and $\Qu(G) = \Qu(F)$.
\end{lemma}

\begin{proof}
	Let us write $(F,r) = \sum_{i=1}^k (T_i,r_i)$. 
	The construction is as follows. 
	Let $d = 1+\max \{ \deg(x) \mid x\in V(F)\}$. 
	Notice that $d\geq 2$. 
	Assume we have $(S_1,s_1),\ldots,(S_d,s_d)$ such that, for all $1\leq i\leq d$, we have:
	\begin{enumerate}
		\item $(S_i,s_i)$ is a quantum asymmetric rooted tree, 
		\item the maximal degree of $S_i$ is at most $d$, 
		\item $(S_i,s_i) \not \qi (S_j,s_j)$ for $j\neq i$,
		\item $(S_i,s_i) \not \qi (T_j,r_j)$ for all $1\leq j\leq k$.
	\end{enumerate}
	Let $(G,x) = x.\left( (F,r) + \sum_{i=1}^d (S_i,s_i) \right)$. 
	We claim that $G$ has the desired properties. 
	First, $G$ is a well-defined tree, and it clearly contains $F$ as a proper subgraph. 
	Then we have $\deg(x) = k+d \geq d+1$, so by Lemma~\ref{lem:degree} every magic unitary adapted to $G$ fixes $x$ and is thus a magic unitary adapted to $(G,x)$. 
	Now let $U$ be a magic unitary adapted to $(G,x)$. 
	By Lemma~\ref{lem:rooted_forest_vs_rooted_tree} we have that $U$ is adapted to $G\setminus \{x\}$ and, by Theorem~\ref{thm:qu_rooted_forest} and by assumption, it is actually the identity when restricted to $\sum_{i=1}^d (S_i,s_i)$. 
	Moreover, if $V$ is a magic unitary adapted to $(F,r)$, then extending it by the identity immediately gives a magic unitary adapted to $G\setminus \{x\}$, and so adding an extra 1 we obtain by Lemma~\ref{lem:rooted_forest_vs_rooted_tree} a magic unitary adapted to $(G,x)$ and in particular to $G$. 
	This is what we wanted.

	It remains to show we can always find such $(S_i,s_i)$. 
	We can for instance take $d$ paths rooted at one extremity of different lengths: they are clearly quantum asymmetrical, of maximum degree 2, two-by-two not quantum isomorphic as well as not quantum isomorphic to any of the $(T_i,r_i)$ if chosen of a different length, in case one of them is a path rooted at one extremity as well.
\end{proof}

We define a transformation $\Psi \colon \Trees \to \RTrees$ which allows us to transfer to trees what we have proved for rooted trees.

Let $T$ be a tree. 
If $Z(T) = K_1$, let $x$ be the vertex such that $Z(T) = \{x\}$. 
We simply set $\Psi(T) = (T,x)$. 
Let us now assume that $Z(T) = \{x,y\}$. 
We define $T'$ to be the graph obtained by subdividing the edge $xy$ once, adding a new vertex that we call $z$: $T'$ is clearly a tree -- for instance because it is connected and we have $\# V(T') = \#V(T) + 1 = \# E(T) + 2 = \# E(T') + 1$. 
We set $\Psi(T) = (T',z)$.

We now extend $\Psi$ to magic unitaries adapted to a rooted tree.

Let $T$ be a tree and let $U$ be a magic unitary adapted to it. 
If $Z(T) = K_1$, we set $\Psi(U) = U$. 
If $Z(T) = K_2$, we set $\Psi(U) = \Diag(U,1)$, where the extra coefficient corresponds to the new vertex. 

Notice that we are doing a slight abuse here: technically we are defining $\Psi(T,U)$ where $T$ is a tree and $U$ is a magic unitary adapted to it.\footnote{If we want to realise $\Psi$ as a function, we can for instance restrict ourselves to magic unitaries with coefficients being operators on a fixed separable Hilbert space, since all $C^*$-algebras considered are separable.}
However, we will maintain this abuse since it shall always be clear from the context with respect to which tree we are associating a given magic unitary.

\begin{lemma}\label{lem:qi_psi}
	Let $T_1$ and $T_2$ be two quantum isomorphic trees. 
	If $U$ is a quantum isomorphism between them, then $\Psi(U)$ is a quantum isomorphism of rooted trees between $\Psi(T_1)$ and $\Psi(T_2)$. 
	Conversely, every quantum isomorphism between $\Psi(T_1)$ and $\Psi(T_2)$ is of the form $\Psi(U)$ for $U$ a quantum isomorphism between $T_1$ and $T_2$.
\end{lemma}

\begin{proof}
	Let $U$ be a quantum isomorphism between $T_1$ and $T_2$. 
	First assume that $Z(T_1) = K_1$. 
	Then $Z(T_2) = K_1$ by Theorem~\ref{thm:qi_center} and $\Psi(U) = U$. 
	We also have $\Psi(T_i) = (T_i,r_i)$ with $Z(T_i) = \{r_i\}$ for $i\in \{1,2\}$. 
	Since $U$ is a quantum isomorphism between $T_1$ and $T_2$, we only need to show that $u_{r_2r_1} = 1$. 
	But this follows from Theorem~\ref{thm:qi_center}.

	Now let us assume that $Z(T_1) = K_2$. 
	We then also have $Z(T_2) = K_2$ by Theorem~\ref{thm:qi_center}. 
	This gives $\Psi(T_i) = (T_i',z_i)$ for $i\in \{1,2\}$, where $z_i$ is the added vertex, and we have $\Psi(U) = \Diag(U,1)$. 
	By Theorem~\ref{thm:qi_center} again, we can write $U$ in the form $U = \Diag(V,W)$, where $V$ and $W$ are magic unitaries, and $W$ is indexed by $Z(T_2)\times Z(T_1)$. 
	So actually $\Psi(U) = \Diag(V,W,1)$ is block diagonal. 
	Let us write $A = \Adj(T_1)$, $A' = \Adj(T_1')$, $B= \Adj(T_2)$, and $B' = \Adj(T_2')$. 
	By construction, we have $\Psi(U)_{z_2z_1} = 1$, and we already know that $\Psi(U)$ is a magic unitary (of the correct dimensions). 
	So to complete the proof we only need to check that $\Psi(U)A' = B'\Psi(U)$. 
	Let us write the matrices in the block decomposition adapted to the diagonal form of $\Psi(U) = \Diag(V,W,1)$. 
	Writing $L = (1,1)$, we have:
	\[ A' = \begin{pmatrix}
		A_{11} & A_{12} & 0\\
		A_{21} & 0 & L^T\\
		0 & L & 0
	\end{pmatrix}\ \text{and}\ B' = \begin{pmatrix}
		B_{11} & B_{12} & 0\\
		B_{21} & 0 & L^T\\
		0 & L & 0
	\end{pmatrix}.\]
	It comes:
	\[\Psi(U)A'\Psi(U)^* = \begin{pmatrix}
		VA_{11}V^* & VA_{12} W^* & 0\\
		W^*A_{21}V & 0 & WL^T\\
		0 & LW^* & 0
	\end{pmatrix}.\]
	Now since $UAU^* = B$ we already get $VA_{11}V^* = B_{11}$, $VA_{12}W^* = B_{12}$, and $W^*A_{21}V = B_{21}$. 
	So to conclude we just need to check that $WL^T = L^T$. 
	But this follows immediately from the facts that $L=(1,1)$ and $W$ is a magic unitary. 
	This completes the proof that $\Psi(U)$ is a quantum isomorphism between $\Psi(T_1)$ and $\Psi(T_2)$.

	Conversely, let us now take $V$ a quantum isomorphism between $\Psi(T_1)$ and $\Psi(T_2)$. 
	In the case where $Z(T_1) = K_1$, we also have $Z(T_2) = K_1$ by Theorem~\ref{thm:qi_center}. 
	So there is nothing to prove: $V$ is already a quantum isomorphism between $T_1$ and $T_2$ and by construction we have $\Psi(V) = V$. 
	Let us assume now that $Z(T_1) = K_2$. 
	Then we have $Z(T_2) = K_2$, and we write $\Psi(T_i) = (T_i',z_i)$ again. 
	Since $V$ is a quantum isomorphism of rooted trees, we have $V_{z_2z_1} = 1$. 
	So $V = \Diag(U,1)$, where $U$ is a magic unitary. 
	Writing $A = \Adj(T_1)$ and $B=\Adj(T_2)$, we only need to show that $UA = BU$ in order to complete the proof. 
	We decompose the matrices in blocks adapted to the centers, writing $A = (A_{ij})_{1\leq i,j\leq 2}$, and similarly for $B$. 
	Since $V$ preserves the roots, and the roots form the centers, we have by Theorem~\ref{thm:qi_center} that $U = \Diag(U_1,U_2)$, where $U_2$ is of size 2. 
	The fact that $V$ conjugates $\Adj(T_1')$ and $\Adj(T_2')$ implies that $U_1A_{11} = B_{11}U_1$, $U_1 A_{12} = B_{12}U_2$, and $U_2A_{21} = B_{21}U_1$. 
	It remains to check that $U_2 A_{22} = B_{22}U_2$. 
	But we have:
	\[ A_{22} = \begin{pmatrix}
		0 & 1\\
		1 & 0
	\end{pmatrix} = B_{22},\]
	and it commutes with $U_2$ since $U_2$ is a 2-dimensional magic unitary. 
	This completes the proof.
\end{proof}

We gather useful facts about the transformation $\Psi$ in the next lemma.

\begin{lemma}\label{lem:psi}
	Let $T$, $T_1$, $T_2$ be trees and let $U$ be a magic unitary adapted to $T$ with coefficients in a unital $C^*$-algebra $X$. 
	We have:
	\begin{itemize}
		\item if $\Psi(T_1) = \Psi(T_2)$ and $T_1$ and $T_2$ have the same number of vertices, then $T_1 = T_2$,
		\item $\Supp(U) = \Supp(\Psi(U))$,
		\item the coefficients of $U$ commute if and only if the ones of $\Psi(U)$ do,
		\item $U$ is a submatrix of $\Psi(U)$,
		\item the coefficients of $U$ generate the same $C^*$-subalgebra of $X$ as the ones of $\Psi(U)$,
		\item $\Psi$ induces an isomorphism of the groups $\Aut(T)$ and $\Aut(\Psi(T))$,
		\item $\Psi$ induces an isomorphism of the quantum groups $\Qu(T)$ and $\Qu(\Psi(T))$.
	\end{itemize}
\end{lemma}

\begin{proof}
	The last two statements follow from Lemma~\ref{lem:qi_psi} and the others are elementary observations.
\end{proof}

Thanks to $\Psi$, we obtain that the class of forests is tractable.

\begin{theorem}\label{thm:tractable_forests}
	The class of forests is tractable. 
	In other words, the following are true:
	\begin{itemize}
		\item[(QA)] a forest is asymmetric if and only if it has a trivial quantum automorphism group,
		\item[(QI)] two quantum isomorphic forests are isomorphic,
		\item[(SA)] a forest has quantum symmetry if and only if it satisfies Schmidt's criterion.
	\end{itemize}
\end{theorem}

\begin{proof}
	By Theorem~\ref{thm:problems_sums}, it is enough to show that the class of trees is tractable. 

	By Theorem~\ref{thm:qa_for_trees}, we know that trees satisfy (QA). 
	We can also obtain it explicitly as follows. 
	Let $T$ be an asymmetric tree and let $U$ be a magic unitary adapted to $T$. 
	By Lemma~\ref{lem:psi}, we have that $\Psi(U)$ is adapted to the rooted tree $\Psi(T)$, and we also have that $\Psi(T)$ is asymmetric. 
	By Theorem~\ref{thm:qa_for_rooted_forests}, we have that $\Psi(T)$ is quantum asymmetric, so $\Psi(U)$ is the identity matrix. 
	So by Lemma~\ref{lem:psi} we have that $U$ is the identity matrix too. 
	Applying this to the fundamental representation of $\Qu(T)$ we obtain that $\Qu(T) = 1$, as desired. 
	So trees satisfy (QA).

	Let $T$ and $S$ be two quantum isomorphic trees and let $U$ be a quantum isomorphism from $T$ to $S$. 
	By Lemma~\ref{lem:qi_psi}, we have that $\Psi(U)$ is a quantum isomorphism of rooted trees from $\Psi(T)$ to $\Psi(S)$. 
	By Theorem~\ref{thm:qi_rooted_forests}, we have that $\Psi(T) = \Psi(S)$. 
	This and the fact that $T$ and $S$ have the same number of vertices implies that $T=S$ by Lemma~\ref{lem:psi}. 
	Thus quantum isomorphic trees are isomorphic.

	Finally, let $T$ be a tree not satisfying Schmidt's criterion. 
	Then $\Psi(T)$ does not satisfy Schmidt's criterion neither by Lemma~\ref{lem:psi}. 
	Since the class of rooted forests satisfies the (rooted) Schmidt alternative by Theorem~\ref{thm:rooted_sa}, we have that $\Psi(T)$ does not have quantum symmetry. 
	But then $T$ does not have quantum symmetry by Lemma~\ref{lem:psi}. 
	This proves that the class of trees satisfies the Schmidt alternative.

	Thus the class of trees is tractable, and so is the class of forests by Theorem~\ref{thm:problems_sums}. 
	This concludes the proof.
\end{proof}

Let us conclude this section by showing that trees and tree-cographs are superrigid. 

\begin{theorem}\label{thm:superrigid_trees}
	The class of trees is superrigid, that is, if $T$ is a tree and $G$ is a graph such that $T\qi G$, then $G=T$. 
\end{theorem}

\begin{proof}
	Let $G$ be a graph and $T$ a tree such that $G\qi T$. 
	Recall that a graph $H$ is a tree if and only if $H$ is connected and $\abs{V(H)} = \abs{E(H)} + 1$. 
	Since quantum isomorphism preserves connectedness (by Corollary~\ref{coro:connected_qi}), the number of vertices (by definition), and the number of edges (by Lemma~\ref{lem:qi_same_nb_edges}), then $G$ is a tree. 
	So $G=T$ because trees satisfy (QI) by Theorem~\ref{thm:tractable_forests}. 
	This shows that the class of trees is superrigid.
\end{proof}

\begin{corollary}\label{coro:superrigid_tree_cographs}
	The class of tree-cographs is superrigid.
\end{corollary}

\begin{proof}
	This follows immediately from Theorem~\ref{thm:superrigid_trees} by Theorem~\ref{thm:superrigid_F_cographs}.
\end{proof}

\subsection{The noncommutative Jordan theorem}\label{subsec:noncommutative_jordan}

In this subsection, we compute the quantum automorphism groups of forests in Corollary~\ref{coro:qu_forest_equal_tree} and Theorem~\ref{thm:noncommutative_jordan}. 
We then show that tree-cographs are tractable in Theorem~\ref{thm:tractable_tree_cographs}. 
Thanks to the results of Section~\ref{sec:tractable}, we obtain the classical and quantum automorphism groups of tree-cographs in Theorem~\ref{thm:aut_tree_cographs_vs_trees} and Theorem~\ref{thm:jordan_tree_cographs}.

\begin{lemma}\label{lem:for_jordan}
	Let $T$ be a tree. 
	Then there exists a tree $H$ such that $T$ is not quantum isomorphic to $H$ and $\Qu(T) = \Qu(H)$. 
\end{lemma}

\begin{proof}
	Let $H$ be the tree given by Lemma~\ref{lem:unrooting} applied to $\Psi(T)$. 
	Since $T$ is a proper subgraph of $H$, they do not have the same number of vertices and thus are not quantum isomorphic. 
	And we have $\Qu(T) = \Qu(\Psi(T)) = \Qu(H)$, as desired.
\end{proof}

We recall Jordan's classification of automorphism groups of finite trees obtained by Jordan in~\cite{jordan1869assemblages}.

\begin{theorem}[Jordan]\label{thm:jordan}
	The class of automorphism groups of finite trees is the smallest family $\Jj$ of finite groups such that:
\begin{enumerate}
	\item $\Jj$ contains the trivial group,
	\item if $A$ and $B \in \Jj$, then $A\times B \in \Jj$,
	\item for $A\in \Jj$ and $n\geq 1$, we have $A\wr S_n \in \Jj$.
\end{enumerate}
\end{theorem}

We now extend it to the noncommutative setting. 
Let us start with a useful lemma.

\begin{lemma}\label{lem:qu_tree_forest}
	Let $G$ be a quantum permutation group. 
	Then the following are equivalent:
	\begin{enumerate}
		\item $G = \Qu(T)$ for some tree $T$,
		\item $G = \Qu(F)$ for some forest $F$,
		\item $G = \Qu(T,r)$ for some rooted tree $(T,r)$,
		\item $G = \Qu(F,(r_i)_{1\leq i\leq k})$ for some rooted forest $(F,(r_i)_{1\leq i\leq k}$.
	\end{enumerate}
\end{lemma}

\begin{proof}
	It is clear that (1) implies (2). 
	We have that (1) implies (3) by Lemma~\ref{lem:psi} and it is clear that (3) implies (4). 

	Let us show that (4) implies (1). 
	Let $(F,r)$ be a rooted forest and let $T$ be the tree given by applying Lemma~\ref{lem:unrooting} to $(F,r)$. 
	Then $\Qu(F,r) = \Qu(T)$. 
	This shows that (4) implies (1). 

	This shows that (1), (3), and (4) are equivalent. 
	
	Let us show that (2) implies (4). 
	Let $F = \sum_{i=1}^k a_iT_i$ be a forest, where $k\geq 1$, $T_i$ is a tree and $a_i\geq 1$ for $1\leq i\leq k$, and $T_i\neq T_j$ when $i\neq j$. 
	By Theorem~\ref{thm:qu_aut_sums} and Theorem~\ref{thm:tractable_forests}, we have that $\Qu(F) = *_{i=1}^k \Qu(T_i) \wr S_{a_i}^+$. 
	Assume first that $\Psi(T_i)\neq \Psi(T_j)$ for all $i\neq j$. 
	Then, setting $(F',r) = \sum_{i=1}^k a_i\Psi(T_i)$ we have by Theorem~\ref{thm:qu_rooted_forest} that $\Qu(F',r) = \Qu(F)$. 
	So it remains to treat the case where there exist $i\neq j$ with $\Psi(T_i) = \Psi(T_j)$. 
	Let $(S_i,r_i) = \Psi(T_i)$. 
	By adding paths of different lengths starting at $r_i$, we can ensure, in a way similar to the construction in the proof of Lemma~\ref{lem:unrooting}, to obtain new rooted trees $(W_i,r_i)$ such that $\deg_{W_i}(r_i) \neq \deg_{W_j}(r_j)$ with $\Qu(W_i,r_i) = \Qu(S_i,r_i)$. 
	Now by Theorem~\ref{thm:qu_rooted_forest} we have that $\Qu(\sum_{i=1}^k a_i(W_i,r_i)) = *_{i=1}^k \Qu(S_i,r_i) \wr S_{a_i}^+ = *_{i=1}^k \Qu(T_i) \wr S_{a_i}^+ = \Qu(F)$, as desired. 
	This shows that (2) implies (4). 

	This concludes the proof.
\end{proof}

\begin{corollary}\label{coro:qu_forest_equal_tree}
	The family of quantum automorphism groups of forests is equal to the family of quantum automorphism groups of trees.
\end{corollary}

We can now prove the noncommutative version of Jordan's theorem.

\begin{theorem}\label{thm:noncommutative_jordan}
	The family of quantum automorphism groups of finite trees is the smallest family $\Ii$ of compact quantum groups such that:
	\begin{enumerate}
		\item $\Ii$ contains the trivial quantum group,
		\item if $A$ and $B\in \Ii$, then $A * B \in \Ii$,
		\item for $A\in \Ii$ and $n\geq 1$, we have $A \wr S_n^+ \in \Ii$.
	\end{enumerate}
\end{theorem}

\begin{proof}
	Let $\Jj$ be the family of quantum automorphism groups of trees. 
	We start by showing that $\Jj \subset \Ii$. 
	By Lemma~\ref{lem:qu_tree_forest}, we have that $\Jj$ is equal to the family of quantum automorphism groups of rooted trees. 
	So it is enough to show by strong induction on the number of vertices that $\Qu(T,r) \in \Ii$ for every rooted tree $(T,r)$. 
	It is clearly true for one vertex. 
	Let us assume that $\Qu(X,x) \in \Ii$ for all rooted trees $(X,x)$ on at most $n$ vertices for some $n\geq 1$. 
	Let $(T,r)$ be a rooted tree on $n+1$ vertices. 
	Since $n+1 \geq 2$, we can write $(T,r) = r. \left(\sum_{i=1}^k a_i.(T_i,r_i) \right)$ for some $k\in \N$, some integers $a_1,\ldots,a_k$, and with $(T_i,r_i) \neq (T_j,r_j)$ when $i\neq j$. 
	By Theorem~\ref{thm:qu_rooted_forest} and Theorem~\ref{thm:rooted_forest_vs_rooted_tree}, we have that $\Qu(T,r) = *_{i=1}^k \Qu(T_i,r_i)\wr S_{a_i}^+$. 
	Notice that for all $1\leq i\leq k$ we have $\abs{V(T_i,r_i)} \leq n$, so, by induction hypothesis, we have that $\Qu(T_i,r_i) \in \Ii$. 
	This shows that $\Qu(T,r) \in \Ii$, as desired. 
	Hence we have proved by induction that $\Jj \subset \Ii$.

	Now let us show that $\Ii \subset \Jj$. 
	For this, it is enough to show that $1\in \Jj$ (clear) and that $\Jj$ is stable by free product and wreath product. 
	Let $T$ and $S$ be two trees. 
	If $T=S$, then, by Lemma~\ref{lem:for_jordan}, we can take a tree $X$ such that $\Qu(T) = \Qu(X)$ and $T\neq X$. 
	So we can assume that $T\neq S$.
	Then we have that $\Qu(T+S) = \Qu(T)*\Qu(S)$ by Theorem~\ref{thm:qu_aut_sums}. 
	By Lemma~\ref{lem:qu_tree_forest}, there is a tree $Y$ such that $\Qu(Y) = \Qu(T+S) = \Qu(T)*\Qu(S)$. 
	Hence $\Jj$ is stable by free product.

	Now, let $d\in \N$. 
	We have $\Qu(d.T) = \Qu(T)\wr S_d^+$ by Theorem~\ref{thm:wreath_product}. 
	Moreover, by Lemma~\ref{lem:qu_tree_forest}, there is a tree $Z$ such that $\Qu(Z) = \Qu(d.T) = \Qu(T)\wr S_d^+$. 
	Hence $\Jj$ is stable by wreath product.

	All in all, $\Jj$ contains the trivial quantum group, is stable by free product and wreath product, so $\Ii \subset \Jj$. 
	This shows that $\Jj = \Ii$ and concludes the proof.
\end{proof}

Thanks to our general approach, our results naturally extend to the class of tree-cographs. 
In order to apply the results of Section~\ref{sec:tractable}, we need to have a family closed under complement and taking connected components. 
Thus we start by characterising the closure of the class of trees under these two operations.

\begin{lemma}\label{lem:connected_tree_complement}
	Let $T$ be a tree and assume that $T^c$ is disconnected. 
	Then there is an integer $n \in \N$ such that $T = K_{1,n}$.
\end{lemma}

\begin{proof}
	Write $T^c = G_1 + \ldots + G_k$ with $k\geq 2$ and $G_1,\ldots,G_k$ connected. 
	Assume that two distinct connected components of $T^c$ have at least two vertices. 
	This means there is $i\neq j$ and $x$, $y\in V(G_i)$ and $a$, $b\in V(G_j)$ with $x\neq y$ and $a\neq b$. 
	Now $(x,a,y,b) = C_4$ is a cycle in $T$, which is a contradiction. 
	Hence there is $1\leq i_0\leq k$ such that if $i\neq i_0$ then $G_i = K_1$. 

	Let us show now that $k=2$. 
	Indeed, assume $k\geq 3$. 
	Then we can take $x$, $y$, $z \in V(G)$ two-by-two nonadjacent, so $(x,y,z) = K_3$ is a cycle in $T$, a contradiction. 
	So $k=2$.

	We now write $T^c = G + K_1$, with $G = G_{i_0}$. 
	Notice that if there is $x \neq y\in V(G)$ such that $xy \notin E(G)$, then, letting $z$ be the vertex of $K_1$, we have that $(x,y,z) = K_3$ is a cycle in $T^c$, a contradiction. 
	Hence $G$ is a complete graph. 
	Letting $n= \abs{V(G)}$, we have that $G = K_n$ so $T^c = K_n + K_1$. 
	This gives us $T = K_{1,n}$, as desired.
\end{proof}

\begin{lemma}\label{lem:F_for_tree_cographs}
	Let $F$ be the smallest class containing all trees and closed under taking complements and connected components. 
	Then the following are true:
	\begin{enumerate}
		\item $F = \Trees \cup \Trees^c \cup \{K_n \mid n\geq 1\} \cup \{ n.K_1 \mid n\geq 1\}$,
		\item $F$ is tractable,
		\item $\co(F) = \co(\Trees) = \co(\Forests)$.
	\end{enumerate}
\end{lemma}

\begin{proof}
	We start by proving (1). 
	Let $L$ be the right-hand side of the equality in (1). 
	It is clear that $L$ is closed under complement and it follows immediately from Lemma~\ref{lem:connected_tree_complement} that it is closed under taking connected components. 
	Since $\Trees \subset L$, we have by definition that $F \subset L$. 
	Let us show the converse. 
	We have $\Trees \cup \Trees^c \subset F$. 
	Let $n\geq 1$. 
	Then $K_{1,n} \in \Trees$ and $K_{1,n}^c = K_n + K_1$ so $K_n \in F$. 
	This implies that $K_n^c = n.K_1 \in F$. 
	This being true for all $n\geq 1$, we have that $L\subset F$, so $L=F$ as wanted. 

	Let us show (2). 
	For (QI), simply notice that by (1) we have that $F\subset \co(\Trees)$, and the latter is superrigid by Corollary~\ref{coro:superrigid_tree_cographs}, so in particular $F$ satisfies (QI).
	Now by Lemma~\ref{lem:tractable_complement}, it is enough to show that $\Trees \cup \{ K_n \mid n\geq 1\}$ satisfies (QA) and (SA). 
	But $\Trees$ does by Theorem~\ref{thm:tractable_forests} and it is clear that $\{ K_n \mid n\geq 1\}$ does too. 
	This shows that $F$ is tractable, as desired.

	Let us show (3). 
	We have $\Trees \subset F$ so $\co(\Trees) \subset \co(F)$. 
	By (1), we have $F \subset \co(\Forests)$, so $\co(F) \subset \co(\Forests)$. 
	Finally it is clear that $\Forests \subset \co(\Trees)$, so $\co(\Forests) \subset \co(\Trees)$. 
	So we have shown that $\co(\Trees) \subset \co(F) \subset \co(\Forests) \subset \co(\Trees)$ which implies (3). 
	This concludes the proof.
\end{proof}

We now obtain the tractability of tree-cographs.

\begin{theorem}\label{thm:tractable_tree_cographs}
	The class of tree-cographs is tractable, that is:
	\begin{itemize}
		\item[(QA)] a tree-cograph is asymmetric if and only if it is quantum asymmetric,
		\item[(QI)] two quantum isomorphic tree-cographs are isomorphic,
		\item[(SA)] a tree-cograph has quantum symmetry if and only if it satisfies Schmidt's criterion.
	\end{itemize}
\end{theorem}

\begin{proof}
	It follows immediately from applying Theorem~\ref{thm:tractable_F_cographs} to the class $F$ of Lemma~\ref{lem:F_for_tree_cographs}.
\end{proof}

Finally, let us compute the classical and quantum automorphism groups of tree-cographs.

\begin{theorem}\label{thm:aut_tree_cographs_vs_trees}
	The family of automorphism groups of tree-cographs is exactly the smallest family of finite groups containing the trivial group and stable by product and wreath product. 
	In particular, it is the class of quantum automorphism group of trees.
\end{theorem}

\begin{proof}
	Let $\Ii$ be the smallest family of finite groups containing the trivial group and stable by product and wreath product. 
	By Theorem~\ref{thm:jordan}, we have that $\Ii = \Aut(\Trees)$. 
	Now let $F$ be the family of Lemma~\ref{lem:F_for_tree_cographs}. 
	We claim that $\Aut(F) = \Ii$. 

	Since $\Trees \subset F$, we have $\Ii \subset \Aut(F)$. 
	Conversely, since $S_n \in \Ii$ for all $n\in \N$, we have $\Aut(F) \subset \Ii$ by (1) of Lemma~\ref{lem:F_for_tree_cographs}. 
	So we have shown that $\Aut(F) = \Ii$.

	Now by Theorem~\ref{thm:Jor_vs_Fcographs} we have that $\Aut(\co(F)) = \Jor(\Aut(F))$. 
	We have $\co(\Trees) = \co(F)$ by (3) of Lemma~\ref{lem:F_for_tree_cographs}, so it comes:
	\[ \Aut(\co(\Trees)) = \Aut(\co(F)) =  \Jor(\Aut(F))) = \Jor(\Ii) = \Ii,\]
	as desired. 
	This concludes the proof.
\end{proof}

We have the noncommutative version as well.

\begin{theorem}\label{thm:jordan_tree_cographs}
	The family of quantum automorphism groups of tree-cographs is exactly the smallest family of quantum permutation groups containing the trivial group and stable by free product and wreath product. 
	In particular, it is the class of quantum automorphism group of trees.
\end{theorem}

\begin{proof}
	Let $\Jj$ be the smallest family of quantum permutation groups containing the trivial group and stable by free product and wreath product. 
	By Theorem~\ref{thm:noncommutative_jordan}, we have that $\Jj = \Qu(\Trees)$. 
	Now let $F$ be the family of Lemma~\ref{lem:F_for_tree_cographs}. 
	We claim that $\Qu(F) = \Jj$. 

	Since $\Trees \subset F$, we have $\Jj \subset \Qu(F)$. 
	Conversely, since $S_n^+ \in \Jj$ for all $n\in \N$, we have $\Qu(F) \subset \Jj$ by (1) of Lemma~\ref{lem:F_for_tree_cographs}. 
	So we have shown that $\Qu(F) = \Jj$.

	Now by Theorem~\ref{thm:quantum_Jor_vs_Fcographs} we have that $\Qu(\co(F)) = \Jor(\Qu(F))$. 
	We have $\co(\Trees) = \co(F)$ by (3) of Lemma~\ref{lem:F_for_tree_cographs}, so it comes:
	\[ \Qu(\co(\Trees)) = \Qu(\co(F)) =  \Jor(\Qu(F))) = \Jor(\Jj) = \Jj,\]
	as desired. 
	This concludes the proof.
\end{proof}

\section*{Acknowledgements}

I would like to warmly thank Amaury Freslon for his many pieces of advice and suggestions during this project, for our valuable discussions, and for his helpful comments on the first version of this article. 
I would also like to deeply thank Pegah Pournajafi for many fruitful discussions, in particular for her expertise in graph theory, and for her feedback on the first version of this article. 
Finally, I would also like to thank Frédéric Meunier for insightful discussions.

\appendix

\section{Lovász' formula for graph classes and $\Ff$-isomorphism}\label{app:lovasz}

In~\cite{LovOpStruct}, Lovász studied the category of relational structures, that is, finite sets with $n$-ary relations for some integer $n\geq 1$. 
His aim was to show a cancellation law, and he achieved so by proving a formula for the number of morphisms from one structure to another, leading to his celebrated result that, if $A$ and $B$ are two structures such that $\# \Hom(C,A) = \# \Hom(C,B)$ for all structures $C$, then $A \simeq B$. 

The methods used by Lovász are categorical in flavor and were immediately generalised, see for instance~\cite{pultr1973isomorphism} by Pultr for a purely categorical reformulation. 
They also apply naturally to the category of graphs and allow for reformulations of the isomorphism theorem when restricting the attention to specific graph classes, though it seems uneasy to find explicit references in the literature. 
This is of immediate interest for the notion of $\Ff$-isomorphism introduced in~\cite{ManRob}, where the main theorem is that quantum isomorphism is equal to $\Ff$-isomorphism for $\Ff$ the class of planar graphs. 

Hence, in this appendix, we translate Lovász' formula for graphs in Theorem~\ref{thm:lovasz_formula} and Corollary~\ref{coro:lovasz_formula_graph_classes}. 
For the sake of completeness, we include the proofs written in modern terminology. 
This allows us to obtain general isomorphism results for $\Ff$-isomorphism in Theorem~\ref{thm:F_iso} and Theorem~\ref{thm:connected_F_iso}. 
As an immediate consequence, we obtain that quantum isomorphic planar graphs are isomorphic in Corollary~\ref{coro:qi_planar}. 

Let us point out that contrarily to the rest of the present article, the approach in this section is not explicit, and involves elementary category theory.

Since in this appendix we work in the category of graphs, we lift temporarily the restriction given by Remark~\ref{rk:iso_graphs} and talk about graphs as represented with sets, and not as equivalence classes\footnote{Albeit we might choose to work only with graphs whose vertex set is of the form $\{1,\ldots,n\}$ for some $n\in \N$, or use any other similar logical device in order to work with a category with a set of objects.}.
Indeed, the category is doing the job of dealing with isomorphisms by itself.

Let $G$ and $H$ be two graphs. 
Recall that a graph morphism is a function $f \colon V(G)\to V(H)$ such that, for all $x$ and $y\in V(G)$ such that $xy \in E(G)$, we have $f(x)f(y) \in V(H)$. 
We denote the set of graph morphisms from $G$ to $H$ by $\Hom(G,H)$ and its cardinality by $\hom(G,H)$. 
Similarly, we denote the set of monomorphisms (resp. epimorphisms) from $G$ to $H$ by $\Mon(G,H)$ (resp. $\Epi(G,H)$), and its cardinality by $\mon(G,H)$ (resp. $\epi(G,H)$). 
We also denote by $\Aut(G)$ the set of automorphisms of $G$ and its cardinality by $\aut(G)$. 

The following elementary characterisation will be often used without explicit reference.

\begin{lemma}\label{lem:isomorphisms}
	A graph morphism $f \colon G\to H$ is an isomorphism if and only if it is a bijection and the map induced between the sets of edges is also bijective.
\end{lemma}

\begin{lemma}\label{lem:epi_mono}
	In the category of graphs, monomorphisms are exactly injective morphisms, and epimorphisms are exactly surjective morphisms. 
	Moreover, given some graphs $G$ and $H$, if $f \in \Mon(G,H)$, then $G$ is a subgraph of $H$, and $f$ is an inclusion of subgraphs.
\end{lemma}

\begin{proof}
	The proof is en easy exercise.
\end{proof}

However, a morphism which is both a monomorphism and an epimorphism is not an isomorphism in general: consider for instance an inclusion $2K_1 \hookrightarrow K_2$. 
Hopefully, we can recover isomorphisms with some slightly stronger requirements which will be quite useful, as shown in the next lemmas.

A graph morphism $f \in \Hom(G,H)$ is said to be a \textit{quotient} if it is surjective and if for every edge $ab \in E(H)$ there is $xy \in E(G)$ such that $f(x)=a$ and $f(y) = b$. 
In other words, $f$ is a quotient if it is surjective both on vertices and edges. 
Given $G$ and $H$ two graphs, we denote by $\Quo(G,H)$ the set of quotients from $G$ to $H$, and we denote by $\quo(G,H)$ the cardinality of $\Quo(G,H)$. 

\begin{lemma}\label{lem:quotient_iso}
	Let $G$ and $H$ be two graphs and let $f\colon G\to H$ be a graph morphism. 
	If $f$ is both a quotient and a monomorphism, then it is an isomorphism.
\end{lemma}

\begin{proof}
	The function $f$ is injective by Lemma~\ref{lem:epi_mono} and surjective since it is a quotient, so it is a bijection. 
	Since it is injective, it induces an injective map from $E(G)$ to $E(H)$, which is also surjective since $f$ is a quotient. 
	So $f$ is actually a graph isomorphism.
\end{proof}

For a result more general than the next lemma, see Lemma~1.5 in~\cite{pultr1973isomorphism}.

\begin{lemma}\label{lem:mono_mono}
	Let $G$ and $H$ be two graphs and let $f \colon G\to H$ and $h\colon H\to G$ be two graph morphisms. 
	Assume that both $f$ and $h$ are monomorphisms. 
	Then they are both isomorphisms.
\end{lemma}

\begin{proof}
	Let $a = h\circ f\colon G\to G$ and $b = f\circ h\colon H\to H$. 
	Since $a$ is injective from the finite set $V(G)$ to itself (by Lemma~\ref{lem:epi_mono}), it is a bijection; similarly, $b$ is a bijection as well. 
	So there are $k\geq 1$ and $l\geq 1$ such that $a^k = 1_G$ and $b^l = 1_H$. 
	Now we have that $fa^{k-1}$ is a right-inverse to $h$ and that $b^{l-1}f$ is a left-inverse to $h$, so $h$ is an isomorphism. 
	Similarly, $a^{k-1}h$ is a left-inverse to $f$, and $hb^{l-1}$ is a right-inverse to $f$, so $f$ is an isomorphism. 
	This concludes the proof.
\end{proof}

\begin{lemma}\label{lem:factorisation_quotient_mono}
	Let $f \colon G\to H$ be a graph morphism. 
	Then there exists a graph $A$ isomorphic to $f(G)$ and a pair $(q,m) \in \Quo(G,A)\times \Mon(A,H)$ such that $f = mq$. 
	Morevoer, if $A'$ is a graph such that there exists $(q',m') \in \Quo(G,A') \times \Mon(A',H)$ with $f = m'q'$, then there exists an isomorphism $\alpha \colon A\to A'$ such that $q' = \alpha q$ and $m' = \alpha m$. 
\end{lemma}

\begin{proof}
	Let $X = f(V(G)) \subset V(H)$. 
	Define a graph $A$ by setting $V(A) = X$ and $E(A) = f(E(G)) \subset E(H)$ -- in other words, we have $A = f(G)$. 
	Let $q \colon V(G) \to X$ be the corestriction of $f$, that is, $q(x) = f(x)$ for all $x \in V(G)$. 
	We claim that $q$ gives rise to a quotient morphism from $G$ to $A$. 
	Indeed, it is a graph morphism since $f$ is, and by construction it is surjective on vertices and edges, as desired. 
	Now $A$ is by construction a subgraph of $H$, let $m \colon A\to H$ be the inclusion morphism. 
	It is a monomorphism by Lemma~\ref{lem:epi_mono}. 
	Finally, we clearly have $f = mq$, so $A$, $q$, an $m$ have the desired properties.

	Let $A'$ be a graph with $(q',m') \in \Quo(G,A')\times \Mon(A',H)$ such that $f = m'q'$. 
	Notice that $m'(V(A')) = m'(q'(V(G))) = f(V(G)) = X$ since $q'$ is a quotient map. 
	Since $m'$ is injective and $m(V(A)) = X$, the function $\alpha \colon m'{}\inv\circ m \colon V(A) \to V(A')$ is well-defined and injective. 
	Let us check that it is a graph morphism. 
	Consider $a$, $b\in V(A)$ such that $ab\in E(A)$. 
	Since $q$ is a quotient map, there are $x$, $y\in V(G)$ with $xy\in E(G)$ such that $q(x)=a$ and $q(y) = b$. 
	It comes:
	\begin{align*}
		\alpha(a) &= m'{}\inv m(a) = m'{}\inv mq(x)\\
		&= m'{}\inv f(x) = m'{}\inv m' q'(x) = q'(x),
	\end{align*}
	and, similarly, $\alpha(b) = q'(y)$. 
	So $\alpha(a)\alpha(b) = q'(x)q'(y) \in E(A')$ since $xy\in E(G)$ and $q'$ is a graph morphism. 

	Finally, let us check that $\alpha$ is the desired graph isomorphism. 
	We already know it is injective; let $c\in V(A')$. 
	We saw that $m'(c) \in X = m(V(A))$, hence there exists $a\in V(A)$ with $m(a) = m'(c)$, so that $\alpha(a) = m'{}\inv m(a) = m'{}\inv m'(c) = c$, and $\alpha$ is surjective as well. 
	Hence for $\alpha$ to be a graph isomorphism, it remains to check that $\alpha$ is surjective on edges as well. 
	Let $a'$, $b'\in V(A')$ be such that $a'b'\in E(A')$. 
	Since $q'$ is a quotient map, there are $x$ and $y\in V(G)$ with $xy\in E(G)$ and $q'(x)=a'$ and $q'(y) = b'$. 
	Now letting $u = f(x)$ and $v = f(y)$ we have by construction $uv \in E(A)$. 
	Then it comes:
	\begin{align*}
		m'(a') &= m'(q'(x)) = f(x)\\
		&= u = m(u),
	\end{align*}
	so $a' = m'{}\inv m(u) = \alpha(u)$. 
	Similarly we obtain that $b' = \alpha(v)$, thus $a'b' = \alpha(u)\alpha(v) \in \alpha(E(A))$ and $\alpha$ is also surjective on edges. 
	Hence we have proved that $\alpha \colon A\to A'$ is a graph isomorphism. 

	Finally, let $x \in V(G)$. 
	We have $mq(x) = f(x) = m'q'(x)$, so $q'(x) = m'{}\inv mq (x) = \alpha q(x)$, which shows that $q' = \alpha q$. 
	And for $a\in V(A)$, we have $m'\alpha(a) = m' m'{}\inv m(a) = m(a)$, so $m'\alpha = m$. 
	This concludes the proof.
\end{proof}

We refer to the graph $A$ given by Lemma~\ref{lem:factorisation_quotient_mono} as the \textit{image} of $f$. 

\begin{lemma}\label{lem:aut_acts_morphisms}
	Let $G$ and $H$ be two graphs. 
	Let $f\in \Hom(G,H)$ and let $A$ be the image of $f$. 
	Then $\# \{ (q,m) \in \Quo(G,A)\times \Mon(A,H) \mid mq=f\} = \aut(A)$.
\end{lemma}

\begin{proof}
	Let $X_f = \{ (q,m) \in \Quo(G,A) \times \Mon(A,H) \mid f = mq\}$. 
	Fix $x = (q,m) \in X_f$. 
	We define a function $\psi_x \colon \Aut(A) \to X_f$ given by $\psi_x(\alpha) = (\alpha q,m\alpha\inv)$. 
	It is a well-defined function and it is surjective by Lemma~\ref{lem:factorisation_quotient_mono}. 
	Now assume that $\psi_x(\alpha) = \psi_x(\beta)$ for some $\alpha$, $\beta\in \Aut(A)$. 
	This means that $\alpha q = \beta q$ as graph morphisms from $G$ to $H$. 
	Since $q$ is an epimorphism by Lemma~\ref{lem:epi_mono}, we conclude that $\alpha = \beta$. 
	This shows that $\psi_x$ is injective. 
	Hence we have shown it is bijective, and we can conclude that $\# X_f = \# \Aut(A) = \aut(A)$.
\end{proof}

Recall that, given an equivalence relation $R$ on a set $X$, a \textit{system of representatives} for $R$ is a subset $S\subset X$ such that, for every $x\in X$, there exists a unique $y\in S$ such that $(x,y) \in R$. 

\begin{theorem}\label{thm:lovasz_formula}
	Let $\Ss \subset \Graphs$ be a system of representatives for the graph isomorphism relation. 
	Let $G$ and $H$ be two graphs. 
	We have:
	\[ \hom(G,H) = \sum_{A \in \Ss} \frac{\quo(G,A)}{\aut(A)} \mon(A,H).\]
\end{theorem}

\begin{proof}
	First, notice that the sum makes sense, since $\aut(A)\neq 0$ for all $A\in \Ss$, and $\mon(A,H) = 0$ if $A$ is not a subgraph of $H$, so the sum has only finitely many nonzero terms. 

	For every $A\in \Ss$, let $Y_A = \{ f\in \Hom(G,H) \mid A \simeq f(G)\}$. 
	Let $X_A = \Quo(G,A)\times \Mon(A,H)$ and notice that for $(q,m) \in X_A$, we have $mq \in Y_A$. 
	Thus the composition $(q,m)\mapsto mq$ defines a function $\psi_A \colon X_A \to Y_A$. 
	Moreover, for $f\in Y_A$, we have $\psi_A\inv(\{f\}) = \{ (q,m) \in X_A \mid mq = f\}$, which has cardinality $\aut(A)$ by Lemma~\ref{lem:aut_acts_morphisms}. 
	Since $\psi_A$ is surjective by Lemma~\ref{lem:factorisation_quotient_mono}, it comes:
	\begin{align*}
		\quo(G,A).\mon(A,H) &= \# X_A\\
		&= \sum_{f\in Y_A} \# \psi_A\inv(\{f\})\\
		&= \aut(A).\# Y_A.
	\end{align*}
	Hence we have $\# Y_A = \frac{\quo(G,A)}{\aut(A)} \mon(A,H)$.

	Now again by Lemma~\ref{lem:factorisation_quotient_mono} we have that $\Hom(G,H) = \bigsqcup_{A\in S} Y_A$. 
	Thus we obtain
	\[\hom(G,H) = \sum_{A\in \Ss} \# Y_A = \sum_{A\in \Ss} \frac{\quo(G,A)}{\aut(A)}\mon(A,H),\]
	as desired. 
	This concludes the proof.
\end{proof}

It implies the following corollary.

\begin{corollary}\label{coro:lovasz_formula_graph_classes}
	Let $\Ff$ be a class of graphs stable by taking subgraphs. 
	Let $G$ be a graph and let $H\in \Ff$. 
	Then
	\[ \hom(G,H) = \sum_{A\in \Ff} \frac{\quo(G,A)}{\aut(A)} \mon(A,H).\]
\end{corollary}

\begin{proof}
	This is immediate since the summand is null when $A$ is not a subgraph of $H$ by Lemma~\ref{lem:epi_mono}.
\end{proof}

It will be useful to be able to compute $\hom(G,H)$ as a function of the number of morphisms between the connected components of $G$ and $H$.

\begin{lemma}\label{lem:connected_morphisms}
	Let $G = G_1 + \ldots + G_r$ be a graph decomposed into its connected components. 
	Let $H$ be a graph. 
	We have:
	\begin{enumerate}
		\item $\hom(G,H) = \prod_{i=1}^r \hom(G_i,H)$,
		\item if $H$ is connected, we also have $\hom(H,G) = \sum_{i=1}^r \hom(H,G_i)$.
	\end{enumerate}
\end{lemma}

\begin{proof}
	Recall that fixing $H$ we have a natural contravariant functor $\Hom(\cdot ,H) \colon \Graphs \to \Set$ which sends colimits to limits. 
	In particular, since the sum of graphs is a coproduct (and so a colimit), we have a bijection between $\Hom(\sum_i G_i,H)$ and $\prod_i \Hom(G_i,H)$, hence (1).

	Let us prove (2). 
	Let $f\in \Hom(H,G)$. 
	Since $H$ is connected, we have that $f(H)$ is also connected, so there is a unique $i$ such that $f(H)\subset G_i$. 
	Writing this index as $i(f)$, we obtain a function $i\colon \Hom(H,G) \to \{1,\ldots,r\}$ such that $f(H) \subset G_{i(f)}$ for all $f\in \Hom(H,G)$. 
	Now $\varphi \colon f \mapsto f^{|G_{i(f)}}$ is a well-defined mapping from $\Hom(H,G)$ to $\sqcup_{i=1}^r \Hom(H,G_i)$. 
	It is clearly surjective and is injective by what precedes, hence it is a bijection between $\Hom(H,G)$ and $\sqcup_{i=1}^r \Hom(H,G_i)$, which gives the result.
\end{proof}

We obtain the following corollary.

\begin{corollary}\label{coro:connected_morphisms}
	Let $G = \sum_{i=1}^k G_i$ and $H = \sum_{j=1}^l H_j$ be two graphs written as sums of their connected components.
	Then we have:
	\[ \hom(G,H) = \prod_{i=1}^k \sum_{j=1}^l \hom(G_i,H_j).\]
\end{corollary}

In the following, all classes are assumed to be nonempty. 
Given $\Ff$ a class of graphs, we follow Section 7 of~\cite{ManRob} and say that two graphs $G$ and $H$ are \textit{$\Ff$-isomorphic} if for all $A\in \Ff$ we have $\hom(A,G) = \hom(A,H)$. 
Recall that $\Ff^+ = \{ G_1 + \ldots + G_k \mid k\geq 1,\ G_1,\ldots,G_k \in \Ff\}$.

\begin{lemma}\label{lem:connected_F_iso}
	Let $\Ff$ be a class of connected graphs. 
	Then two graphs are $\Ff$-isomorphic if and only if they are $\Ff^+$-isomorphic.
\end{lemma}

\begin{proof}
	It is clear that $\Ff^+$-isomorphism implies $\Ff$-isomorphism, let us show the converse. 
	Let $A \in \Ff^+$, we write $A = A_1 + \ldots + A_r$ as a sum of its connected components, notice that $A_i\in \Ff$ for all $1\leq i\leq r$. 
	Let $G$ and $H$ be two $\Ff$-isomorphic graphs. 
	Using Lemma~\ref{lem:connected_morphisms}, we have:
	\[ \hom(A,G) = \prod_{i=1}^r \hom(A_i,G) = \prod_{i=1}^r \hom(A_i,H) = \hom(A,H).\]
	This shows that $G$ and $H$ are $\Ff^+$-isomorphic, as desired.
\end{proof}

The next theorem, together with Theorem~\ref{thm:connected_F_iso}, form the main motivation of this appendix. 
Though the statement is new, the proof is the same as Lovász' proof (see II, end of paragraph 3 of\cite{LovOpStruct}).

\begin{theorem}\label{thm:F_iso}
	Let $\Ff$ be a class of graphs stable by taking subgraphs. 
	If $G$ and $H\in \Ff$ are $\Ff$-isomorphic, then they are isomorphic.
\end{theorem}

\begin{proof}
	Let $G$ and $H\in \Ff$ be $\Ff$-isomorphic. 
	Let us prove by induction on $v(A) \geq 1$ that $\mon(A,G) = \mon(A,H)$ for all $A\in \Ff$. 
	Since $\Ff$ is stable by taking subgraphs, we have $K_1 \in \Ff$, and $\mon(K_1,G) = \hom(K_1,G) = \hom(K_1,H) = \mon(K_1,H)$, as desired. 
	Now let $A\in \Ff$ and assume that, for all $B\in \Ff$ with $v(B)<v(A)$, we have $\mon(B,G) = \mon(B,H)$. 
	Let $L\in \Ff$. 
	By Corollary~\ref{coro:lovasz_formula_graph_classes}, it comes:
	\begin{align*}
		\hom(A,L) &= \sum_{B\in \Ff} \frac{\quo(A,B)}{\aut(B)} \mon(B,L)\\
		&= \sum_{v(B) = v(A)} \frac{\quo(A,B)}{\aut(B)} \mon(B,L) + \sum_{v(B) < v(A)} \frac{\quo(A,B)}{\aut(B)}\mon(B,L)\\
		&= \mon(A,L) + \sum_{v(B) < v(A)} \frac{\quo(A,B)}{\aut(B)}\mon(B,L)
	\end{align*}
	since if $v(A) = v(B)$ and $\quo(A,B)\neq 0$ then $A=B$ by Lemma~\ref{lem:quotient_iso} and Lemma~\ref{lem:epi_mono}. 
	Applying the precedent computation to $G$ and $H$, and using the fact that $\hom(A,G) = \hom(A,H)$, we obtain:
	\[ \mon(A,G) - \mon(A,H) = \sum_{v(B)<v(A)} \frac{\quo(A,B)}{\aut(B)} \left( \mon(B,H) - \mon(B,G) \right)  = 0\]
	since $\mon(B,G) = \mon(B,H)$ by induction hypothesis. 
	This concludes the proof that $\mon(A,G) = \mon(A,H)$ for all $A\in \Ff$.

	Now, since $G\in \Ff$, we have that $1\leq \mon(G,G) = \mon(G,H)$, and similarly $\mon(H,G)\leq 1$. 
	Thus $G$ is isomorphic to $H$ by Lemma~\ref{lem:mono_mono}. 
	This finishes the proof.
\end{proof}

\begin{theorem}\label{thm:connected_F_iso}
	Let $\Ff$ be a family of connected graphs closed under taking connected subgraphs. 
	If $G$ and $H\in \Ff^+$ are $\Ff$-isomorphic, then they are isomorphic.
\end{theorem}

\begin{proof}
	First, we claim that $\Ff^+$ is closed under taking subgraphs. 
	Indeed, let $G_1,\ldots,G_r\in \Ff$ and let $G = G_1 + \ldots + G_r$. 
	Now let $H\subset G$. 
	For every $1\leq i\leq r$, let $H_i = H[V(G_i)]$, and write $H_i = \sum_{j=1}^{l_i} H_{ij}$ the decomposition of $H_i$ into its connected components, where $l_i\geq 1$ is the number of its connected components. 
	Since $G_i \in \Ff$, we have that $H_{ij} \in \Ff$ for all $1\leq i\leq r$ and all $1\leq j\leq l_i$ by assumption. 
	Notice that, since there are no edges between $G_i$ and $G_j$ for $i\neq j$, we have that for $1\leq i\leq r$ the connected components of $H_i$ are connected components of $H$. 
	So $H = \sum_{i=1}^r \sum_{j=1}^{l_i} H_{ij}$, which shows that $H\in \Ff^+$. 
	Hence $\Ff^+$ is closed under taking subgraphs.

	Now we claim that $H = \sum_{i=1}^r H_i$: indeed, since there are no edges between $G_i$ and $G_j$ for $i\neq j$, and since $H\subset G$, there are no edges in $H$ between vertices in $H_i$ and vertices in $H_j$. 
	Now for all $1\leq i\leq r$ we have that $H_i$ is a subgraph of $G_i$, so by assumption $H_i \in \Ff$. 
	This shows that $H\in \Ff^+$.

	Now let $G$ and $H \in \Ff^+$ be two $\Ff$-isomorphic graphs. 
	By Lemma~\ref{lem:connected_F_iso}, they are $\Ff^+$-isomorphic. 
	Since by what precedes we have that $\Ff^+$ is closed under taking subgraphs, we can apply Theorem~\ref{thm:F_iso}. 
	So $G$ and $H$ are isomorphic, as desired.
\end{proof}

We can now obtain the following.

\begin{theorem}\label{thm:tree_planar_iso}
	Two tree-isomorphic forests are isomorphic, and two planar-isomorphic planar graphs are isomorphic.
\end{theorem}

\begin{proof}
	Trees are connected and closed under taking connected subgraphs, so we can apply Theorem~\ref{thm:connected_F_iso} for the first part, and planar graphs are closed under taking subgraphs, so we can apply Theorem~\ref{thm:F_iso} for the second.
\end{proof}

Using the theorem of Mančinska and Roberson~\cite{ManRob}, we have the following consequence.

\begin{corollary}\label{coro:qi_planar}
	Two quantum isomorphic planar graphs are isomorphic.
\end{corollary}

\begin{proof}
	By Theorem 7.16 of~\cite{ManRob}, quantum isomorphism is equal to planar-isomorphism, so the result follows from Theorem~\ref{thm:tree_planar_iso}.
\end{proof}

\section{Graphs on at most 5 vertices}\label{app:small_graphs}

In this section, we show in Theorem~\ref{thm:tractable_G5} that the class of all graphs on at most 5 vertices is tractable. 
These results are used in Section~\ref{sec:tractable}.

\begin{lemma}\label{lem:tree_cograph_G5}
	The graphs on at most 5 vertices which are not tree-cographs are the pan and its complement, the bull, and $C_5$.
\end{lemma}

\begin{proof}
	By Theorem~\ref{thm:cographs}, a graph which is not a cograph contains $P_4$ as an induced subgraph. 
	This shows that all graphs on at most 4 vertices are cographs except $P_4$, which is a tree. 
	So all graphs on at most 4 vertices are tree-cographs.

	Let $G$ be a graph on 5 vertices which is not a tree-cograph. 
	In particular, since all graphs on at most 4 vertices are tree-cographs by what precedes, $G$ is connected and its complement is connected. 
	Inspecting the list~\cite{graphclasses}, we see immediately that the only graphs on 5 vertices which are connected as well as their complement are the 4 mentioned graphs together with $P_5$ and its complement, and the chair and its complement (see Figure~\ref{fig:small_graphs}). 
	Since $P_5$ and the chair are trees, they are tree-cographs, so $G$ is either the pan or its complement, or the bull, or $C_5$. 
	Moreover, by Theorem~\ref{thm:decomposition_F_cographs}, every connected tree-cograph with a connected complement is either a tree or the complement of a tree, and none of these 4 graphs is one. 
	Hence they are not tree-cographs. 
	This concludes the proof.
\end{proof}

\begin{center}
	\begin{figure}[h]
		\includegraphics[width=12cm]{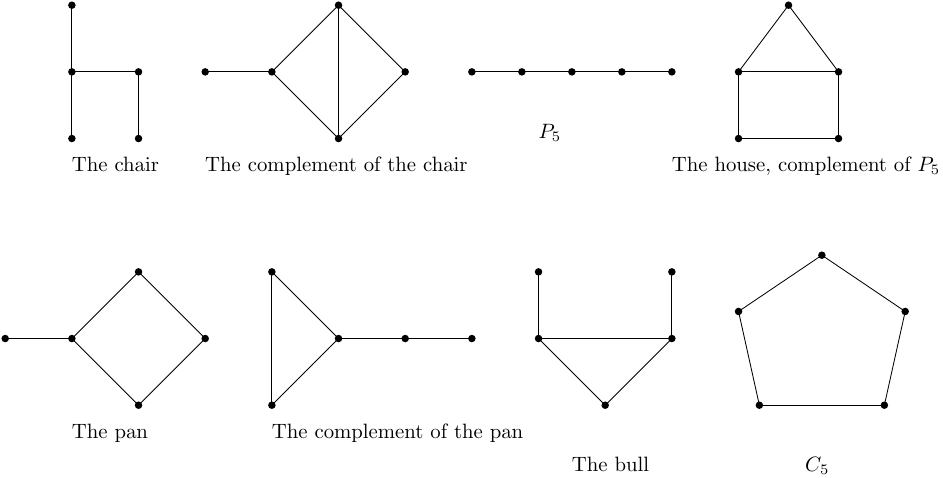}
		\caption{The connected graphs on 5 vertices with a connected complement. The bull and $C_5$ are self-complementary.}\label{fig:small_graphs}
	\end{figure}
\end{center}

\begin{theorem}\label{thm:tractable_G5}
	The class $\Gg_5$ is tractable.
\end{theorem}

\begin{proof}
	We start with (SA). 
	Let $G\in \Gg_5$ and suppose that $G$ does not satisfy Schmidt's criterion. 
	If $G$ is a tree-cograph, then $G$ does not have quantum symmetry by Theorem~\ref{thm:tractable_tree_cographs}. 
	So we can assume that $G$ is not a tree-cograph. 
	Hence it is one of the 4 graphs given by Lemma~\ref{lem:tree_cograph_G5}. 
	But by Theorem~\ref{thm:qi_center} it is easy to see that neither the pan, nor its complement, nor the bull have quantum symmetry. 
	So it remains to show that $C_5$ does not have quantum symmetry. 
	But this was done by Banica in Lemma 3.5 of~\cite{banica2005quantum}. 
	Hence $\Gg_5$ satisfies~(SA).

	Since (SA) implies (QA) by Lemma~\ref{lem:sa_qa}, we have that $\Gg_5$ satisfies (QA) as well.

	Finally, let us show that $\Gg_5$ satisfies (QI). 
	Let $G$ and $H$ be two quantum isomorphic graphs on at most 5 vertices. 
	Since tree-cographs are superrigid by Corollary~\ref{coro:superrigid_tree_cographs}, we can assume that neither $G$ nor $H$ are tree-cographs. 
	Hence by Lemma~\ref{lem:tree_cograph_G5} they are both among the 4 graphs mentioned. 
	By Theorem~\ref{thm:qi_center}, the centers of $G$ and $H$ are quantum isomorphic. 
	Notice that $Z(C_5) = C_5$, $Z(\text{bull}) = K_3$, $Z(\text{pan}) = P_3$, and $Z(\text{pan}^c) = K_2$, which are two-by-two not quantum isomorphic (for instance because they do not have the same number of vertices and edges). 
	This immediately implies that $G=H$, as desired. 
	Hence $\Gg_5$ satisfies~(QI).

	This concludes the proof.
\end{proof}


\begin{thebibliography}{10}

\bibitem{QiNotI}
Albert Atserias, Laura Man{\v{c}}inska, David~E. Roberson, Robert
  {\v{S}}{\'a}mal, Simone Severini, and Antonios Varvitsiotis.
\newblock Quantum and non-signalling graph isomorphisms.
\newblock {\em Journal of Combinatorial Theory, Series B}, 136:289--328, 2019.

\bibitem{banica2005quantum}
Teodor Banica.
\newblock Quantum automorphism groups of small metric spaces.
\newblock {\em Pacific Journal of Mathematics}, 219(1):27--51, 2005.

\bibitem{Bell}
John~S. Bell.
\newblock On the Einstein Podolsky Rosen paradox.
\newblock {\em Physics Physique Fizika}, 1(3):195, 1964.

\bibitem{Bichon2003}
Julien Bichon.
\newblock Quantum automorphism groups of finite graphs.
\newblock {\em Proceedings of the American Mathematical Society},
  131(3):665--673, 2003.

\bibitem{Bichon2004}
Julien Bichon.
\newblock Free wreath product by the quantum permutation group.
\newblock {\em Algebras and representation theory}, 7:343--362, 2004.

\bibitem{BondyMurty}
John~Adrian Bondy and Uppaluri Siva~Ramachandra Murty.
\newblock {\em Graph theory}.
\newblock Springer Publishing Company, Incorporated, 2008.

\bibitem{ConnesBook}
Alain Connes.
\newblock {\em Noncommutative geometry}.
\newblock Springer, 1994.

\bibitem{debruyn2024asymmetric}
Josse van Dobben de Bruyn, David~E. Roberson, and Simon Schmidt.
\newblock Asymmetric graphs with quantum symmetry.
\newblock {\em arXiv preprint arXiv:2311.04889}, 2024.

\bibitem{de2023quantum}
Josse Van~Dobben De~Bruyn, Prem~Nigam Kar, David~E. Roberson, Simon Schmidt, and
  Peter Zeman.
\newblock Quantum automorphism groups of trees.
\newblock {\em arXiv preprint arXiv:2311.04891}, 2023.

\bibitem{FreslonBook}
Amaury Freslon.
\newblock {\em Compact Matrix Quantum Groups and their Combinatorics}, volume
  106.
\newblock CAMBRIDGE University Press, 2023.

\bibitem{Fulton2006}
Melanie~B. Fulton.
\newblock {\em The Quantum Automorphism Group and Undirected Trees}.
\newblock PhD thesis, Virginia Polytechnic Institute and State University,
  2006.

\bibitem{graphclasses}
Graph classes.
\newblock {\em List of small graphs.}
\newblock 24 December 2023.
\newblock {\tt URL: https://www.graphclasses.org/smallgraphs.html\#nodes5}

\bibitem{jordan1869assemblages}
Camille Jordan.
\newblock Sur les assemblages de lignes.
\newblock {\em Journal f\"ur Mathematik}, 70(2):185--190, 1869.

\bibitem{LovOpStruct}
László Lovász.
\newblock Operations with structures.
\newblock {\em Acta Mathematica Academiae Scientiarum Hungaricae}, 18:321--328,
  1967.

\bibitem{LupManRob}
	Martino Lupini, Laura Man\v{c}inska, and David~E. Roberson.
\newblock Nonlocal games and quantum permutation groups.
\newblock {\em Journal of Functional Analysis}, 279(5):108592, 2020.

\bibitem{ManRob}
Laura Mančinska and David~E. Roberson.
\newblock Quantum isomorphism is equivalent to equality of homomorphism counts
  from planar graphs.
\newblock {\em arXiv preprint arXiv:1910.06958}, 2019.

\bibitem{MeIHES}
Paul Meunier.
\newblock {\em Quantum automorphism groups of some classes of graphs.}
\newblock Institut des Hautes \'Etudes Scientifiques, December 2022.
\newblock {\tt URL: https://www.youtube.com/watch?v=lmOnM\_Uw9T0}

\bibitem{pultr1973isomorphism}
Ale\v{s} Pultr.
\newblock Isomorphism types of objects in categories determined by numbers of morphisms.
\newblock {\em Acta Scientiarum Mathematicarum}, Acta Universitatis Szegediensis, 35:155--160, 1973.

\bibitem{RobSchm}
David~E Roberson and Simon Schmidt.
\newblock Solution group representations as quantum symmetries of graphs.
\newblock {\em Journal of the London Mathematical Society}, 106(4):3379--3410,
  2022.

\bibitem{Schm2020}
Simon Schmidt.
\newblock On the quantum symmetry of distance-transitive graphs.
\newblock {\em Advances in Mathematics}, 368:107150, 2020.

\bibitem{SchmCrit}
Simon Schmidt.
\newblock Quantum automorphisms of folded cube graphs.
\newblock {\em Annales de l'Institut Fourier}, 70:949--970, 2020.

\bibitem{Wang}
Shuzhou Wang.
\newblock Quantum symmetry groups of finite spaces.
\newblock {\em Communications in Mathematical Physics}, 195(1):195--211, 1998.

\bibitem{Weichsel1971}
Paul~M Weichsel.
\newblock A note on assymmetric graphs.
\newblock {\em Israel Journal of Mathematics}, 10(2):234--243, 1971.

\end{thebibliography}
\end{document}